\documentclass[12pt,reqno]{amsart}
\usepackage{amsthm}
\usepackage{amssymb}
\usepackage{amsmath}
\usepackage{mathrsfs}
\usepackage{amsfonts}
\usepackage{amssymb,amsmath}
 \usepackage[colorlinks, linkcolor=black, citecolor=blue,   hypertexnames=false]{hyperref}
\usepackage[a4paper,margin=1.05in]{geometry}
\usepackage{amsmath,amssymb,amsthm,mathtools}
\usepackage{mathrsfs}
\usepackage{enumitem}
\usepackage{xcolor}
\usepackage{microtype}
\usepackage{bm}
\usepackage{bm}
\usepackage{enumitem}
\usepackage{cite}

\hypersetup{
    colorlinks=true,
    linkcolor=blue,
    citecolor=blue,
    urlcolor=blue
}

\numberwithin{equation}{section}
\newtheorem{theorem}{Theorem}[section]
\newtheorem{proposition}[theorem]{Proposition}
\newtheorem{lemma}[theorem]{Lemma}
\newtheorem{corollary}[theorem]{Corollary}
\theoremstyle{definition}
\newtheorem{definition}[theorem]{Definition}
\theoremstyle{remark}
\newtheorem{remark}[theorem]{Remark}

\hypersetup{
    colorlinks=true,
    linkcolor=blue!60!black,
    citecolor=blue!60!black,
    urlcolor=blue!60!black
}

\theoremstyle{definition}
\theoremstyle{plain}
\numberwithin{equation}{section}

\newcommand{\R}{\mathbb R}
\newcommand{\Z}{\mathbb Z}
\newcommand{\dd}{\,d}

\newcommand{\supp}{\operatorname{supp}}
\newcommand{\Tr}{\operatorname{Tr}}

\newcommand{\calH}{\mathcal H}

\newcommand{\T}{\mathbb T}
\newcommand{\C}{\mathbb C}
\newcommand{\ii}{\mathrm i}
\newcommand{\ee}{\mathrm e}
\newcommand{\Lp}{L_p}
\newcommand{\Ran}{\operatorname{Ran}}
\newcommand{\spn}{\operatorname{span}}
\newcommand{\Dom}{\mathcal D}
\newcommand{\1}{\mathbf 1}
\newcommand{\eps}{\varepsilon}
\newcommand{\dist}{\operatorname{dist}}

\newcommand{\OO}{\mathcal O}
\newcommand{\YY}{\mathcal Y}

\newcommand{\N}{\mathbb N}
\newcommand{\Ker}{\operatorname{Ker}}

\newcommand{\Id}{I}

\title[Hausdorff type Time-Trace Observability for  Airy type Equations ]{Hausdorff type Time-Trace Observability for Airy Equations on the Line and Point Observability on the Torus}

\author[  Z. Li ]
{Ze Li}

\address{Ze Li\newline\indent School of Mathematics and Statistics, Ningbo University \newline\indent Ningbo, 315211, Zhejiang, P.R. China}
\email{rikudosennin@163.com}

\begin{document}

\maketitle

\begin{abstract}
The main results of this paper are threefold. First, we prove an observability inequality for the Airy equation on the real line from Hausdorff-thick sets in a time-trace sense  for every observation time $T>0$.  The observation functional is a block supremum of $L^2(0,T)$ time traces over the Hausdorff-thick set. Second, we prove observability inequalities for the Airy equation on the real line with observations on some periodic sets, which in particular yields observability on a class of spatial point sequences.   Third, we give a necessary and sufficient condition for finite point observability for the Airy equation on the torus  with a bounded real-valued potential. Indeed, for a finite observation set $F$,  a  Kalman rank condition on a finite-dimensional invariant subspace is found.  As a corollary, we obtain  sharp point observability results for the Airy and linear KdV equations on the torus. Moreover, for potentials with finite order regularity assumptions we prove that every observation set with an accumulation point, in particular any set of positive Hausdorff dimension, gives an observability inequality for each  observation time $T>0$. 

 Since the Airy equation has neither high frequency exponential decay nor pointwise smoothing effects, which are essential in recent works on Hausdorff type observation results on heat equations, we introduce several new ideas adapted to the Airy case.
\end{abstract}

\tableofcontents

\section{Introduction}
In this work, we mainly consider the Airy equation
\begin{align}\label{pb}
\begin{cases}
\partial_t u(t,x)+\partial_x^3 u(t,x)=0,
& (t,x)\in (0,T)\times \Omega,\\[2mm]
u|_{t=0}=u_0\in L^2,
& x\in \Omega,
\end{cases}
\end{align}
where $\Omega$ denotes the real line $\mathbb R$ or the torus $\mathbb T=\mathbb R/2\pi \mathbb Z$. 
We focus on the following two problems. 

{\it Problem 1.} Let $\Omega=\mathbb R$. Let $E\subset\R$ of zero Lebesgue measure satisfy some thick condition. 
We ask which observability inequalities for (\ref{pb}) hold with the observation set $E$, under the requirement that the corresponding observation operator is $L^2$-admissible.

{\it Problem 2.} Let $\Omega=\mathbb T$. Let $F\subset\T$ be a finite set of points. Which observability inequalities for (\ref{pb}) hold  with the observation set $F$, under the requirement that the associated observation operator is $L^2$-admissible.

Before giving full answers to {\it Problem 1} and {\it Problem 2}, we recall some closely related works, which motivate our final answers. 
In general, the observability theory for observations on open sets can be proved by the global Carleman inequalities method of Fursikov-Imanuvilov \cite{FI1996} and the spectral inequality method of Lebeau-Robbiano \cite{LR1995} and the frequency function method of Phung-Wang \cite{PW2013}. 
For Lipschitz and locally star-shaped  bounded domains in $\mathbb R^d$, observability from open subsets was first extended to
measurable observation sets with $d$-dimensional positive Lebesgue measures by Apraiz-Escauriaza \cite{AE2013}  and
Apraiz-Escauriaza-Wang-Zhang \cite{AEWZ2014}  
in the heat equation setting.

The situation on the whole space is different.  On $\mathbb R^d$, a set of positive $d$-dimensional Lebesgue measure  is not sufficient for heat observability. In fact, Egidi-Veseli\'c \cite{EV2018} and Wang-Wang-Zhang-Zhang \cite{WWZZ2019} proved that thickness conditions on the observation set are necessary for  observability. More precisely,  let $u\in C([0,\infty); L^2(\mathbb R^d))$ be the mild solution to 
\begin{align}\label{u89il}
\begin{cases}
u_t-\Delta u=0,\qquad (t,x)\in (0,\infty)\times \mathbb R^d\\
u(0,x)=u_0(x),\qquad x\in\mathbb R^d.
\end{cases}
\end{align}
A set $\omega\subset \mathbb R^d$ is called an observable set of (\ref{u89il}) if for any $T>0$ there exists $C_{T,\omega}>0$ such that 
$$
\|u(T,\cdot)\|^2_{L^2(\mathbb R^d)}\le C_{T,\omega} \int^T_0\int_{\omega}|u(t,x)|^2 dxdt. 
$$
A measurable subset $\omega\subset \mathbb R^d$ is called a thick set if there exist  $\mu>0,L>0$  such that 
\begin{align}\label{hu89r}
|\omega \cap B(x,L)|\ge \mu |B(x,L)|, \qquad \forall x\in \mathbb R^d.    
\end{align}
\noindent\cite{EV2018,WWZZ2019} together proved that given a set $\omega\subset \mathbb R^d$ with positive $d$-dimensional Lebesgue measure, $\omega$
is observable if and only if $\omega$ is thick in the sense of (\ref{hu89r}). The above result was recently extended to Schr\"odinger equations on $\mathbb R^d$ with power potentials by Huang-Wang-Wang  \cite{HWW2022,HWW2025} and to  the Airy equation by Li-Wang \cite{LW2025}. 

More recently, using the new progress on the propagation of smallness in Logunov-Malinnikova \cite{LM2017},  
Burq-Moyano \cite{BM2023} established a new type observability for zero measure sets. 
In fact, given a $d$-dimensional compact manifold $(M,g)$ with a sharp Lipschitz assumption 
and  a set $\omega\subset M$ with zero $d$-dimensional Lebesgue measure, 
\cite{BM2023} proved that for any mild solution to
\begin{align*}
\begin{cases}
\partial_t u-\Delta_{g}u=0,\qquad &(t,x)\in(0,\infty)\times M\\
u(0,x)=u_0(x),\qquad &\forall x\in M,
\end{cases}
\end{align*}
and any $T>0$, there exists $C_{T,\omega}>0$ such that 
$$
\|u(T,\cdot)\|_{L^2(M)}\le C_{T,\omega} \int^T_0\sup_{\omega}|u(t,x)| dt, 
$$
provided that  $\omega$ has a  positive $(d-\delta)$-Hausdorff content for  a $\delta\in (0,1)$ that is a priori close to 0. This result was then generalized to all $\delta\in (0,1)$ when $d\ge 1$  by Green-Le Balc'h-Martin-Orsoni \cite{GBMO2025}  and heat equations driven by a rough divergence elliptic operator when $d=1$  by Zhu \cite{Z2024}, respectively. 

To cover the $\mathbb R^d$ and heat equations with potentials,  as the Hausdorff analogy of thick sets, Le Balc'h-Martin \cite{LM2025} introduced the notion of $(m,d')$-uniformly distributed set. In fact, letting $R>0$, $m>0$ and $0<d'\le d$,  a measurable set $\omega\subset \mathbb R^d$ is said $(m,d')$-uniformly distributed  at scale $R$ if 
\begin{align}\label{4u89}
{\mathcal{H}}^{d'}_{\infty}(\omega\cap B(x,R))\ge m, \mbox{ }\forall  x\in \mathbb R^d.  
\end{align}
Consider the second order parabolic  equation with bounded potentials on $\mathbb R^d$, i.e. 
\begin{align}\label{Poi8}
\begin{cases}
u_t-\frac{1}{\kappa(x)}{\rm div}(g^{-1}(x)\kappa(x)\nabla u) +V(x)u=0,&\qquad (t,x)\in (0,\infty)\times \mathbb R^d\\
u(0,x)=u_0(x),&\qquad x\in\mathbb R^d,
\end{cases}
\end{align}
where $V\in L^{\infty}(\mathbb R^d; \mathbb R)$,  $\kappa(x)$ is a positive bounded Lipschitz function, and $g(x)$ is a symmetric uniformly elliptic matrix with Lipschitz entries. 

Let $R>0$ satisfy $\mathbb R^d=\cup_{k\in\mathbb Z^d} B(k,R)$. Theorem 2.13 of  \cite{LM2025} proved that there exists a $\delta_d\in (0,1)$ such that for any  $(m,d-\delta)$-uniformly distributed set $\omega\subset \mathbb R^d$ with $\delta\in (0,\delta_d)$ at scale $R$, any mild solution to (\ref{Poi8}) and any time $T>0$, there holds 
\begin{align}\label{i9omm}
\|u(T,\cdot)\|^2_{L^2(\mathbb R^d)} \le Ce^{\frac{C}{T}}\sum_{k\in \mathbb Z^d}\int^T_0\sup_{x\in\omega \cap B(k,R)}|u|^2(t,x)dt. 
\end{align}

A further development was obtained by Huang-Wang-Wang \cite{HWW2026}, who introduced log-type Hausdorff
contents adapted to the heat kernel and proved observability inequalities
for heat equations both on bounded domains and on 
$\mathbb R^d$ from sets measured by the log-type Hausdorff contents.  Their main result for $\mathbb R^d$ is as follows: For a class of gauge functions $f$ and sets $\omega$ satisfying 
\begin{align*}
\inf_{x\in \mathbb R^d} c_{f}(\omega \cap Q_{K}(x))\ge \mu |Q_{K}(0)|,
\end{align*}
for some $K>0$, $\mu>0$, where $c_{f}$ denotes the Hausdorff content defined by $f$ and $Q_{K}(x)$ denotes the closed cube in $\mathbb R^d$ centered at $x$ with side length $K$, and for any $T>0$, any mild  solution $u$ to the heat equation on $\mathbb R^d$ with initial data $u_0\in L^2$, there exists $C=C_{K,\mu,\omega,f,T}>0 $ such that 
\begin{align}\label{11omm}
\|u(T,\cdot)\|_{L^2(\mathbb R^d)} \le C \int^T_0\sum_{k\in \mathbb Z^d} \left(\sup_{x\in \omega\cap Q(k)}|u|^2(t,x)\right)^{\frac{1}{2}}dt,
\end{align}
where $\{Q(k)\}$ is a sequence of closed unit cubes in $\mathbb R^d$ centered at $k\in \mathbb Z^d$ such that $\mathbb{R}^d=\cup_{k\in\mathbb Z^d} Q(k)$.

{\it Problem 1} aims to  establish an observability inequality for (\ref{pb}) for the set $E\subset \mathbb R$ of a zero Lebesgue measure. Unlike the above settings of thick sets, we introduce the notion of a Hausdorff-thick set. 

\begin{definition}\label{ht}
Let $E\subset \mathbb R$. We call $E$ Hausdorff-thick, if for some $0<s\le1$, $R>0$ and $m>0$,
\begin{equation}\label{thicko}
    \mathcal H^s_\infty\bigl(E\cap[3jR,3jR+R]\bigr)\ge m,
    \qquad \forall  j\in\Z .
\end{equation}
\end{definition}
 
It is easy to see that sets $\omega\subset \mathbb R$  satisfying (\ref{4u89}) also fulfill  (\ref{thicko}).  

We will establish observability inequalities for (\ref{pb}) with an observation set $E$ satisfying 
(\ref{thicko}). Note that neither (\ref{i9omm}) nor (\ref{11omm}) for the heat equation  makes sense for the Airy equation (\ref{pb}). The somewhat superficial reason is that the right hand sides of (\ref{i9omm}) and (\ref{11omm}) generally are infinite for $u$ solving the Airy equation with $L^2$ data. 

\subsection{Main results on Airy on the real line}
We begin with the answers of {\it Problem 1}.

Let $S(t)$
be the Airy group on $L^2(\R)$, i.e. $u(t)=S(t)u_0$ solves
\begin{align}\label{2pb}
\begin{cases}
\partial_t u(t,x)+\partial_x^3 u(t,x)=0,
& (t,x)\in (0,\infty)\times \mathbb R,\\[2mm]
u(0,x)=u_0(x),
& x\in \mathbb R. 
\end{cases}
\end{align}
For general $u_0\in L^2(\R)$, the value $u(t,x)$ at a fixed time is not a classical point value.  Nevertheless, the Airy equation has a local smoothing: for each fixed $x$, the map
\begin{align*}
u_0\mapsto \Tr_x S(\cdot)u_0
\end{align*}
extends continuously from Schwartz data to a bounded map from $L^2(\R)$ to $L^2_{\rm loc}(\R_t)$, see  Lemma \ref{sing}.  It is therefore natural to ask whether a thick spatial observation set can determine the full $L^2$ norm of the initial data through time traces.

Let $E$ satisfy (\ref{thicko}). 
Set
\begin{align}\label{Ej}
    E_j=E\cap[3Rj,3Rj+R].
\end{align}
We define the observation functional by 
\begin{equation}\label{OE1}
    \OO_{E,T}(u_0)=
    \sum_{j\in\Z}\sup_{x\in E_j}
       \bigl\|\Tr_x S(\cdot)u_0\bigr\|_{L^2(0,T)}^2.
\end{equation}
The main result for {\it Problem 1}  is as follows.

\begin{theorem}\label{th1}
Let $E\subset \mathbb R$ satisfy (\ref{thicko}) for some $R,m>0$, and define $E_j$ by (\ref{Ej}).  For every $T>0$, define $\OO_{E,T}(u_0)$ by (\ref{OE1}). Then there exist $C:=C_{T,E,R,m}>0$ and  $C':=C'_{T,E,R,m}>0$ such that 
\begin{equation}\label{main}
    \|u_0\|_{L^2(\R)}^2\le C\OO_{E,T}(u_0),
    \qquad \forall \mbox{ }u_0\in L^2(\R).
\end{equation}
and 
\begin{equation}\label{admiss}
    \OO_{E,T}(u_0)\le C' \|u_0\|_{L^2(\R)}^2, \qquad  \forall \mbox{ }u_0\in L^2(\mathbb R).
\end{equation}
For smooth data, \eqref{main} says
\begin{align*}
    \|u_0\|_2^2\le C 
    \sum_{j\in\Z}\sup_{x\in E_j}
    \int_0^T |S(t)u_0(x)|^2\dd t .
\end{align*}
\end{theorem}

Theorem \ref{th1} applies to every Hausdorff-thick  set in the sense of Definition \ref{ht}. The following result shows that for sets with special structures, the thick condition is not necessary.  
In fact, one can also obtain observability inequalities for any periodic Borel measure, see Proposition \ref{p1} below.

Let $h>0$, let ${\mathcal C}\subset[0,h)$ be a Borel set, and let $\nu$ be a finite
positive Borel measure supported on ${\mathcal C}$, with
\begin{align*}
    0<\nu(\mathcal C)<\infty.
\end{align*}
Set
\begin{align}\label{gy99v}
    E={\mathcal C}+h\mathbb Z,
    \qquad
    \mu_E=\sum_{n\in\mathbb Z}(\tau_{nh})_\#\nu,
\end{align}
where $\tau_{nh}(y)=y+nh$. 
\begin{proposition} \label{p1}
 For every $T>0$, there exist constants
$0<c_{T,h,\nu}\le C_{T,h,\nu}<\infty$ such that
\begin{align}\label{Er8}
    c_{T,h,\nu}\|u_0\|_{L^2(\mathbb R)}^2
    \le
    \int_0^T\int_E |S(t)u_0(x)|^2\,d\mu_E(x)\,dt
    \le
    C_{T,h,\nu}\|u_0\|_{L^2(\mathbb R)}^2.
\end{align}
For general $u_0\in L^2(\mathbb R)$, the observation functional in the middle of (\ref{Er8}) is
understood as the $L^2$-extension from Schwartz data.
\end{proposition}

\begin{remark}
Some remarks related to Theorem \ref{th1} and Proposition \ref{p1} are as follows.  
\begin{enumerate}[label=\textup{(\roman*)}]
 \item 
In particular, taking $E=b+h\mathbb Z$, $\mathcal{C}=\{b\}$, ${\nu}=\delta_{b}$ in (\ref{gy99v}), we find that a point sequence observation, which is not thick in any aforementioned sense, is available for the observability inequality.  

\item Note that Li-Wang \cite{LW2025} has shown that  for classical type observability, i.e.,
$$\|u_0\|^2_{L^2}\le C_T\int^T_{0}\int_{E}|u(t,x)|^2dxdt,\qquad \forall \mbox{ }u_0\in L^2,$$
the thick condition, i.e., there exists $l>0$ such that 
$$
\inf_{x\in\mathbb R} |E\cap [x,x+l]|>0,
$$
is a necessary condition. 

\item From the  two facts (i) and (ii), one sees that  there is an essential difference between the classical $L^{2}_{t,x}([0,T]\times E)$  observation and the $\sup_xL^2_t$ type observation in our setting.

\item By tracking all the constants in the proof of Theorem \ref{th1}, one can explicitly write down the dependence of constant $C$ on $T$ in (\ref{main}). 
For example, $C$ in (\ref{main}) can be chosen as 
\begin{align}\label{RotKL}
B\exp\left(BT^{-1}\right),
\end{align}
when $0<T\ll 1$, where $B>0$ depends only on $E,m,R$.  In fact,  taking $\sigma \in (1,2)$, $N=\frac{C}{T}$  in the proof of Proposition \ref{proline}, the high frequency  estimates in Corollary \ref {coho} and the separation estimates  in Lemma \ref{lesepl} and Lemma \ref{lehfs} hold, provided that $T>0$ is very close to 0. Then the low frequency part gives the final observation cost constant $\frac{1}{T}{\exp} (C(1+N){})$, which becomes (\ref{RotKL}) by enlarging $B$.

\item  Wang and Wang
\cite{WW2024} proved that the
Airy  equation can be observed from two segments in spacetime with restrictions on the two slopes.

\item More recently, Li and Wang \cite{LW2022,LW2025} proved
observability inequalities at two time points from half-lines for linear KdV
equations.

\end{enumerate}
\end{remark}

\begin{remark}
Hausdorff measure entered the analysis of partial differential equations first through the study
of nodal and critical sets of elliptic and parabolic equations, see  Hardt-Simon \cite{HS1989}, Lin
\cite{Lin1991}  and recent works  \cite{Ma2009,Logunov20181,Logunov20182,LMNN2021}. 
Recently, Logunov-Malinnikova
\cite{LM2017} proved the quantitative propagation of smallness.  In fact, they proved Remez-type inequalities for
solutions of elliptic equations with Lipschitz coefficients, see also \cite{Z2024,GBMO2025,BM2025,HWW2026}. 
These works motivate the control problems on Hausdorff observation sets. 
\end{remark}

\begin{remark}
The spectral inequality in the setting of Hausdorff measures plays an important role in recent pioneering works \cite{Z2024,GBMO2025,BM2025,HWW2026,MZ2025} on heat equations with observation sets of positive Hausdorff measures. This type of  spectral inequality may be seen as the fractional analogy of the Logvinenko-Sereda uncertainty principle.  Note also that there is another closely  related uncertainty principle, known as the fractional uncertainty principle (FUP), which states that no function can be localized in both spatial and frequency space near a fractional set. 
The FUP has been intensively studied and applied to various control problems, see e.g. \cite{D2019,DJ2018,C2025,DJN2022}. 
\end{remark}


It will be interesting to compare the results on the Airy equation with the linear KdV equation. 
In fact, we have the following results. 
\begin{proposition}\label{p2}
Consider the linear KdV equation
\begin{align}\label{kdv}
    w_t+w_x+w_{xxx}=0, \qquad w(0,x)=w_0(x), \mbox{ }x\in \mathbb R.
\end{align}
\begin{enumerate}
    \item Let $E\subset \mathbb R$ satisfy (\ref{thicko}) for some $R,m>0$, and let $E_j$ be defined in  (\ref{Ej}). 
Then there exist constants $C:=C_{T,E,m,R}$ and $C':=C'_{T,E,m,R}$ such that 
\begin{equation}\label{RUK}
   C \|w_0\|_2^2\le \sum_j\sup_{x\in E_j}
       \|\Tr_{x}w(t,t+x)\|_{L^2(0,T)}^2\le C' \|w_0\|_2^2
\end{equation}
 \item 
Given $T>0$, there is no $C_T>0$ such that
\begin{align*}
    \|w_0\|_2^2\le C_T\int_0^T\sum_{k\in\Z}|w(t,2\pi k)|^2\dd t
\end{align*}
for all smooth compactly supported $w_0$.
\end{enumerate}
\end{proposition}

\begin{remark}
The positive result (\ref{RUK}) follows simply by the transformation: $u(t,x):=w(t,x+t)$
solves the Airy equation if $w$ solves (\ref{kdv}). 
Moreover, Proposition \ref{p1} can also be extended to (\ref{kdv}) in the same way. However, the negative result in Proposition \ref{p2} implies that, generally,  temporally fixed trace observation fails for the linear KdV due to the transport term $w_x$. 
\end{remark}

\subsection{Main results on the torus }

For {\it Problem 2}, let us consider the equation 
\begin{align}\label{pb3}
\begin{cases}
\partial_t v(t,x)+\partial_x^3 v(t,x)+p(x)v(t,x)=0,
& (t,x)\in (0,T)\times \mathbb T,\\[2mm]
v(0,x)=v_0(x),
& x\in \mathbb T,
\end{cases}
\end{align}
where $v_0\in L^2$ and $p\in L^{\infty}(\mathbb T; \mathbb R)$. 

Let
\begin{align*}
  \Lp=\partial_x^3+p(x),\qquad \Dom(\Lp)=H^3(\T).
\end{align*}
The operator $-\Lp$ generates a $C_0$-group on $L^2(\T)$, denoted
by
\begin{align*}
  U_p(t)=\ee^{-t\Lp},\qquad t\in\R.
\end{align*}

\begin{definition}
Let $F\subset\T$ be finite and non-empty.  We say that $F$ is observable for
$\Lp$ in time $T>0$ if there exists $C_{T,F,p}>0$ such that
\begin{equation}\label{2obs-def}
  \|v_0\|_{L^2(\T)}^2
  \le C_{T,F,p}\int_0^T\sum_{a\in F}|\Tr_a U_p(t)v_0|^2\dd t
\end{equation}
for all $v_0\in L^2(\T)$.  
\end{definition}
The trace $\Tr_a U_p(\cdot)v_0$ is defined in
Definition \ref{def:trace} below by spectral expansion.  For $v_0\in H^3(\T)$ it is the
classical point value $(U_p(t)v_0)(a)$.
 
\begin{theorem}\label{th3}
Let $F\subset\T$ be finite and non-empty. Given $p\in L^{\infty}(\mathbb T; \mathbb R)$, there exist $d\in \mathbb N$ and a  $d$-dimensional space $X_0$ which is invariant under $\Lp$ such that for every $T>0$ and
\begin{align*}
  \mathcal{A}:=\Lp|_{X_0},
  \qquad C_F y:=(y(a))_{a\in F},
\end{align*}
the following statements are
equivalent:
\begin{enumerate}[label=\textup{(\roman*)}]
  \item $F$ is observable in the sense of \eqref{2obs-def}.
  \item If $y\in X_0$ and
  \begin{align*}
    C_F\mathcal{A}^q y=0,
    \qquad q=0,1,\dots,d-1,
  \end{align*}
  then $y=0$.
  \item The matrix
  \begin{align*}
    \begin{pmatrix}
    C_F\\ C_F\mathcal{A}\\ \vdots\\ C_F\mathcal{A}^{d-1}
    \end{pmatrix}:X_0\to(\C^F)^d
  \end{align*}
  has rank $d$.
\end{enumerate}
 \end{theorem}

Note that these conditions do not depend on the particular $T>0$.   

In the following two particular cases, we have a more explicit  observability result.   

\begin{proposition}   \label{c1}
\begin{enumerate} [label=\textup{(\roman*)}]
\item Consider the periodic Airy equation
\begin{align*}
    u_t+u_{xxx}=0, \mbox{ }(t,x)\in\mathbb R\times \mathbb T,
    \qquad
    u(0,x)=u_0(x),\mbox{ }x\in\mathbb T.
\end{align*}
For every $T>0$ and every $x_0\in\T$, there exist  $C_T>0$ and $C'_T>0$ such that
\begin{align*}
 C_T    \|u_0\|_{L^2(\T)}^2
    \le \|\Tr_{x_0}u\|_{L^2(0,T)}^2 \le C'_{T}\|u_0\|^2_{L^2(\mathbb T)}.
\end{align*}
\item
Consider the periodic linear KdV equation,
\begin{align*}
    u_t+u_x+u_{xxx}=0, 
    \mbox{ } (t,x)\in \mathbb R\times \mathbb T,
    \qquad
     u(0,x)=u_0(x),\mbox{ }x\in\mathbb T.
\end{align*}
Let $F\subset\T$ contain at least three distinct points.  Then for every $T>0$, there exist $C_{T,F},C'_{T,F}>0$ so that 
\begin{align}
  C_{T,F}  \|u_0\|_{L^2(\T)}^2
    &\le \sup_{x\in F}\|\Tr_xu\|_{L^2(0,T)}^2 \label{OIujL}\\
    \sup_{x\in F}\|\Tr_xu\|_{L^2(0,T)}^2 & \le C'_{T,F}\|u_0\|^2_{L^2(\mathbb T)}.\nonumber
\end{align}
 \item 
Consider the periodic linear KdV equation, if $F\subset \mathbb T$ contains 
 two or fewer points, then for any $T>0$ and any constant $C_{T,F}>0$, there exists $u_0$
so that (\ref{OIujL}) fails. 
\end{enumerate}
\end{proposition}

\begin{remark} 
Some remarks on Theorem \ref{th3} and Proposition \ref{c1} are listed as follows. 
\begin{enumerate}[label=\textup{(\roman*)}]
\item  \cite{CKP2025} proved pointwise controllability  of the linearized Gear-Grimshaw system under some assumptions on the coefficients in the system. This result is  close to (i) of Proposition   \ref{c1}. 
\item One point observation type results also hold for the free heat equation. 
For the free heat equation on $x\in (0,1)$ with Dirichlet boundary condition,  Dolecki \cite{D1973} and Samb \cite{S2015} proved that, given any $x_0\in (0,1)$ being an algebraic number of order strictly larger than 1 and any $T>0$ and any  weak solution $u(t,x)$ with initial data $u_0\in L^2(0,1)$, there exists some $C>0$ such that 
\begin{align*}
\|u(T,\cdot)\|^2_{L^2(0,1)}\le Ce^{\frac{C}{T}}\int^T_0 |u(t,x_0)|^2 dt.
\end{align*} 
 
 \item 
Note also that [Prop. 3.2, \cite{GBMO2025}] proved that in the spatial domain $\Omega=(0,1)^{n}$ with $n\ge 2$, for any $T>0$ and any finite subset $\omega$ of $\Omega$, the heat equation on $\Omega$ is not observable from $\omega$ at time $T$. 

\item In the Lebesgue measure setting, Niu-Wang-Xiang \cite{NWX2025}  recently established the local exact controllability of the KdV equation on the torus  around equilibrium states, where both the spatial control region and the temporal control region are sets of positive measure.

 \item One may also compare the Airy case with the Schr\"odinger equation. 
On compact manifolds and tori, observability from open sets is
closely connected with semiclassical measures and propagation of mass, see Jaffard \cite{Je} and 
Anantharaman--Maci\`a \cite{AM2014}. The Schr\"odinger equation on tori with rough potentials was treated in
dimension two by Bourgain-Burq-Zworski \cite{BBZ2013}. More recently, Haak-Jaming-Wang-Wang \cite{HJWW2026} proved that solutions of the toroidal Schr\"odinger equation can be observed from suitably curved space-time trajectories, thus of zero Lebesgue measure.
\end{enumerate}
\end{remark} 

\begin{remark}\label{Rty0}
The results of Proposition \ref{c1} can be heuristically explained  from the frequency modes of the Airy equations and the linear KdV equations. In fact, for the free periodic Airy
equation, since the frequencies $n^3$ are monotone in $n$ and have gaps tending to infinity,  the Haraux--Ingham
inequality implies that one spatial point observes all Fourier modes in any positive
time.  For the periodic linear KdV equation,  the frequencies $n^3-n$ have a
three-dimensional stationary eigenspace corresponding to $n=-1,0,1$.  This finite
stationary obstruction explains why one or two points are not enough and why three
distinct points suffice.
\end{remark}

If the observation set $F$ has an accumulation point, we have the following result. 
\begin{corollary}\label{c2}
Let $p\in L^{\infty}(\mathbb T;\mathbb R)$ and  $E\subset\T$ have an accumulation point. There exists $k_0\in \mathbb N$ depending only on $\|p\|_{L^{\infty}}$ such that if 
 $p\in C^{k_0}(\T;\R)$, then for every $T>0$ there exists $C_{T,E,p}>0$ such that
\begin{align}\label{hyjj}
  \|v_0\|_{L^2(\T)}^2
  \le
  C_{T,E,p}\sup_{a\in E}\|\Tr_aU_p(\cdot)v_0\|_{L^2(0,T)}^2
\end{align}
for all $v_0\in L^2(\T)$.  Moreover, there exists $C'_{T,E,p}>0$ such that 
\begin{align*}
  \sup_{a\in E}\|\Tr_aU_p(\cdot)v_0\|_{L^2(0,T)}
  \le C'_{T,E,p}\|v_0\|_{L^2(\T)} 
\end{align*}
\end{corollary}

\begin{remark}
Since every set of positive Hausdorff dimension has an accumulation point, Corollary \ref{c2} holds for  $E\subset\T$ of positive Hausdorff dimension.  
\end{remark}

\begin{remark}
The integer $k_0$ in Corollary \ref{c2} can be explicitly determined, see Proposition \ref{usdorff}. 
\end{remark}

\subsection{Main ideas} 

In this subsection, we explain the main ideas in the proof of Theorem \ref{th1}
and Theorem \ref{th3}. 

Let us first recall the mechanism behind the known results on parabolic equations. The Hausdorff-type observation results
for heat equations were obtained by combining quantitative propagation of smallness
of Logunov--Malinnikova type with Lebeau--Robbiano iteration.  In particular, the
works on Hausdorff or logarithmic Hausdorff contents reduce the problem to a
low frequency spectral inequality. Then the high frequency part is eliminated by the
parabolic decay factor $e^{-t|\xi|^2}$ or its analogue.

This strategy is not available for the Airy equation.  There are two essential
obstructions.  First, since the Airy flow is unitary on $L^2$,  the high frequency part does
not decay. Hence, any type of Hausdorff spectral inequality for low frequencies alone cannot imply
observability.  Second, for general $L^2$ initial data the value $u(t,x)$ at a fixed time is
not an intrinsic spatially pointwise quantity.  Thus the pointwise supremum  expressions used in
parabolic Hausdorff observation, such as
\begin{align*}
        \int_0^T \sup_{x\in E} |u(t,x)|^2\,dt,
\end{align*}
do not have a direct meaning for Airy equations with $L^2$ data.  The correct
replacement is a time-trace quantity:
\begin{align*}
        f \longmapsto \Tr_x S(\cdot)f \in L^2_{\mathrm{loc}}(\mathbb R_t),
\end{align*}
which is well-defined by Airy local smoothing, see Lemma  \ref{sing}. 

The first new point of this paper is the formulation of an admissible
time-trace observation norm.  Let  $E$ be Hausdorff-thick  in the sense of 
Definition \ref{ht} and   $E_j$
be defined by (\ref{Ej}). 
We observe the solution through the block supremum trace norm
\begin{align*}
        \|S(\cdot)f\|_{\YY_E(0,T)}^2
        :=
        \sum_{j\in\mathbb Z}
        \sup_{x\in E_j}
        \|\Tr_x S(\cdot)f\|_{L^2(0,T)}^2 .
\end{align*}
This norm is neither a simple variant of the usual $L^2_{t,x}$ observation nor the heat Hausdorff observation (see (\ref{i9omm}) and (\ref{11omm})).    It is intrinsic for $L^2$ initial data and exploits the dispersive smoothing effects.  

The second point is the low-frequency Hausdorff propagation estimate in a
Hilbert-valued form. In fact, we prove
\begin{align*}
        \int_{\mathbb R}
        \|F(x)\|_H^2\,dx
        \leq
        C\exp\{C(1+N){}\}
        \sum_{j\in\mathbb Z}
        \sup_{x\in E_j}\|F(x)\|_H^2 ,
\end{align*}
where $H$ is an arbitrary Hilbert space, and the Fourier transform of $F$ in the $x$ variable  is supported in $[-N,N]$.  This is
derived from 1D Remez-type estimates for holomorphic functions, see Lemma \ref{hkey}, whose key part is the $e^{C_*B}$ cost in (\ref{RtkkL}). The Hilbert form enables us to treat the block supremum
trace norm $\YY_{E}(0,T)$. 

The third new point is the treatment of high frequencies without parabolic decay.
Instead of canceling the high frequency part by heat decay, we use the regular sampling sequence
observation and deal with high frequencies directly.  For a
regular sampling sequence $\Lambda=\{x_j\}_{j\in\mathbb Z}$, see Definition \ref{regular} below, the sampling operator
\begin{align*}
        C_\Lambda f=(f(x_j))_{j\in\mathbb Z}
\end{align*}
satisfies the resolvent estimates on the high frequency
region,   
\begin{align*}
        \|g\|_{L^2}^2
        \leq
        C\|C_\Lambda g\|_{\ell^2}^2
        +
        \frac{C}{N^2}\|(i\partial^3_x+\lambda)g\|_{L^2}^2,
        \qquad \lambda\in\mathbb R, \mbox{ } \supp \widehat{g}\subset \{\xi: |\xi|\ge N\}.
\end{align*}
Miller's resolvent criterion for unitary groups then gives exact observability of the
high frequency part on every prescribed time interval of length $T>0$, once $N$ is chosen
sufficiently large compared with $T^{-1}$.  This is the dispersive substitute for the
high frequency damping in the Lebeau--Robbiano method.

The fourth new idea is the use of Gevrey time-frequency filters to
separate low and high spatial frequencies from the observed full trace.  For the solution to free
Airy equations, its spacetime Fourier transform is supported in the curve $\tau=\xi^3$.  Hence
spatial cutoffs in $\xi$ can be realized as time multipliers on the traces.  The difficulty
is that the observation is only known on $(0,T)$, whereas a Fourier multiplier in time is
nonlocal.  We resolve this by applying compactly supported Gevrey cutoffs on an
interior interval $I_T\Subset(0,T)$.  Their kernels satisfy subexponential decay
\begin{align*}
        |K(t)|\leq C e^{-c|t|^{1/\sigma}},
        \qquad 1<\sigma<3 .
\end{align*}
Consequently the unknown exterior time tails contribute an error of size
$e^{-cN^{3/\sigma}}\|f\|_{L^2}$, which can be absorbed because
\begin{align*}
       N^{3/\sigma} \gg N  
        \qquad (1<\sigma<3),
\end{align*}
where $e^{C N}$ is the cost of the low frequency Hausdorff propagation estimate.
This comparison is the point at which the cubic Airy dispersion and the choice of
Gevrey order enter in an essential way. Assembling the low and high frequency parts together yields Theorem \ref{th1}.

For a bounded potential on the torus, the operator
\begin{align*}
        L_p=\partial_x^3+p(x), \qquad D(L_p)=H^3(\mathbb T),
\end{align*}
is generally non-normal, and one cannot assume a global orthonormal or simple
eigenfunction expansion.  The main idea is to push the non-normal difficulties into a
finite-dimensional subspace.  The cubic gaps explained in Remark \ref{Rty0} dominate the bounded potential at high
frequency.  As a result, the high frequency Riesz spectral projections of $L_p$ are
one dimensional, close to the Fourier projections in the square summation sense, and thus provide  a Riesz
decomposition of high frequency spaces.  All possible spectral troubles, such as multiple eigenvalues,
Jordan chains and low frequency interactions, are confined to a finite-dimensional 
invariant space $X_0$ under $L_p$.

Point observability for the potential Airy equation on the torus is therefore reduced to
a finite-dimensional visibility problem. In fact, we can prove that
the finite point set $F\subset \mathbb T$ is observable if and only if the finite-dimensional system
\begin{align*}
        \dot y + \mathcal{A} y=0,\qquad C_Fy=(y(a))_{a\in F},
        \qquad \mathcal{A}=L_p|_{X_0},
\end{align*}
satisfies the Kalman criterion 
\begin{align*}
        C_F \mathcal{A}^q y=0,\quad q=0,\ldots,\dim X_0-1
        \quad\Longrightarrow\quad y=0 .
\end{align*}
This gives a necessary and sufficient finite rank condition for finite point observability for each
fixed potential, i.e. Theorem \ref{th3}. 

Moreover, observation sets with an accumulation point automatically
separate the finite-dimensional root spaces under the stated finite regularity
assumption. Hence the  time-trace observability in Corollary \ref{c2} holds.  The result is sharp in the sense
that no finite set can work uniformly for all smooth real potentials. In fact, for any prescribed
finite set one can construct a smooth stationary solution vanishing on it, see Appendix A.

\subsection{Proof strategy}

We now describe the proof strategy of  Theorem \ref{th1} and  Theorem \ref{th3}. 

\subsubsection{The free Airy equation on the line}

Let $S(t)=e^{-t\partial_x^3}$ be the Airy group on $L^2(\mathbb R)$, and let  $D=-i\partial_x$. Define  $\widetilde P_N$ and $\widetilde Q_N$ to be the low frequency cutoff and the high frequency cutoff respectively by  
\begin{align*}
        \widetilde P_N=\chi(D^3/N^3),
        \qquad
        \widetilde Q_N=I-\widetilde P_N ,
\end{align*}
where $\chi\in C^{\infty}_c(\mathbb R)$ is  a compactly supported  cutoff function in the Gevrey class, see Appendix B. 

\paragraph{Step 1: Trace admissibility and the block observation space.}
For each fixed $x\in\mathbb R$, the map
\begin{align*}
        f\mapsto \Tr_x S(\cdot)f
\end{align*}
extends continuously from Schwartz data to a bounded map
\begin{align*}
        L^2(\mathbb R)\longrightarrow L^2_{\mathrm{loc}}(\mathbb R_t).
\end{align*}
The proof is based on a $TT^*$ argument, see Lemma \ref{sing}. 

\noindent For a regular sequence $\Lambda=\{x_j\}_{j\in\mathbb Z}$ in the sense of Definition \ref{regular}, the trace map is upgraded to
\begin{align*}
        f\mapsto \{\Tr_{x_j}S(\cdot)f\}_{j\in\mathbb Z}
        \in L^2_{\mathrm{loc}}(\mathbb R_t;\ell^2).
\end{align*}
For a temporal interval $J\subset\R$ define the block norm by 
\begin{align*}
   \|g\|_{{\mathcal Y}_E(J)}^2
   & :=\sum_{j\in\Z}\sup_{x\in E_j}
       \bigl\|g\|_{L^2(J)}^2,
    \end{align*} 
where $E_j$ is defined by (\ref{Ej}). 
Viewing $S(\cdot)f$ as $\Tr_x S(\cdot)f$, we define  
\begin{align}   
\|S(\cdot)f\|_{\YY_E(J)}^2:= \sum_{j\in\Z}\sup_{x\in E_j}
       \bigl\|\Tr_xS(\cdot)f\bigr\|_{L^2(J)}^2.\label{ye}  
\end{align}
Since every selection $x_j\in E_j$ is a regular sequence with constants depending only
on $R$, taking the supremum over all such selections gives the admissibility of the block trace norm:
\begin{align*}
        \|S(\cdot)f\|_{{\mathcal Y}_E(J)}^2
        \leq C_R(1+|J|)\|f\|_{L^2}^2
\end{align*}
for every bounded interval $J$, see Lemma \ref{balg}.

\paragraph{Step 2: High frequency observability by a unitary Hautus estimate.}
For regular sampling sequences, we have a frequency-localized resolvent estimate for
$D^3$:
\begin{align*}
        \|g\|_{L^2}^2
        \leq
        C\|C_\Lambda g\|_{\ell^2}^2
        +
        \frac{C}{N^2}
        \|(D^3-\lambda)g\|_{L^2}^2,
        \qquad
        \supp\widehat g\subset\{|\xi|\geq N\}.
\end{align*}
The key ingredient is a narrow frequency band sampling inequality: on a short frequency interval,
regular point samples determine the $L^2$ norm, while away from that interval the
resolvent $(D^3-\lambda)^{-1}$ gains enough powers of frequency. Then by Miller's criterion, for every $T>0$, there exists a threshold $N_1(T)$ such that for $N\geq N_1(T)$,
\begin{align}\label{Huoo7}
        \|\widetilde Q_N f\|_{L^2}^2
        \leq
        C_T\|S(\cdot)\widetilde Q_N f\|_{{\mathcal Y}_E(I_T)}^2 ,
        \qquad I_T=[T/3,2T/3],
\end{align}
see Lemma \ref{lem:HF-Hautus} and Corollary \ref{coho}.

\paragraph{Step 3: Extracting the high frequency observation from the full observation.}
The preceding estimate (\ref{Huoo7}) observes $S(t)\widetilde Q_N f$, but  Theorem \ref{th1} only assumes the
full trace $S(t)f$ on $(0,T)$.  Since the free Airy phase satisfies $\tau=\xi^3$, we have
\begin{align}\label{p2ws}
        \widetilde {\mathcal Q}_N\big(\Tr_x S(\cdot)f\big)
        =
        \Tr_x S(\cdot)\widetilde Q_N f,
\end{align}
where $\widetilde{\mathcal{Q}}_{N}=\chi(\frac{D_t}{N})$  denotes the time frequency cutoff, see Lemma \ref{lefc}. 
By (\ref{p2ws}) and the definition (\ref{ye}) of $\mathcal{Y}_{E}$ norm, Lemma \ref{lem:tail} and Lemma  \ref{lehfs}  give
\begin{align}\label{y7vv}
\| S(t)\widetilde{Q}_{N} f\|_{\mathcal{Y}_{E}(I_{T})}\le\| S(t)f\|_{\mathcal{Y}_{E}(0,T)} +
        C_T e^{-c_TN^{3/\sigma}}\|f\|_{L^2},
\end{align}
and thus 
\begin{align}\label{y9vv}
\| S(t)\widetilde{P}_{N} f\|_{\mathcal{Y}_{E}(I_{T})}\le 2\| S(t)f\|_{\mathcal{Y}_{E}(0,T)} +
        C_T e^{-c_TN^{3/\sigma}}\|f\|_{L^2},
\end{align}
provided that $N$ is sufficiently large.  
Hence after increasing $N$, one obtains from (\ref{Huoo7}) and (\ref{y7vv}) that 
\begin{align}\label{Hu855}
        \|\widetilde Q_N f\|_{L^2}^2
        \leq
        C_T\|S(\cdot)f\|_{\YY_E(0,T)}^2
        +
        C_{T} e^{-c_TN^{3/\sigma}}\|\widetilde P_N f\|_{L^2}^2 .
\end{align}

\paragraph{Step 4: Low frequency Hausdorff propagation.}
Let $g=\widetilde P_N f$ and set
\begin{align*}
        F_g(x)=\Tr_x S(\cdot)g\big|_{I_T}\in L^2(I_T).
\end{align*}
The function $F_g$ is $L^2(I_T)$-valued and compactly supported in the spatial frequency  space.  The Hilbert-valued
Hausdorff--Remez estimate (see Lemma \ref{lemHH})  yields
\begin{align*}
        \int_{\mathbb R}\|F_g(x)\|_{L^2(I_T)}^2\,dx
        \leq
        C e^{C(1+N){}}
        \sum_{j\in\mathbb Z}
        \sup_{x\in E_j}\|F_g(x)\|_{L^2(I_T)}^2 .
\end{align*}
By unitarity,
\begin{align*}
        \int_{\mathbb R}\|F_g(x)\|_{L^2(I_T)}^2\,dx
        =
        |I_T|\|g\|_{L^2}^2 .
\end{align*}
Thus (\ref{y9vv}) yields 
\begin{align}
        \|\widetilde P_N f\|_{L^2}^2
       & \leq 
       CT^{-1}  e^{C(1+N){}}\|S(\cdot)\widetilde{P}_{N}f\|_{\YY_E(I_{T})}^2\nonumber \\
      &\leq C_{T}  e^{C(1+N){}}\|S(\cdot) f\|_{\YY_E(0,T)}^2+ C_T e^{C(1+N){}}e^{-c_TN^{\frac{3}{\sigma}}}\|f\|_{L^2}^2\nonumber \\
       &\leq   C_{N,T}\|S(\cdot)  f\|_{\YY_E(0,T)}^2+ C_T\delta_{N}\|f\|_{L^2}^2,\label{Hu444}
\end{align}
where $\lim_{N\to \infty}\delta_{N}=0$, since $1<\sigma<3$. See Lemma \ref{lesepl} for low frequency separation. 

\paragraph{Step 5: Absorption.}
Fix $T>0$. Letting $N$ be sufficiently large in (\ref{Hu444}) and combining the high and low frequency estimates in (\ref{Hu855}) and (\ref{Hu444}), we obtain    
\begin{align*}
        \|f\|_{L^2}^2
        \leq
        C_{T}\|S(\cdot)f\|_{\YY_E(0,T)}^2
        +
        \delta\|f\|_{L^2}^2
\end{align*}
with $\delta<1$.  Absorbing the last term proves the  time-trace observability
inequality on $\mathbb R$, see Section 7.

 \begin{remark}
 Note that in the above explanations, we omit the dependence of constants $C,c$ on $E,R,m$, which are fixed throughout the paper, and just keep the dependence of $C,c$ on $T>0$. 
 \end{remark}

\subsubsection{Point observability on the torus}

We next consider
\begin{align*}
        u_t+u_{xxx}+p(x)u=0,\qquad x\in\mathbb T,
\end{align*}
with
\begin{align*}
        L_p=\partial_x^3+p(x),\qquad D(L_p)=H^3(\mathbb T).
\end{align*}

\paragraph{Step 1: High frequency spectral decomposition for $L_p$.}
For general bounded real $p$, the operator $L_p$ has compact resolvent but is usually
non-normal.  The proof avoids assuming a global basis of eigenfunctions. Instead, we 
use the Riesz projections.  Let $P_n^0$ be the Fourier projection onto
$\mathbb C e^{inx}$.  For large $|n|$, the cubic gap around $-in^3$ dominates the
perturbation of bounded $p$.  The Riesz projection $P_n$ of $L_p$ near $-in^3$ is then
well-defined, one-dimensional, and satisfies
\begin{align*}
        \sum_{|n|\geq N_0}\|P_n-P_n^0\|^2<\infty .
\end{align*}
By the Bari--Markus theorem for projections, the high spectral subspaces form a Riesz
basis of subspaces.  Therefore
\begin{align*}
        L^2(\mathbb T)=X_0\dotplus H_\infty,
\end{align*}
where $X_0$ is finite-dimensional and invariant under $L_p$, and
\begin{align*}
        H_\infty
        =
        \overline{\operatorname{span}}\{\phi_n:\ |n|\geq N_0\},
\end{align*}
where $\{\phi_n\}_{|n|\ge N_0}$ denote the eigenfunctions of $L_p$ normalized by $\langle \phi_n, e^{inx}\rangle=1$. 
Moreover,
\begin{align*}
        \phi_n=(2\pi)^{-\frac{1}{2}}e^{inx}+r_n,\qquad \|r_n\|_{C^0(\mathbb T)}\to0.
\end{align*}
  Hence every fixed point $a\in\mathbb T$ sees all sufficiently high
modes:
\begin{align}\label{see}
        |\phi_n(a)|\geq c>0,
        \qquad |n|\geq N_0, \qquad \forall \mbox{ }a\in \mathbb T
\end{align}
by increasing $N_0$. 
This step is presented in Proposition \ref{prospec}. 

\paragraph{Step 2: Complex Ingham estimates for the high tail.}
The eigenvalues $\lambda_n$ of $L_p$ satisfy the cubic asymptotics
\begin{align*}
        \lambda_n=-in^3+O(1).
\end{align*}
Thus their imaginary parts have asymptotically increasing gaps, and the complex
Haraux--Ingham estimate can be applied to
\begin{align} \label{gnyu99}
        \sum_{|n|\geq N} c_n e^{-\lambda_n t}\phi_n(a).
\end{align}
Together with the Riesz equivalence of the high spectral decomposition and the lower
bound on $|\phi_n(a)|$,  (\ref{gnyu99}) satisfies  
\begin{align*}
\frac{1}{C} \sum_{|n|\ge N} |c_n|^2\le   \|\sum_{|n|\geq N} c_n e^{-\lambda_n t}\phi_n(a)\|^2_{L^2}\le {C} \sum_{|n|\ge N} |c_n|^2.    
\end{align*}
Therefore, with the lower bound (\ref{see}) all the high modes can be observed from any point 
$a\in \mathbb T$. See Proposition \ref {cuin} for this step. 

\paragraph{Step 3: Reduction to a finite-dimensional visibility problem.}
Let $F\subset\mathbb T$ be finite and non-empty.
All possible failures of point observability are contained in $X_0$.  Set
\begin{align*}
        \mathcal{A}=L_p|_{X_0},
        \qquad
        C_Fy=(y(a))_{a\in F}.
\end{align*}
The finite-dimensional component is observable from $F$ if and only if
\begin{align*}
        C_F e^{-t\mathcal{A}}y=0 \quad\text{for all }t\in(0,T)
        \quad\Longrightarrow\quad y=0.
\end{align*}
Equivalently, by the finite-dimensional Kalman--Hautus criterion,
\begin{align*}
        C_F\mathcal{A}^q y=0,\qquad q=0,1,\ldots,\dim X_0-1,
        \quad\Longrightarrow\quad y=0.
\end{align*}
The high frequency estimate and this finite-dimensional estimate are then combined by
a direct sum argument for exponential families.  This proves the necessary and
sufficient condition for finite point observability on the torus, see Section \ref{finite} for the details.

\subsection{Organization and notation}

The paper is organized as follows. 

In Section  \ref{pre}, we set up the trace observation operator and prove the admissibility.   Section \ref{remez} proves a Remez type inequality in the low frequency. 
Section \ref{gevrey} proves the Gevrey time filters estimates. Section \ref{high-cross} is devoted to the high frequency observability and separation. In Section \ref{separation}, we obtain the  low frequency separation estimate. Section \ref{proof-th1} finishes the proof of Theorem \ref{th1}. 

In Section \ref{related}, we prove consequences and related results for Airy type equations on the line, particularly Proposition \ref{p1} and Proposition \ref{p2}.

Section \ref{point} studies point observability for periodic Airy and linear KdV equations and  particularly proves Proposition \ref{c1}. 
Section \ref{high-spectral} gives the high frequency spectral decomposition for the  Airy operator with potential.  
Section \ref{finite} gives the finite point observability
criterion and proves Theorem \ref{th3}.
Section \ref{hausdorff} treats observation sets with accumulation points and proves  Corollary \ref{c2}.  

Throughout this paper, we make the convention that $C_{\gamma_1,...,\gamma_k}$ represents a constant depending on the parameters $\gamma_1,...,\gamma_k$, and it may vary  from line to line. The notation $A \asymp_{\gamma} B$ means that there exist constants $C_{\gamma}>0 $ and $c_{\gamma}>0$
such that $c_{\gamma}B\le A\le C_{\gamma} B$. We use $\lceil x \rceil$ to represent the smallest integer which  is greater than or equal to $x\in \mathbb R$. 

Recall the definition of  $s$-dimensional Hausdorff content on $\mathbb{R}$. 
Let $s\ge 0$, $E\subset \mathbb R$.
The $s$-dimensional Hausdorff content of $E$ is defined as
\begin{align}
\mathcal H_\infty^s(E)
:=
\inf\left\{
\sum_{k=1}^\infty \bigl|I_k\bigr|^s
\;\bigg|\;
E\subset \bigcup_{k=1}^\infty I_k,\ \text{each }I_k\subset\mathbb R \text{ is an interval}
\right\},
\end{align}
where $|I_k|=\operatorname{diam}(I_k)$ denotes the length of the interval $I_k$,
and the infimum ranges over all countable intervals whose union covers $E$.

The Fourier transform and Fourier inverse transform are defined by 
\begin{align*}
    \widehat f(\xi)=\frac1{\sqrt{2\pi}}\int_\R e^{-ix\xi}f(x)\dd x,
    \qquad
    f(x)=\frac1{\sqrt{2\pi}}\int_\R e^{ix\xi}\widehat f(\xi)\dd\xi .
\end{align*}

\section{Preliminaries}\label{pre}

In this section, we  focus on the  time-trace and admissibility of time-trace observations on regular sampling sequences. 

 .

\subsection{Point time-traces and regular sampling sequence }

The Airy group is
\begin{align*}
    S(t)f(x)=\frac1{\sqrt{2\pi}}\int_\R e^{i(x\xi+t\xi^3)}\widehat f(\xi)\dd\xi.
\end{align*}
For Schwartz $f$, this is a classical function.  For $f\in L^2(\R)$, point time-traces are defined by the admissibility results below.

\begin{lemma}
\label{sing}
For every $T>0$ there is a constant $C_T>0$ such that, for every
$x\in\mathbb R$, the classical map
\begin{align*}
    \mathcal S(\mathbb R)\ni f
    \longmapsto
    \bigl(S_{\mathbb R}(\cdot)f\bigr)(x)
\end{align*}
extends uniquely to a bounded linear map
\begin{align*}
    \Tr_x S(\cdot):
    L^2(\mathbb R)\longrightarrow L^2(0,T),
\end{align*}
and
\begin{align*}
    \|\Tr_xS(\cdot)f\|_{L^2(0,T)}
    \le C_T\|f\|_{L^2(\mathbb R)}.
\end{align*}
For $f\in\mathcal S(\mathbb R)$, this extended trace coincides with the
classical function
\begin{align*}
    t\mapsto (S(t)f)(x).
\end{align*}
\end{lemma}

\begin{proof}
Fix $x\in\mathbb R$. For $f\in\mathcal S(\mathbb R)$, define
\begin{align*}
    E_x f(t)
    =
    \frac1{\sqrt{2\pi}}
    \int_{\mathbb R}
    e^{i(x\xi+t\xi^3)}\widehat f(\xi)\,d\xi,
    \qquad 0<t<T.
\end{align*}
It suffices to prove
\begin{align*}
    \|E_x f\|_{L^2(0,T)}
    \le C_T\|f\|_{L^2(\mathbb R)}
\end{align*}
with $C_T$ independent of $x$.

We prove this by the $TT^*$ argument. For $g\in C_c^\infty(0,T)$, the formal
adjoint $E^*_x$ of $E_x$ fulfills
\begin{align*}
   {E_x^*g}(y)
    =
    \frac1{\sqrt{2\pi}}
   \int_{\mathbb R} e^{-ix\xi-iy\xi}
   \left( \int_0^T e^{-is\xi^3}g(s)\,ds\right) d\xi.
\end{align*}
Therefore
\begin{align*}
    \|{E_x^*g}\|^2_{L^2_y}&=\|\mathcal{F}_{y}[E_x^*g]\|_{L^2_\xi}^2
    =
    \frac1{2\pi}
    \int_{\mathbb R}
    \left|
        \int_0^T e^{-is\xi^3}g(s)\,ds
    \right|^2d\xi                                      \\
    &=
    \int_0^T\int_0^T
    K(s-t)g(s)\overline{g(t)}\,ds\,dt,
\end{align*}
where
\begin{align*}
    K(\tau)
    =
    \frac1{2\pi}
    \int_{\mathbb R}e^{-i\tau\xi^3}\,d\xi.
\end{align*}

The  estimate for $K$ follows by the standard stationary phase method. For $\tau\neq0$, the change of variables
\begin{align*}
    \eta=|\tau|^{1/3}\xi
\end{align*}
gives
\begin{align*}
    K(\tau)
    =
    \frac1{2\pi}|\tau|^{-1/3}
    \int_{\mathbb R} e^{-i\,\operatorname{sgn}(\tau)\eta^3}\,d\eta .
\end{align*}
The last integral is bounded  as an oscillatory integral. Hence
\begin{align*}
    |K(\tau)|
    \le C|\tau|^{-1/3},
    \qquad \tau\neq0.
\end{align*}
Since
\begin{align*}
    |\tau|^{-1/3}\in L^1(-T,T),
\end{align*}
by  extending $g$ by zero outside $(0,T)$ and Young's inequality,  we obtain 
\begin{align*}
    \|E_x^*g\|_{L^2(\mathbb R)}^2
    &\le
    \int_0^T\int_0^T
    |K(s-t)|\,|g(s)|\,|g(t)|\,ds\,dt                         \\
    &\le
    C_T\|g\|_{L^2(0,T)}^2.
\end{align*}
Thus
\begin{align*}
    E_x^*:L^2(0,T)\longrightarrow L^2(\mathbb R)
\end{align*}
extends to a bounded operator with norm at most $C_T^{1/2}$, uniformly in
$x$. By duality, $E_x$ extends to a bounded operator
\begin{align*}
    E_x:L^2(\mathbb R)\longrightarrow L^2(0,T)
\end{align*}
with the same bound, i.e., 
\begin{align*}
    \|E_x f\|_{L^2(0,T)}
    \le C_T\|f\|_{L^2(\mathbb R)}.
\end{align*}

Finally, the extension from
$\mathcal S(\mathbb R)$ to $L^2(\mathbb R)$ follows by density, and for
Schwartz data the extended trace coincides with the classical integral formula.
\end{proof}

 \begin{remark}
There are many results on the dispersive smoothing effects on the Airy equation. For example,  Kenig-Ponce-Vega \cite{KPV1993} proved  
$$
\|\partial_x S(t)f\|_{L^{\infty}_x L^2_t}\le C\|f\|_{L^2},
$$
which inspires Lemma \ref{sing}. 
 \end{remark}

Lemma \ref{sing}
is a toy model for the dispersive local smoothing used in 
this work. For further applications, we will take traces along a point sequence rather than a single point. The following is the definition of regular sequence which will be used as the sampling points.   Sequences of this type are standard
in nonuniform sampling theory, see e.g.
\cite[Chapters~6 and~10]{H1996}. 

\begin{definition}\label{regular}
A sequence $\Lambda=\{x_j\}_{j\in\Z}\subset\R$ is called $(\rho,L)$-regular if
\begin{align*}
    0<\rho\le x_{j+1}-x_j\le L<\infty,
    \qquad \forall  j\in\Z.
\end{align*}
For such a sequence and $f\in C(\mathbb R)$, set
\begin{align*}
    C_\Lambda f=(f(x_j))_{j\in\Z}.
\end{align*}
\end{definition}

For regular sequences, we have the sampling upper bound,  see Lemma \ref{samu}. It is a standard discrete Sobolev estimate. If the function is compactly supported in the spatial frequency  space, then one has a  sampling lower bound, see Lemma \ref{nasa}. It is an elementary maximum-gap sampling estimate for
functions supported in a sufficiently narrow frequency band, see e.g. \cite{ H1996}. We include the short proofs of the two Lemmas for completeness.

\begin{lemma}\label{samu}
If $\Lambda$ is $(\rho,L)$-regular, then
\begin{align*}
    \sum_j |g(x_j)|^2\le C_\rho\|g\|_{H^1(\R)}^2,
    \qquad g\in H^1(\R).
\end{align*}
\end{lemma}

\begin{proof}
Let $I_j=[x_j-\rho/3,x_j+\rho/3]$.  Note that these intervals are disjoint.  The  one dimensional Sobolev embedding  gives
\begin{align*}
    |g(x_j)|^2\le C_\rho\int_{I_j}(|g|^2+|g'|^2)\dd x.
\end{align*}
Summing over $j$ proves the claim. Note that the constant $C_{\rho}$ is independent of $x_j$ by applying a shift transform.  
\end{proof}

\begin{lemma}\label{nasa}
Let $\Lambda$ be $(\rho,L)$-regular.  There exists $a=a(L)>0$ such that, for every $r\in\R$ and every $v\in L^2(\R)$ with
\begin{align*}
    \supp\widehat v\subset[r-a,r+a],
\end{align*}
one has
\begin{align*}
    \|v\|_2^2\le 4L\sum_j |v(x_j)|^2.
\end{align*}
\end{lemma}
\begin{proof}
Choose $a>0$ so that $2L^2a^2\le 1/2$.  Set $w(x)=e^{-irx}v(x)$.  Then $\supp\widehat w\subset[-a,a]$ and Bernstein's inequality gives $\|w'\|_2\le a\|w\|_2$.  On each interval $[x_j,x_{j+1}]$, there holds 
\begin{align*}
    |w(x)|^2\le 2|w(x_j)|^2+2L\int_{x_j}^{x_{j+1}}|w'(y)|^2\dd y .
\end{align*}
Integrating the above formula in $x\in [x_j,x_{j+1}] $ and summing in $j$ yields
\begin{align*}
    \|w\|_2^2\le 2L\sum_j|w(x_j)|^2+2L^2\|w'\|_2^2
       \le 2L\sum_j|w(x_j)|^2+\frac12\|w\|_2^2.
\end{align*}
Absorbing the last term proves the result.
\end{proof}

The following is the Plancherel--Pólya upper sampling inequality for
uniformly separated sequences, stated in the form needed below, see e.g.  \cite[Chapters~6 and~10]{H1996}.
For completeness, we give a short proof based on the discrete Sobolev
estimate of Lemma~\ref{samu} and Bernstein's inequality. 
\begin{lemma}\label{naup}
Let $\Lambda$ be $(\rho,L)$-regular and let $a>0$ be fixed.  If
\begin{align*}
    \supp\widehat v\subset[r-a,r+a]
\end{align*}
for some $r\in\R$, then
\begin{align*}
    \sum_j |v(x_j)|^2\le C_{\rho,a}\|v\|_2^2 .
\end{align*}
\end{lemma}

\begin{proof}
Let $w(x)=e^{-irx}v(x)$. Hence $\supp\widehat w\subset[-a,a]$ and $\|w'\|_2\le  C a\|w\|_2$.  Applying Lemma \ref{samu} to $w$ gives 
\begin{align*}
    \sum_j|v(x_j)|^2=\sum_j|w(x_j)|^2\le C_{\rho}\|w\|^2_{H^1}\le  C_{\rho,a}\|w\|_2^2=C_{\rho,a}\|v\|_2^2 .
\end{align*}
\end{proof}

\subsection{A unitary Hautus criterion}

We use the following standard consequence of Miller's resolvent criterion for unitary groups \cite{Miller2012}.  It is stated in the form needed below.

\begin{proposition}[\cite{Miller2012}]\label{prop:Miller}
Let $A$ be self-adjoint on a Hilbert space $X$, and let $C:D(A)\to Y$ be bounded from $D(A)$ with the graph norm to another Hilbert space $Y$.

\smallskip
\noindent\emph{Admissibility criterion.}  If there exist $L,l>0$ such that
\begin{equation}\label{ard}
    \|Cz\|_Y^2\le L\|(A-\lambda)z\|_X^2+l\|z\|_X^2,
    \qquad z\in D(A),\quad \lambda\in\R,
\end{equation}
then $C$ is admissible for the unitary group $e^{itA}$ on every bounded time interval; that is, for every bounded interval $J$ there is $A_{J,L,l}>0$ such that
\begin{align*}
    \int_J\|Ce^{itA}z\|_Y^2\dd t\le A_{J,L,l}\|z\|_X^2,
    \qquad z\in X.
\end{align*}

\smallskip
\noindent\emph{Observability criterion.}  Assume that $C$ is admissible and that there are $M,m>0$ such that
\begin{equation}\label{eq:hautus}
    \|z\|_X^2\le M\|(A-\lambda)z\|_X^2+m\|Cz\|_Y^2,
    \qquad z\in D(A),\quad \lambda\in\R .
\end{equation}
Then, for every $T>\pi\sqrt M$, there is $K_{T,M,m}>0$ such that
\begin{align*}
    \|z\|_X^2\le K_{T,M,m}\int_0^T\|Ce^{itA}z\|_Y^2\dd t,
    \qquad z\in X.
\end{align*}
The same estimate holds on any interval of length $T>0$ with the same constants. Here, $A_{J,L,l}$ and $K_{T,M,m}$ can be chosen as 
\begin{align}\label{opcc}
A_{J,L,l}=l|J|+\sqrt{Ll}, \mbox{ }
\mbox{ }K_{T,M,m}=\frac{2mT}{T^2-\pi^2M}.
\end{align}
\end{proposition}

\subsection{Trace admissibility for regular sequences and blocks}

In this subsection, we extend the single point time-trace introduced in Lemma \ref{sing} to regular sequences.

\begin{proposition}\label{psa}
Let $\Lambda=\{x_j\}_{j\in\mathbb Z}$ be $(\rho,L)$-regular.  For every bounded interval $J\subset\R$ there exists $A_{|J|,\rho,L}>0$ such that the classical map
\begin{align*}
    H^3(\R)\ni f\longmapsto \bigl(S(\cdot)f(x_j)\bigr)_{j\in\mathbb Z}
    \in L^2(J;\ell^2)
\end{align*}
extends uniquely to a bounded map
\begin{align*}
    \Tr_\Lambda S(\cdot):L^2(\R)\longrightarrow L^2(J;\ell^2),
\end{align*}
and
\begin{equation}\label{eq:selected-admiss}
    \int_J\sum_j\bigl|\bigl(\Tr_\Lambda S(t)f\bigr)_j\bigr|^2\dd t
       \le A_{|J|,\rho,L}\|f\|_2^2,
    \qquad f\in L^2(\R),
\end{equation}
where $A_{|J|,\rho,L}$ can be chosen as 
$$A_{|J|,\rho,L}=C_{\rho,L}(1+|J|)$$
for some $C_{\rho,L}>0$ depending only on $(\rho,L)$. 
For $f\in H^3(\R)$, $(\Tr_\Lambda S(t)f)_j=S(t)f(x_j)$ for a.e. $t\in J$.
\end{proposition}

\begin{proof}
For $f\in H^3(\R)$ define $C_\Lambda f=(f(x_j))_j$.  Lemma \ref{samu} gives $C_\Lambda\in\mathcal L(H^3,\ell^2)$.  We verify the admissibility resolvent estimate \eqref{ard} for $A=D^3$, where $D=-i\partial_x$.

Let $\lambda=r^3$ and define the frequency projections $P_r,Q_r$ by
\begin{align*}
    \widehat{P_rf}(\xi)=\mathbf 1_{\{|\xi-r|\le 2\}}\widehat f(\xi),
    \qquad Q_r=I-P_r.
\end{align*}
By Lemma \ref{naup},
\begin{align}\label{y7gfb}
    \|C_\Lambda P_rf\|_{\ell^2}\le C\|P_rf\|_2\le C\|f\|_2 .
\end{align}
On the support of $Q_r$,
\begin{align*}
    |\xi-r|>2,
    \qquad
    \xi^3-r^3=(\xi-r)(\xi^2+\xi r+r^2),
\end{align*}
and $\xi^2+\xi r+r^2\ge \frac{1}{2}(\xi^2+r^2)$. Hence, by considering two cases $|r|\ge 1$ and $|r|<1$, we find 
\begin{align*}
    \sup_{|\xi-r|>2}\frac{\langle\xi\rangle}{|\xi^3-r^3|}<\infty
\end{align*}
with a constant independent of $r$.  Therefore
\begin{align*}
    \|Q_rf\|_{H^1}\le C\|(D^3-\lambda)f\|_2.
\end{align*}
Lemma \ref{samu} gives 
\begin{align}\label{7gfb}
    \|C_\Lambda Q_rf\|_{\ell^2}\le C\|(D^3-\lambda)f\|_2.
\end{align}
Combining   (\ref{y7gfb}) with (\ref{7gfb})  yields
\begin{align*}
    \|C_\Lambda f\|_{\ell^2}^2
       \le L_0\|(D^3-\lambda)f\|_2^2+l_0\|f\|_2^2,
       \qquad \lambda\in\R.
\end{align*}
The admissibility part of Proposition \ref{prop:Miller} gives the asserted bounded extension and estimate.  Uniqueness follows from the density of $H^3(\R)$ in $L^2(\R)$.
\end{proof}

In Lemma \ref{sing}, we have defined the single point time-trace. We now give another definition of single point time-trace in the setting of Proposition  \ref{psa} and prove that the two definitions indeed coincide. 
\begin{definition}\label{dsst}
Fix $x\in\R$ and a bounded time interval $J$.  Choose any $(\rho,L)$-regular sequence $\Lambda=\{x_j\}_{j\in\mathbb Z}$ containing $x$ as one of its elements, say $x=x_{j_0}$.  For $f\in L^2(\R)$ we define
\begin{align*}
    \Tr_xS(\cdot)f\big|_J:=\bigl(\Tr_\Lambda S(\cdot)f\bigr)_{j_0}
    \in L^2(J).
\end{align*}

\end{definition}

This definition is independent of the regular sequence and of the index chosen for $x$. Indeed, we have the following result. 
\begin{lemma}\label{lemc}
The map in Definition~\ref{dsst} is the unique bounded extension of the classical point-value map
\begin{align*}
    H^3(\R)\ni f\longmapsto S(\cdot)f(x)\in L^2(J).
\end{align*}
If $x=x_j$ is included in any regular sequence $\Lambda$, then $\Tr_xS(\cdot)f$ agrees with the corresponding coordinate of $\Tr_\Lambda S(\cdot)f$.
\end{lemma}

\begin{proof}
The map $\Tr_{\Lambda}$ in Proposition~\ref{psa} is bounded from $L^2(\R)$ to $L^2(J)$. For $H^3$ data it coincides with the classical point value.  If two regular sequences contain the same point, the two coordinate maps agree on the dense subspace $H^3(\R)$; hence their bounded extensions to $L^2(\R)$ agree.  This proves both independence and uniqueness.
\end{proof}

Let now $E$ satisfy \eqref{thicko}.  Throughout the real line part we set
\begin{align*}
    E_j=E\cap[3Rj,3Rj+R].
\end{align*}
For an interval $J\subset\R$ define the block supremum trace norm
\begin{equation}\label{eq:blocknorm}
    \|S(\cdot)f\|_{\YY_E(J)}^2
    :=\sum_{j\in\Z}\sup_{x\in E_j}
       \bigl\|\Tr_xS(\cdot)f\bigr\|_{L^2(J)}^2 .
\end{equation}
Note that, equivalently, this is the supremum over all selections $x_j\in E_j$ of
\begin{align*}
    \int_J\sum_j |\Tr_{x_j}S(t)f|^2\dd t.
\end{align*}
Indeed, for a non-negative family $a_j(x)$ one has
\begin{align*}
    \sum_j\sup_{x\in E_j}a_j(x)
    =\sup_{\{x_j\in E_j\}}\sum_j a_j(x_j).
\end{align*}

Now, we are ready to prove the observation $\|S(\cdot)f\|_{\YY_E(J)}$ is admissible for all $L^2$ data. 
\begin{lemma} \label{balg}
For every bounded interval $J\subset\R$, there holds 
\begin{align*}
    \|S(\cdot)f\|_{\YY_E(J)}^2\le A_{|J|,R}\|f\|_2^2,
    \qquad f\in L^2(\R),
\end{align*}
where $A_{|J|,R}$ can be chosen as 
$$A_{|J|,R}=C_{R}(1+|J|)$$
for some $C_{R}>0$ depending only on $R$. 
\end{lemma}

\begin{proof}
For any selection $x_j\in E_j$, one has
\begin{align*}
    2R\le x_{j+1}-x_j\le 4R.
\end{align*}
Hence the selected sequence is $(2R,4R)$-regular.  Proposition \ref{psa} gives a bound independent of the selection.  Taking the supremum over selections proves the claim.
\end{proof}

\section{Hausdorff--Remez propagation of smallness}\label{remez}

This section gives the precise propagation of smallness for the low frequency part.

Lemma \ref{hkey} below  is a convenient fixed domain formulation of the known Hausdorff content propagation of smallness inequality for holomorphic functions. It follows, up to a standard finite chain localization, from  Proposition 2.3 of \cite{Z2024}, See also \cite{GBMO2025} for the corresponding Remez framework. We include a proof to record the precise dependence on $s,m_0,U$ and the $e^{CB
}$ growth.

\begin{lemma}
\label{hkey}
Let $0<s\leq 1$, let $m_{0}>0$, and let
$U\subset\mathbb C$ be a connected open neighborhood of
$$
I=[0,1].
$$
Then there exists a constant
$$
C_*=C(s,m_{0},U)>0
$$
with the following property. If $A\subset I$ satisfies
$$
\mathcal H^{s}_{\infty}(A)\geq m_{0},
$$
and $g$ is holomorphic in $U$ and satisfies
$$
\|g\|_{L^\infty(U)}
   \leq e^{B}\|g\|_{L^\infty(I)}
\qquad\text{for some } B\geq 0,
$$
then
\begin{align}\label{RtkkL}
\|g\|_{L^\infty(I)}
   \leq e^{C_*B}\|g\|_{L^\infty(A)}.
\end{align}
\end{lemma}

\begin{proof}
We first recall the precise form of the holomorphic propagation
estimate that will be used. For $E\subset\mathbb C$, denote by
$$
\mathscr H^{s}_{\mathbb C}(E)
 :=
 \inf\left\{
       \sum_{j=1}^{\infty}r_{j}^{s}
       :
       E\subset\bigcup_{j=1}^{\infty}B(z_{j},r_{j})
     \right\}
$$
the Hausdorff content used in \cite[Proposition~2.3]{Z2024}.
That proposition, together with its proof, implies the following
uniform form: for every $s>0$ and $h>0$, there exists
$$
\alpha_{\mathrm Z}=\alpha_{\mathrm Z}(s,h)\in(0,1)
$$
such that, whenever $0<R<1/4$,
$$
G\subset B(z_{0},R),
\qquad
\mathscr H^{s}_{\mathbb C}(G)\geq h,
$$
and $F$ is holomorphic in $B(z_{0},5R)$, one has
\begin{equation}
\label{eq:Zhu-uniform-form}
\|F\|_{L^\infty(B(z_{0},R))}
 \leq
 \|F\|_{L^\infty(G)}^{\alpha_{\mathrm Z}}
 \|F\|_{L^\infty(B(z_{0},4R))}^{1-\alpha_{\mathrm Z}}.
\end{equation}
Indeed, \cite{Z2024}'s stated exponent is $1/(1+C_{\mathrm Z})$.
Moreover, the proof depends on the observation set through the
factor
$$
\left(\frac{20}{\mathscr H^{s}_{\mathbb C}(G)}\right)^{N/s}.
$$
Hence, $C_{\mathrm Z}$, and thus
$\alpha_{\mathrm Z}$, may be chosen uniformly under the lower
bound
$\mathscr H^{s}_{\mathbb C}(G)\geq h$.

We divide the proof into three steps.

\medskip
\noindent
\textbf{Step 1: Localization of the Hausdorff content.}

Since $I\Subset U$, we may choose
$$
0<\rho<\frac14
$$
so small that
\begin{equation}
\label{eq:rho-choice}
\overline{B(x,5\rho)}\subset U
\qquad\text{for every }x\in I.
\end{equation}
Choose $x_{1},\dots,x_{N}\in I$, where $N=N(\rho)$, such that
\begin{equation}
\label{eq:finite-cover-I}
I\subset\bigcup_{\nu=1}^{N}B(x_{\nu},\rho).
\end{equation}
Set
$$
A_{\nu}:=A\cap B(x_{\nu},\rho).
$$
By the countable subadditivity of Hausdorff content,
$$
m_{0}
 \leq \mathcal H^{s}_{\infty}(A)
 \leq \sum_{\nu=1}^{N}
       \mathcal H^{s}_{\infty}(A_{\nu}).
$$
Therefore, there exists an index $\nu_{0}\in\{1,\dots,N\}$
such that
\begin{equation}
\label{lorec}
\mathcal H^{s}_{\infty}(A_{\nu_{0}})
 \geq \frac{m_{0}}{N}.
\end{equation}

We next compare the one-dimensional content with \cite{Z2024}'s planar
content. If $E\subset\mathbb R$, then
\begin{equation}
\label{eq:content-comparison}
\mathcal H^{s}_{\infty}(E)
 \leq 2^{s}\mathscr H^{s}_{\mathbb C}(E).
\end{equation}
Indeed, if
$$
E\subset\bigcup_{j}B(z_{j},r_{j}),
$$
then each intersection
$B(z_{j},r_{j})\cap\mathbb R$ is contained in an interval of
length at most $2r_{j}$. Hence
$$
\mathcal H^{s}_{\infty}(E)
 \leq \sum_{j}(2r_{j})^{s}.
$$
Taking the infimum over all disk coverings proves
\eqref{eq:content-comparison}.

Combining \eqref{lorec} and
\eqref{eq:content-comparison}, we obtain
\begin{equation}
\label{eq:localized-complex-content}
\mathscr H^{s}_{\mathbb C}(A_{\nu_{0}})
 \geq h_{0},
\qquad
h_{0}:=\frac{m_{0}}{2^{s}N}>0.
\end{equation}

\medskip
\noindent
\textbf{Step 2: Local propagation from $A$ to a disk.}

For brevity, denote
$$
x_{*}:=x_{\nu_{0}},
\qquad
G:=A_{\nu_{0}}.
$$
Then
$$
G\subset B(x_{*},\rho),
\qquad
\mathscr H^{s}_{\mathbb C}(G)\geq h_{0}.
$$
In view of \eqref{eq:rho-choice}, the function $g$ is
holomorphic in $B(x_{*},5\rho)$. Thus
\eqref{eq:Zhu-uniform-form}, applied with
$$
R=\rho,
\qquad
\alpha_{0}:=\alpha_{\mathrm Z}(s,h_{0}),
$$
gives
\begin{equation}
\label{eq:local-Zhu}
\|g\|_{L^\infty(B(x_{*},\rho))}
 \leq
 \|g\|_{L^\infty(G)}^{\alpha_{0}}
 \|g\|_{L^\infty(B(x_{*},4\rho))}^{1-\alpha_{0}}.
\end{equation}
Since $G\subset A$ and $B(x_{*},4\rho)\subset U$, it follows
that
\begin{equation}
\label{eq:local-bound}
q
 :=
 \|g\|_{L^\infty(B(x_{*},\rho))}
 \leq
 m^{\alpha_{0}}M^{1-\alpha_{0}},
\end{equation}
where
$$
m:=\|g\|_{L^\infty(A)},
\qquad
M:=\|g\|_{L^\infty(U)}.
$$

\medskip
\noindent
\textbf{Step 3: Propagation from the local disk to the entire
interval.}

We claim that there exists
$$
\beta=\beta(U)>0
$$
such that
\begin{equation}
\label{eq:disk-to-I}
\|g\|_{L^\infty(I)}
 \leq q^{\beta}M^{1-\beta}.
\end{equation}

Fix $x\in I$. Choose points
$$
c_{0},c_{1},\dots,c_{\ell}\in I
$$
such that
$$
c_{0}=x_{*},
\qquad
c_{\ell}=x,
\qquad
|c_{j}-c_{j-1}|\leq\frac{\rho}{2},
$$
and
\begin{equation}
\label{eq:chain-length}
\ell\leq L,
\qquad
L:=\left\lceil\frac{2}{\rho}\right\rceil.
\end{equation}
Such a chain is obtained, for example, by dividing the segment
joining $x_{*}$ to $x$ into sufficiently many equal parts.

For $j=0,\dots,\ell$, set
$$
M_{j}:=\|g\|_{L^\infty(B(c_{j},\rho))}.
$$
Since
$$
B(c_{j},\rho/2)\subset B(c_{j-1},\rho),
$$
Applying Hadamard's three-circle theorem to the center $c_j$
with radii
$
\frac{\rho}{2},\rho,2\rho,
$
yields
\begin{align}
M_{j}
&\leq
 \|g\|_{L^\infty(B(c_{j},\rho/2))}^{1/2}
 \|g\|_{L^\infty(B(c_{j},2\rho))}^{1/2}
\nonumber\\
&\leq M_{j-1}^{1/2}M^{1/2}.
\label{eq:one-chain-step}
\end{align}
Here $B(c_{j},2\rho)\subset U$ follows from
\eqref{eq:rho-choice}. Iterating \eqref{eq:one-chain-step}, and
using $M_{0}=q$,  one obtains 
\begin{align*}
M_{j}
 \leq
 q^{2^{-j}}M^{1-2^{-j}},
\qquad 0\leq j\leq\ell.
\end{align*}
Since $q\leq M$, $\ell\leq L$, and set
$$
\beta:=2^{-L},
$$
we obtain
$$
|g(x)|
 \leq M_{\ell}
 \leq q^{2^{-\ell}}M^{1-2^{-\ell}}
 \leq q^{\beta}M^{1-\beta}.
$$
Taking the supremum over $x\in I$ proves
\eqref{eq:disk-to-I}.

Combining \eqref{eq:disk-to-I} with
\eqref{eq:local-bound}, we find
\begin{align}
\|g\|_{L^\infty(I)}
&\leq
 \left(m^{\alpha_{0}}M^{1-\alpha_{0}}\right)^{\beta}
 M^{1-\beta}
\nonumber\\
&=
 m^{\alpha}M^{1-\alpha},
\label{glointer}
\end{align}
where
$$
\alpha:=\alpha_{0}\beta\in(0,1).
$$
The number $\alpha$ depends only on $s$, $m_{0}$, and $U$.

Finally, let
$$
K:=\|g\|_{L^\infty(I)}.
$$
If $K=0$, there is nothing to prove. Assume $K>0$. From the
hypothesis
$$
M\leq e^{B}K
$$
and \eqref{glointer}, we obtain
$$
K
 \leq m^{\alpha}(e^{B}K)^{1-\alpha}.
$$
Thus
$$
K^{\alpha}
 \leq e^{(1-\alpha)B}m^{\alpha},
$$
and consequently
$$
K
 \leq
 \exp\left(\frac{1-\alpha}{\alpha}B\right)m.
$$
The conclusion follows by setting
$$
C_*:=\frac{1-\alpha}{\alpha}.
$$
Since $\alpha$ depends only on $s$, $m_{0}$, and $U$, so
does $C_*$.
\end{proof}

\begin{remark}
The preceding proof shows that the fixed domain estimate is a
consequence of Zhu's local disk estimate and Hadamard's
three-circle theorem.
\end{remark}

\subsection{Hilbert-valued version of propagation of  smallness}

We first prove  a scalar Hausdorff propagation  of  smallness result.

\begin{lemma} 
\label{scalH}
Let $0<s\leq 1$, $R>0$, and $m>0$. Suppose that
\begin{align}
E_j\subset [3Rj,3Rj+R],
\qquad
\mathcal H^s_\infty(E_j)\geq m,
\qquad j\in\mathbb Z.
\end{align}
Then there exist constants $C_0,C_1>0$, depending only on
$R$, $m$, and $s$, such that, for every $N\geq 1$ and every
nonnegative function $U\in L^1(\mathbb R)$ satisfying
\begin{align}
\operatorname{supp}\widehat U\subset[-N,N],
\end{align}
one has
\begin{equation}
\int_{\mathbb R}U(x)\,dx
\leq
C_0 e^{C_1N}
\sum_{j\in\mathbb Z}\sup_{x\in E_j}U(x).
\label{eq:scalar-hausdorff-remez}
\end{equation}
Here $U$ is identified with its continuous representative.
\end{lemma}

\begin{proof}
The case $U=0$ is trivial. Suppose that $U$ is non-zero. 

Since  $\operatorname{supp}\widehat U\subset[-N,N]$,
\begin{equation}
U(z)
=
\frac{1}{\sqrt{2\pi}}
\int_{-N}^{N}e^{iz\xi}\widehat U(\xi)\,d\xi,
\qquad z\in\mathbb C,
\label{entir}
\end{equation}
defines an entire function. Its
restriction to $\mathbb R$ is the continuous representative of the
original $L^1$-function.

 Bernstein estimate shows that 
there exists an absolute constant $C_{\mathrm B}\geq1$ such that
\begin{equation}
\bigl\|U^{(\ell)}\bigr\|_{L^1(\mathbb R)}
\leq
(C_{\mathrm B}N)^\ell
\|U\|_{L^1(\mathbb R)},
\qquad \ell=0,1,2,\ldots.
\label{BernL}
\end{equation}

For $j\in\mathbb Z$, set
\begin{equation*}
I_j:=[3Rj,3R(j+1)]
\quad\text{and}\quad
b_j:=\int_{I_j}U(x)\,dx.
\end{equation*}
The intervals $I_j$ cover $\mathbb R$ and overlap only at their
endpoints. Hence
\begin{equation}
\sum_{j\in\mathbb Z}b_j
=
\int_{\mathbb R}U(x)\,dx.
\label{eq:sum-bj}
\end{equation}
In addition,
\begin{equation*}
b_j>0
\qquad\text{for every }j\in\mathbb Z.
\end{equation*}
Indeed, if $b_j=0$, then continuity and nonnegativity imply that
$U$ vanishes identically on $I_j$. The identity theorem applied
to the entire extension \eqref{entir} would then give
$U\equiv0$, which is contrary to our assumption.

Now, we can use the bad-good interval decomposition argument of \cite{K2001} and \cite{HWW2026} in the Lebesgue and respectively log-Hausdorff settings. 
Fix
$
A:=3.
$
We call $I_j$ \emph{good} if, for every integer $\ell\geq1$,
\begin{equation}
\int_{I_j}|U^{(\ell)}(x)|\,dx
\leq
(A C_{\mathrm B}N)^\ell b_j.
\label{goodin}
\end{equation}
Otherwise $I_j$ is called \emph{bad}. For each bad interval, choose
an integer $\ell(j)\geq1$ for which
\eqref{goodin} fails. Then using
\eqref{BernL}, we obtain
\begin{align}
\sum_{j:\,I_j\ {\rm bad}}b_j
&\leq
\sum_{\ell=1}^{\infty}
(A C_{\mathrm B}N)^{-\ell}
\sum_{\substack{j:\,I_j\ {\rm bad}\\ \ell(j)=\ell}}
\int_{I_j}|U^{(\ell)}(x)|\,dx
\nonumber\\
&\leq
\sum_{\ell=1}^{\infty}
(A C_{\mathrm B}N)^{-\ell}
\bigl\|U^{(\ell)}\bigr\|_{L^1(\mathbb R)}
\nonumber\\
&\leq
\sum_{\ell=1}^{\infty}A^{-\ell}
\|U\|_{L^1(\mathbb R)}
\nonumber\\
&=
\frac{1}{A-1}\|U\|_{L^1(\mathbb R)}
=
\frac12\|U\|_{L^1(\mathbb R)}.
\label{bdgood}
\end{align}
It follows from \eqref{eq:sum-bj} and
\eqref{bdgood} that
\begin{equation}
\int_{\mathbb R}U(x)\,dx
\leq
2\sum_{j:\,I_j\ {\rm good}}b_j.
\label{goointer}
\end{equation}

We now derive a uniform holomorphic growth estimate on every good
interval. Fix a good interval $I_j$, and define
\begin{align*}
G_j(x)
:=
\sum_{\ell=0}^{\infty}
\frac{|U^{(\ell)}(x)|}
     {(2A C_{\mathrm B}N)^\ell},
\qquad x\in I_j,
\end{align*}
where the term corresponding to $\ell=0$ is $U(x)$. By monotone
convergence and \eqref{goodin},
\begin{align*}
\int_{I_j}G_j(x)\,dx
&=
b_j+
\sum_{\ell=1}^{\infty}
(2A C_{\mathrm B}N)^{-\ell}
\int_{I_j}|U^{(\ell)}(x)|\,dx
\nonumber\\
&\leq
b_j+
\sum_{\ell=1}^{\infty}2^{-\ell}b_j
=
2b_j.
\end{align*}
Since $|I_j|=3R$, there exists a point $x_j\in I_j$ such that
\begin{equation*}
G_j(x_j)
\leq
\frac{2b_j}{3R}.
\end{equation*}
Thus  for every integer $\ell\geq0$, one obtains 
\begin{equation}
|U^{(\ell)}(x_j)|
\leq
\frac{2b_j}{3R}
(2A C_{\mathrm B}N)^\ell.
\label{deriboundxi}
\end{equation}

Because $U$ is entire, its Taylor series centered at $x_j$
converges throughout $\mathbb C$. Thus
\eqref{deriboundxi} implies that, for every
$z\in\mathbb C$,
\begin{align}
|U(z)|
&\leq
\sum_{\ell=0}^{\infty}
\frac{|U^{(\ell)}(x_j)|}{\ell!}|z-x_j|^\ell
\nonumber\\
&\leq
\frac{2b_j}{3R}
\sum_{\ell=0}^{\infty}
\frac{(2A C_{\mathrm B}N|z-x_j|)^\ell}{\ell!}
\nonumber\\
&=
\frac{2b_j}{3R}
\exp\!\bigl(2A C_{\mathrm B}N|z-x_j|\bigr).
\label{comgrowth}
\end{align}

Let
\begin{equation*}
\Omega_*:=
\left\{
z\in\mathbb C:
\operatorname{dist}(z,[0,1])<1
\right\}.
\end{equation*}
This is a fixed connected open neighborhood of $[0,1]$. Define
\begin{equation*}
g_j(z):=U(3Rj+3Rz),
\qquad
\widetilde E_j
:=
\frac{E_j-3Rj}{3R}.
\end{equation*}
Then $g_j$ is entire and
\begin{align*}
\widetilde E_j\subset[0,1/3]\subset[0,1].
\end{align*}
By the scaling property of Hausdorff content,
\begin{align*}
\mathcal H^s_\infty(\widetilde E_j)
=
(3R)^{-s}\mathcal H^s_\infty(E_j)
\geq
m_*,
\qquad
m_*:=\frac{m}{(3R)^s}>0.
\end{align*}

Set
\begin{align*}
y_j:=\frac{x_j-3Rj}{3R}\in[0,1].
\end{align*}
If $z\in\Omega_*$, then there exists $t\in[0,1]$ such that
$|z-t|<1$. Therefore
\begin{align*}
|z-y_j|
\leq
|z-t|+|t-y_j|
<2,
\end{align*}
and hence
\begin{align*}
|3Rj+3Rz-x_j|
=
3R|z-y_j|
<6R.
\end{align*}
Using \eqref{comgrowth}, together with
\begin{align*}
b_j
=
\int_{I_j}U(x)\,dx
\leq
3R\sup_{x\in I_j}U(x),
\end{align*}
we obtain
\begin{align*}
\sup_{z\in\Omega_*}|g_j(z)|
&\leq
\frac{2b_j}{3R}
\exp(12A C_{\mathrm B}RN)
\nonumber\\
&\leq
2\exp(12A C_{\mathrm B}RN)
\sup_{x\in I_j}U(x)
\nonumber\\
&=
\exp(B_N)
\sup_{x\in[0,1]}|g_j(x)|,
\end{align*}
where
\begin{equation*}
B_N:=\log2+12A C_{\mathrm B}RN.
\end{equation*}

Let
\begin{align*}
C_*:=C(s,m_*,\Omega_*)
\end{align*}
be the constant in Lemma~\ref{hkey}. Applying Lemma~\ref{hkey} to $g_j$,
with observation set $\widetilde E_j$, complex neighborhood
$\Omega_*$, and growth parameter $B_N$, gives
\begin{align}
\sup_{x\in I_j}U(x)
&=
\sup_{x\in[0,1]}|g_j(x)|
\nonumber\\
&\leq
\exp(C_*B_N)
\sup_{x\in\widetilde E_j}|g_j(x)|
\nonumber\\
&=
2^{C_*}
\exp(12A C_*C_{\mathrm B}RN)
\sup_{x\in E_j}U(x).
\label{localgood}
\end{align}
This estimate is uniform in $j$.

Finally, combining \eqref{goointer} with
\eqref{localgood}, we conclude that
\begin{align*}
\int_{\mathbb R}U(x)\,dx
&\leq
2\sum_{j:\,I_j\ {\rm good}}b_j
\\
&\leq
6R\sum_{j:\,I_j\ {\rm good}}
\sup_{x\in I_j}U(x)
\\
&\leq
6R\,2^{C_*}
\exp(12A C_*C_{\mathrm B}RN)
\sum_{j\in\mathbb Z}\sup_{x\in E_j}U(x).
\end{align*}
Thus \eqref{eq:scalar-hausdorff-remez} holds with
\begin{align*}
C_0:=6R\,2^{C_*},
\qquad
C_1:=12A C_*C_{\mathrm B}R.
\end{align*}
Since
\begin{align*}
m_*=\frac{m}{(3R)^s},
\end{align*}
the constant $C_*$, and therefore $C_0$ and $C_1$, depends only
on $R$, $m$, and $s$. This completes the proof.
\end{proof}

\begin{remark}
Lemma \ref{scalH} can also be proved by other Remez type estimates without using Lemma~\ref{hkey}. Lemma~\ref{hkey} is of independent interest and  may be applied to other problems. 
\end{remark}

Now, we extend Lemma \ref{scalH} to the Hilbert-valued form. 
\begin{lemma} \label{lemHH}
Let $H$ be a separable Hilbert space.  Suppose $F\in L^2(\mathbb R;H)$ and its distributional Fourier transform in the $x$ variable is supported in $[-M,M]$.  Then
\begin{equation}\label{HH}
    \int_\mathbb R \|F(x)\|_H^2\,dx
    \le C_0 \exp\{C_1(1+M){}\}
    \sum_j\sup_{x\in E_j}\|F(x)\|_H^2.
\end{equation}
\end{lemma}

\begin{proof}
Choose a sequence of finite rank orthogonal projections $\Pi_K$ on $H$ so that $\Pi_{K}$ converges strongly to the identity as $K\to \infty$.  Write $F_K=\Pi_KF$.  Then the Fourier transform of $F_K$ in the $x$ variable is still supported in $[-M,M]$. After choosing an orthonormal basis $\{e_{\ell }\}^{K}_{\ell=1 }$ of $\Pi_K H$, define the scalar-valued function  $U_K(x)$ by 
\begin{align*}
    U_K(x):=\|F_K(x)\|_H^2=\sum_{\ell=1}^K |F_{K,\ell}(x)|^2, \mbox{ }F_{K,\ell }=  \langle \Pi_{K} F,e_{\ell}\rangle.
\end{align*}
Note that it has Fourier support contained in $[-2M,2M]$.  Applying Lemma~\ref{scalH} to $U_K$ yields
\begin{align*}
    \int_\mathbb R\|F_K(x)\|_H^2\,dx
    \le C\exp\{C(1+M){}\}
       \sum_j\sup_{x\in E_j}\|F_K(x)\|_H^2.
\end{align*}
Since $\|F_K(x)\|_H\le\|F(x)\|_H$, the right hand side is bounded by the same expression with $F$.  Letting $K\to\infty$ and using dominate convergence theorem on the left gives \eqref{HH}.
\end{proof}

\section{Gevrey time filters}\label{gevrey}

We need compactly supported cutoff functions whose Fourier transforms have sub-exponential decay.  Fix once and for all a Gevrey order
\begin{align}\label{gevord}
    1<\sigma<3.
\end{align}
Appendix B shows there exists a function   $\chi\in C_c^\infty(\R)$ of Gevrey class $\sigma$, with
\begin{equation}\label{eq:chi-def}
    \chi\ge 0,
    \qquad
    \chi(\tau)=1\quad(|\tau|\le1),
    \qquad
    \chi(\tau)=0\quad(|\tau|\ge2).
\end{equation}
By the ultradifferentiable Paley-Wiener theorem (e.g. \cite{{R1992}}), its inverse Fourier transform $K=\check\chi$  satisfies
\begin{equation}\label{gekernl}
    |K(t)|\le C e^{-c|t|^{1/\sigma}}.
\end{equation}
Recall $D=-i\partial_x$. For $N\ge1$ set
\begin{equation}\label{eq:QNPN}
   \widetilde{P}_N=\chi(D^3/N^3),
    \qquad
    \widetilde{Q}_N=I- \widetilde{P}_N .
\end{equation}
Hence $\widetilde{P}_N$ is a smooth low frequency cutoff, and ${\widetilde{Q}}_N$ is a smooth high frequency cutoff respectively.  Let $D_t=-i\partial_t$. On the time side define 
\begin{align*}
   \widetilde{\mathcal{P}}_{N}=\chi(D_t/N^3),
    \qquad
   \widetilde{\mathcal{Q}}_{N}=I- \widetilde{\mathcal{P}}_{N}.
\end{align*}

\begin{lemma}\label{lefc}
For every $x\in\mathbb R$, every $f\in L^2(\mathbb R)$ and every bounded time interval $J$,
\begin{equation}\label{eqtifilter}
   \widetilde{\mathcal{P}}_{N}\bigl(\Tr_xS(\cdot)f\bigr)=\Tr_xS(\cdot)\widetilde{P}_Nf,
    \qquad
   \widetilde{\mathcal{Q}}_{N}\bigl(\Tr_xS(\cdot)f\bigr)=\Tr_xS(\cdot)\widetilde{Q}_Nf
\end{equation}
in $L^2_{loc}$.
\end{lemma}

\begin{proof}
By Lemma \ref{sing}, there exist  $a,b>0$ such that 
\begin{align*}
  \|\Tr_{x}S(\cdot)f\|^2_{L^2(J)}\le (a|J|+b)\|f\|^2_{2}, \qquad \forall x\in \mathbb R. 
\end{align*}
Then Lemma \ref{loc} below shows 
\begin{align} 
&\|\widetilde{\mathcal{P}}_{N}\Tr_x S(\cdot)f\|^2_{L^2(J)}+\|\widetilde{\mathcal{Q}}_{N}\Tr_x S(\cdot)f\|^2_{L^2(J)}\nonumber\\
&\le C\left[(a|J|+b)+C(a+b)|J|\right]\|f\|^2_2, \qquad \forall x\in \mathbb R,\label{6bn}
\end{align}
for any bounded interval $J\subset \mathbb R$. 

For Schwartz $f$, the trace $\Tr_x S(\cdot )f$  is the classical point value. Since the spacetime Fourier transform of  $u:=S(t)f$ is supported in the curve $\{(\tau,\xi)\in \mathbb R^2: \tau=\xi^3\}$, applying the time multiplier $\chi(D_t/N^3)$ is the same as applying the spatial multiplier $\chi(D^3/N^3)$. Hence (\ref{eqtifilter}) holds in $L^2_{loc}$ by (\ref{6bn}) for every $x\in \mathbb  R$. 

For general $L^2$ initial data, we use the density argument. Let $f_n\to f$ in $L^2$, then $\widetilde{Q}_Nf_n\to {\widetilde Q}_Nf$ and ${\widetilde P}_Nf_n\to {\widetilde P}_Nf$ in $L^2$, and 
\begin{equation*}
   \widetilde{\mathcal{P}}_{N}\bigl(\Tr_xS(\cdot)f_n\bigr)=\Tr_xS(\cdot)\widetilde{P}_Nf_n,
    \qquad
   \widetilde{\mathcal{Q}}_{N}\bigl(\Tr_xS(\cdot)f_n\bigr)=\Tr_xS(\cdot)\widetilde{Q}_Nf_n, \qquad \forall n\in \mathbb N. 
\end{equation*}
By  (\ref{6bn}), letting $n\to \infty$, we obtain  that (\ref{eqtifilter}) holds in $L^2_{loc}$ for general $f\in L^2$.

\end{proof}


Let
\begin{align*}
    I_T=[T/3,2T/3].
\end{align*}
The following tail estimate is used for  the low and high frequency separation.

\begin{lemma}\label{lem:tail}
Let $\mathcal M_N$ be either $\widetilde{\mathcal Q}_N$ or $\widetilde{\mathcal P}_N$.  For every $T>0$ there are constants $C,c>0$, depending on $T,R$ and the cutoff, such that
\begin{equation}\label{eq:tail-est}
\begin{split}
 &\left\|\mathcal M_N\bigl(\1_{\R\setminus(0,T)}\Tr_xS(\cdot)f\bigr)
     \right\|_{\YY_E(I_T)}  
 \le C e^{-c N^{3/\sigma}}\|f\|_2,
    \qquad\forall \mbox{ } N\ge1 ,
\end{split}
\end{equation}
where  the norm $\YY_{E}(I_{T})$ is defined in \eqref{eq:blocknorm}.
\end{lemma}

\begin{proof}
For $\widetilde{\mathcal P}_N$, the convolution kernel is $K_N(t)=N^3K(N^3t)$, where $K$ satisfies \eqref{gekernl}.  For $\widetilde{\mathcal Q}_N=I-\widetilde{\mathcal P}_N$, the identity part gives no contribution on $I_T$, since ${\bf 1}_{\mathbb R\setminus (0,T)}$ vanishes in $I_{T}$. Hence, the same estimate with kernel $-K_N$ applies to $\widetilde{\mathcal{Q}}_N$. 

We prove the estimate for the kernel $K_{N}$.  Decompose $\R\setminus(0,T)$ into intervals of length one,
\begin{align*}
    J_m^+=[T+m,T+m+1],\qquad J_m^-=[-m-1,-m],\qquad m=0,1,2,\dots .
\end{align*}
For $t\in I_T$ and $s\in J_m^\pm$ one has $|t-s|\ge c_T(1+m)$  with $c_T>0$.  Therefore, we see 
\begin{align*}
    \sup_{t\in I_T,\,s\in J_m^\pm}|K_N(t-s)|
    \le C_T N^3\exp\{-c_T N^{3/\sigma}(1+m)^{1/\sigma}\}
    =: \alpha_{N,m}.
\end{align*}
Let $y_x(t)=\Tr_xS(t)f$.  For each block $E_j$, using Cauchy  inequality in $s$ and the fact that $|J_m^\pm|=1$  one obtains 
\begin{align*}
\begin{split}
&\sup_{x\in E_j}
\left\|\int_{J_m^\pm}K_N(\cdot-s)y_x(s)\dd s\right\|_{L^2(I_T)} \\
&\qquad\le |I_T|^{1/2}\alpha_{N,m}
   \sup_{x\in E_j}\|y_x\|_{L^2(J_m^\pm)} .
\end{split}
\end{align*}
Taking the $\ell^2$ norm in $j$ yields
\begin{align*}
\begin{split}
&\left(\sum_j\sup_{x\in E_j}
\left\|\int_{J_m^\pm}K_N(\cdot-s)y_x(s)\dd s\right\|_{L^2(I_T)}^2\right)^{1/2} \\
&\qquad\le C_T\alpha_{N,m}
    \|S(\cdot)f\|_{\YY_E(J_m^\pm)} .
\end{split}
\end{align*}
Minkowski  inequality then gives
\begin{align*}
\begin{split}
&\left\|\mathcal M_N(\1_{\R\setminus(0,T)}\Tr_xS(\cdot)f)\right\|_{\YY_E(I_T)} \\
&\qquad\le C_T\sum_{m\ge0}\alpha_{N,m}
   \bigl(\|S(\cdot)f\|_{\YY_E(J_m^+)}+
         \|S(\cdot)f\|_{\YY_E(J_m^-)}\bigr).
\end{split}
\end{align*}
By Lemma~\ref{balg}, each block norm on the intervals $J_m^\pm$ is bounded by $A_{1,R}^{1/2}\|f\|_2$.  Hence
\begin{align*}
    \left\|\mathcal M_N(\1_{\R\setminus(0,T)}\Tr_xS(\cdot)f)\right\|_{\YY_E(I_T)}
    \le C_{T,R}\left(\sum_{m\ge0}N^3e^{-c_{T}N^{3/\sigma}(1+m)^{1/\sigma}}\right)\|f\|_2.
\end{align*}
The sum is bounded by $C_{T,R}e^{-c_T'N^{3/\sigma}}$ after changing constants, because the exponential term dominates the polynomial factor $N^3$.  This proves \eqref{eq:tail-est}.
\end{proof}

The following lemma serves as a supplement to the proof of Lemma \ref{lefc}. 
\begin{lemma}
\label{loc} 
Let $\chi\in C_c^\infty(\mathbb R)$, and define
\begin{align*}
   \widetilde{\mathcal{P}}_{N}=\chi(D_t/N^3),\qquad N\ge1.
\end{align*}
Suppose that $f\in L^2_{\mathrm{loc}}(\mathbb R)$ satisfies
\begin{align}\label{HmL}
    \int_J |f(t)|^2\,dt\le a|J|+b
\end{align}
for every bounded interval $J\subset\mathbb R$, where $a,b\ge0$. Then
there exists a constant $C_N>0$, depending only on $N$ and
$\chi$, such that
\begin{align}\label{R544}
    \|{\widetilde {\mathcal P}}_N f\|_{L^\infty(\mathbb R)}
    \le C_N(a+b)^{1/2}.
\end{align}
And for any bounded interval $J\subset \mathbb R$ there exists $C_{N}>0$ such that 
\begin{align}
    \|{\widetilde{\mathcal Q}}_N f\|_{L^2(J)}
    \le C_{N}|J|^{\frac{1}{2}}(a+b)^{1/2}+(a|J|+b)^{\frac{1}{2}}.\label{R59966}
\end{align}
\end{lemma}

\begin{proof}
Let
\begin{align*}
    K=\check\chi,
    \qquad
    K_N(t)=N^3K(N^3t).
\end{align*}
Then $K_N\in\mathcal S(\mathbb R)$ and
\begin{align*}
   \widetilde{\mathcal{P}}_{N} f=K_N*f.
\end{align*}
For $m\in\mathbb Z$, set
\begin{align*}
    I_m=[m,m+1].
\end{align*}
By the assumption (\ref{HmL}) applied to $I_m$, we have
\begin{align*}
    \|f\|_{L^2(I_m)}\le (a+b)^{1/2}.
\end{align*}
Therefore, for every $t\in\mathbb R$,
\begin{align*}
\begin{aligned}
    |\widetilde{\mathcal P}_N f(t)|
    &=
    \left|\int_{\mathbb R}K_N(t-s)f(s)\,ds\right|  \\
    &\le
    \sum_{m\in\mathbb Z}\int_{I_m}|K_N(t-s)||f(s)|\,ds  \\
    &\le
    \sum_{m\in\mathbb Z}
    \|K_N(t-\cdot)\|_{L^2(I_m)}
    \|f\|_{L^2(I_m)}  \\
    &\le
    (a+b)^{1/2}
    \sum_{m\in\mathbb Z}
    \|K_N(t-\cdot)\|_{L^2(I_m)}.
\end{aligned}
\end{align*}
Since $K_N\in\mathcal S(\mathbb R)$, the quantity
\begin{align*}
    C_N:=
    \sup_{t\in\mathbb R}
    \sum_{m\in\mathbb Z}
    \|K_N(t-\cdot)\|_{L^2(I_m)}
\end{align*}
is finite. Hence
\begin{align*}
    |\widetilde{\mathcal P}_N f(t)|
    \le
    C_N(a+b)^{1/2}.
\end{align*}
Taking the supremum over $t\in\mathbb R$, we obtain
\begin{align*}
    \|\widetilde{\mathcal P}_N f\|_{L^\infty(\mathbb R)}
    \le
    C_N(a+b)^{1/2}.
\end{align*}
Therefore, (\ref{R544}) follows and $\widetilde{\mathcal P}_N f\in L^2_{\mathrm{loc}}(\mathbb R)$. (\ref{R59966}) follows by $\widetilde{\mathcal{Q}}_N=I-\widetilde{\mathcal{P}}_{N}$ and (\ref{HmL}). 
\end{proof}

\section{High-frequency crossing and separation}
\label{high-cross}

Recall the definitions of $\widetilde{P}_N$ and  $\widetilde{Q}_N$ in (\ref{eq:QNPN}). 

\begin{remark}
In all the proofs  of Section 5 to Section 7, we omit the dependence of various constants $C,c$ on $E,R,m$, which are fixed from the beginning, and just keep the dependence of $c,c$ on $T>0$. 
 \end{remark}

\subsection{The high-frequency Hautus estimate}

Recall that $D=-i\partial_x $. 
\begin{lemma}\label{lem:HF-Hautus}
Let $\Lambda$ be $(\rho,L)$-regular in Definition \ref{regular}.  There are constants $C>0$ and $N_0\ge1$ depending only on $\rho,L$ such that, for all $N\ge N_0$, all $f\in H^3(\R)$ with 
$$ \supp\widehat{f}\subset \{\xi\in \mathbb R: |\xi|\ge N\},$$
and all $\lambda\in\R$, there holds 
\begin{equation}\label{eq:HF-Hautus}
    \|f\|_2^2\le C\|C_\Lambda f\|_{\ell^2}^2
       +\frac{C}{N^2}\|(D^3-\lambda)f\|_2^2.
\end{equation}
\end{lemma}

\begin{proof}
Let $\lambda=r^3$. Define the frequency  cutoff $\Pi_r$   by
\begin{align*}
    \widehat{\Pi_rf}(\xi)=\mathbf 1_{\{|\xi-r|\le a\}}\widehat f(\xi),
\end{align*}
where $a$ is chosen as in Lemma \ref{nasa}. And set $\Phi_r=I-\Pi_r$.  Put $g=(D^3-\lambda)f$.

If $|r|\le N/2$, then $|\xi|\ge N$ on $\supp\widehat f$ implies $|\xi^3-r^3|\gtrsim N^3$, hence $\|f\|_2\lesssim N^{-3}\|g\|_2$, which is stronger than \eqref{eq:HF-Hautus}.

Assume $|r|>N/2$.  On the support of $\Phi_rf$,
\begin{align*}
    |\xi-r|>a,
    \qquad
    \xi^3-r^3=(\xi-r)(\xi^2+\xi r+r^2),
\end{align*}
and
\begin{align*}
    \xi^2+\xi r+r^2\ge \frac{1}{2}(\xi^2+r^2).
\end{align*}
Since $|\xi|\ge N$ on $\supp\widehat f$ and $|r|>N/2$, this gives
\begin{align}\label{Huytgm}
    \|\Phi_rf\|_{H^1}\le \frac{C}{N}\|g\|_2.
\end{align}
By Lemma \ref{samu}, one has 
\begin{align*}
    \|C_\Lambda \Phi_rf\|_{\ell^2}\le \frac{C}{N}\|g\|_2.
\end{align*}
By Lemma \ref{nasa},
\begin{align*}
    \|\Pi_rf\|_2\le C\|C_\Lambda \Pi_rf\|_{\ell^2}
       \le C\|C_\Lambda f\|_{\ell^2}+\frac{C}{N}\|g\|_2,
\end{align*}
which  together with the estimate (\ref{Huytgm}) on $\Phi_rf$  proves \eqref{eq:HF-Hautus}.
\end{proof}

As a corollary of Lemma \ref{lem:HF-Hautus}, we have the high frequency observability on arbitrary short intervals. 
\begin{corollary}\label{coho}
Let $\Lambda$ be $(\rho,L)$-regular in Definition \ref{regular}.  For every $\tau>0$ there exist  $N_\tau\ge1$ and $K_\tau>0$ such that, for every $a\in\R$ and every $g\in L^2(\R)$ with $\supp\widehat g\subset\{|\xi|\ge N_\tau\}$, there holds 
\begin{equation}\label{eq:HF-obs}
    \|g\|_2^2\le K_\tau\int_a^{a+\tau}
       \sum_j |(\Tr_{\Lambda}S(t)g)_j|^2\dd t .
\end{equation}
Here $N_{\tau},K_{\tau}$ can be chosen as 
$$N_{\tau}=\frac{C_1}{\tau}, \qquad K_{\tau}=\frac{C_2}{\tau}$$
for some constants $C_1,C_2>0$ depending only on $(\rho,L)$. 
\end{corollary}

\begin{proof}
Apply Proposition \ref{prop:Miller} to the high frequency subspace.  The admissibility constants are inherited from Proposition~\ref{psa}. The observability Hautus constant in \eqref{eq:hautus} can be chosen as $M =C N^{-2}$.  We fix $N_\tau$ so large that $\pi\sqrt{M}<\tau/2$ for all $N\ge N_\tau$. Then the desired result follows by  Proposition \ref{prop:Miller}. 
\end{proof}

\subsection{High frequency separation from the full observation}

The following lemma  separates the high frequency part  from the full observation. 

\begin{lemma}\label{lehfs}
For every $T>0$ and every $\eps>0$ there exist $N_1=N_1(T,\eps,R)$ and $C=C(T,R)$  such that for all $N\ge N_1$ and all $f\in L^2(\R)$,
\begin{equation}\label{eq:HF-sep}
    \|\widetilde{Q}_Nf\|_2^2
    \le C \|S(\cdot)f\|_{\YY_E(0,T)}^2
       +\eps\|\widetilde{P}_Nf\|_2^2 .
\end{equation}
\end{lemma}

\begin{proof}
Fix a selection $x_j\in E_j$.  Then $\{x_j\}$ is $(2R,4R)$-regular.  Recall that $I_{T}=[\frac{1}{3}T,\frac{2}{3}T]$. For $N$ large enough, $\widetilde{Q}_Nf$ is supported in high frequencies and Corollary \ref{coho} on $I_T$ gives
\begin{align}\label{Kiu8G}
    \|\widetilde{Q}_Nf\|_2^2
    \le K_T\int_{I_T}\sum_j|\Tr_{x_j}S(t)\widetilde{Q}_Nf|^2\dd t.
\end{align}
By \eqref{eqtifilter},
\begin{align*}
    \Tr_{x_j}S(\cdot)\widetilde{Q}_Nf=\widetilde{\mathcal Q}_N(\Tr_{x_j}S(\cdot)f).
\end{align*}
Split $\Tr_{x_j}S(\cdot)f$ into the restriction to $(0,T)$ part and the exterior part, i.e.,
\begin{align*}
\widetilde{\mathcal Q}_N(\Tr_{x_j}S(\cdot)f)&= \widetilde{\mathcal Q}_N({\bf 1}_{(0,T)}\Tr_{x_j}S(\cdot)f)+ \widetilde{\mathcal Q}_N({\bf 1}_{\mathbb R\setminus (0,T)}\Tr_{x_j}S(\cdot)f)\\
&:={\bf I}_1(x_j)+{\bf I}_2(x_j).
\end{align*}
By the $L^2(\mathbb R)\to L^2(\mathbb R)$ boundedness of $\widetilde{\mathcal Q}_N$ and  $x_j\in E_j$, the inner part ${\bf I}_1$ is dominated by  
\begin{align*}
\sum_{j\in\mathbb Z} \|{\bf I}_1(x_j)\|^2_{L^2(I_{T})}\le  \sum_{j\in \mathbb Z} \|\Tr_{x_j}S(\cdot)f)\|^2_{L^2(0,T)}\le  \|S(\cdot)f\|^2_{\YY_{E}(0,T)}. 
\end{align*}
The exterior part is bounded by Lemma \ref{lem:tail} as 
\begin{align*}
\sum_{j\in\mathbb Z} \|{\bf I}_2(x_j)\|^2_{L^2(I_T)}\le  e^{-c_{T}N^{\frac{3}{\sigma}}}\|f\|^2_{2} 
\end{align*}
for some $c_T>0$. 
Thus (\ref{Kiu8G}) yields
\begin{align*}
    \|\widetilde{Q}_Nf\|^2_2
    \le C_T\|S(\cdot)f\|^2_{\YY_E(0,T)}+
      C_T \delta_N\|f\|^2_2,
\end{align*}
where $\delta_N\to0$ as $N\to\infty$.  Since $f=\widetilde{P}_Nf+\widetilde{Q}_Nf$,
\begin{align*}
    \|{\widetilde{Q}}_Nf\|^2_2
    \le C_T\|S(\cdot)f\|^2_{\YY_E(0,T)}+
      C_{T}\delta_N\|{\widetilde{Q}}_Nf\|^2_2+
      C_{T} \delta_N\|\widetilde{P}_Nf\|^2_2.
\end{align*}
For $N$ sufficiently large, absorbing the term $C_T\delta_N\|\widetilde{Q}_Nf\|_2$, we obtain  \eqref{eq:HF-sep}.
\end{proof}

\section{Low-frequency Hausdorff separation}\label{separation}

The following lemma  separates the low frequency part from the full observation. 
\begin{lemma}\label{lesepl} 
For every $T>0$ there exist $N_2=N_2(T,R,m,s)$ and $C_{T,N}>0$ such that, for all $N\ge N_2$ and all $f\in L^2(\R)$,
\begin{equation}\label{eq:LF-sep}
    \|\widetilde{P}_Nf\|_2^2
    \le C_{T,N}\|S(\cdot)f\|_{\YY_E(0,T)}^2
       +\frac1{16}\|f\|_2^2 .
\end{equation}
\end{lemma}

\begin{proof}
Set $g={\widetilde P}_Nf$ and  
$$
F_g(x):=\Tr_x S(\cdot)g\big|_{I_T}\in H:=L^2(I_T).
$$ 
Then $F_g$ is $H$-valued and its Fourier transform in variable $x$  is supported  in $[-2N,2N]$.
Then Lemma \ref{lemHH} gives
\begin{align*}
    \int_\R\|F_g(x)\|_{L^2(I_T)}^2\dd x
       \le C \exp\{C(1+N){}\}
       \sum_j\sup_{x\in E_j}\|F_g(x)\|_{L^2(I_T)}^2.
\end{align*}
The left hand side equals
\begin{align*}
    \int_{I_T}\|S(t)g\|_2^2\dd t=|I_T|\|g\|_2^2=\frac{T}{3}\|\widetilde{P}_Nf\|_2^2.
\end{align*}
Therefore one has 
\begin{equation}\label{eqFprF}
    \|\widetilde{P}_Nf\|_2^2
       \le A_N\|S(\cdot)\widetilde{P}_Nf\|_{\YY_E(I_T)}^2,
    \qquad
    A_N\le C_T\exp\{C(1+N){}\}.
\end{equation}
By \eqref{eqtifilter},
\begin{align*}
    \Tr_xS(\cdot)\widetilde{P}_Nf=  \widetilde{\mathcal P}_N(\Tr_xS(\cdot)f).
\end{align*}
Divide  $\widetilde{\mathcal P}_N(\Tr_xS(\cdot)f)$ into 
$$\widetilde{\mathcal P}_N(\Tr_xS(\cdot)f)=\widetilde{\mathcal P}_N({\bf 1}_{(0,T)}\Tr_x S(\cdot)f)
+\widetilde{\mathcal P}_N({\bf 1}_{\mathbb R\setminus (0,T)}\Tr_x S(\cdot)f).
$$
By the same arguments as the proof of Lemma \ref{lehfs},  we obtain  from the $L^2$ boundedness of $\widetilde{\mathcal P}_N$ and Lemma \ref{lem:tail} that 
\begin{align}
    \|S(\cdot){\widetilde{P}}_Nf\|_{\YY_E(I_T)}^2
       \le C_T\|S(\cdot)f\|_{\YY_E(0,T)}^2
          +C_T e^{-c_TN^{3/\sigma}}\|f\|_2^2,\label{Yugvrt}
\end{align}
Combining (\ref{Yugvrt}) with \eqref{eqFprF} gives 
\begin{align*}
    \|\widetilde{P}_Nf\|_2^2
       &\le C_T\exp\{C(1+N){}\}\|S(\cdot)f\|_{\YY_E(0,T)}^2 \\
       &\quad +C_T \exp\{C(1+N){}\}e^{-c_TN^{3/\sigma}}\|f\|_2^2 .
\end{align*}
Since $1<\sigma<3$, we have $N^{3/\sigma}\gg N$.  Taking $N\ge N_2$ to be sufficiently large, we obtain  \eqref{eq:LF-sep}.
\end{proof}

\section{Proof of Theorem \ref{th1}}\label{proof-th1}

In this section, we prove Theorem \ref{th1}. For reader's convenience, we restate Theorem \ref{th1} as the following proposition. 
\begin{proposition}\label{proline}
Let $0<s\le1$, $R>0$, $m>0$, and let $E\subset\R$ satisfy
\begin{align*}
    \mathcal H^s_\infty(E\cap[3Rj,3Rj+R])\ge m,
    \qquad \forall \mbox{ }j\in\Z.
\end{align*}
Then for every $T>0$ there exist constants $C,C'>0$, depending on $T,R,m,s$, such that
\begin{equation}\label{qeline1}
    \|u_0\|_{L^2(\R)}^2
    \le C
    \sum_{j\in\Z}\sup_{x\in E_j}
       \|\Tr_xS(\cdot)u_0\|_{L^2(0,T)}^2,
    \qquad \forall \mbox{ }u_0\in L^2(\R),
\end{equation}
and
\begin{equation}\label{malineup2}
    \sum_{j\in\Z}\sup_{x\in E_j}
       \|\Tr_xS(\cdot)u_0\|_{L^2(0,T)}^2
    \le C'\|u_0\|_{L^2(\R)}^2  \forall \mbox{ }u_0\in L^2(\R).
\end{equation}
For $u_0\in H^3(\R)$ the traces are classical point value, and  particularly 
\begin{align*}
    \|u_0\|_2^2
    \le C \sum_j\sup_{x\in E_j}\int_0^T |S(t)u_0(x)|^2\dd t .
\end{align*}
\end{proposition}

\begin{proof}
The upper bound \eqref{malineup2} is Lemma \ref{balg}.  We prove the lower bound.  Fix $\eps=1/16$ in Lemma \ref{lehfs}.  Choose $N$ large enough so that both Lemma \ref{lehfs} and Lemma \ref{lesepl} hold.  Write $f=u_0$ and $f=\widetilde {P}_Nf+\widetilde{Q}_Nf$.  Then
\begin{align*}
    \|f\|_2^2\le 2\|\widetilde{P}_Nf\|_2^2+2\|\widetilde{Q}_Nf\|_2^2.
\end{align*}
By the  high frequency separation proved in Lemma \ref{lehfs}, one obtains 
\begin{align*}
    \|{\widetilde Q}_Nf\|_2^2
    \le C_T\|S(\cdot)f\|_{\YY_E(0,T)}^2+\frac1{16}\|{\widetilde P}_Nf\|_2^2.
\end{align*}
By the low frequency separation given by Lemma \ref{lesepl}, we obtain  
\begin{align*}
    \|\widetilde{P}_Nf\|_2^2
    \le C_{T,N}\|S(\cdot)f\|_{\YY_E(0,T)}^2+\frac1{16}\|f\|_2^2.
\end{align*}
Substituting gives
\begin{align*}
    \|f\|_2^2\le C_{T}\|S(\cdot)f\|_{\YY_E(0,T)}^2+\frac14\|f\|_2^2.
\end{align*}
Absorbing the last term proves \eqref{qeline1}.
\end{proof}

\section{Consequences and related results on the line}\label{related}

In this section, we extend the result of Theorem \ref{th1} to observations on periodic Hausdorff sets  and prove Proposition \ref{p1}.

\subsection{Preliminaries on  Haraux--Ingham type theorem}

In the following proposition, we recall Haraux--Ingham theorem, see Theorem 4 in \cite{Haraux1989}  and also \cite{Ingham1936}.
\begin{proposition}[Haraux--Ingham theorem]
\label{yugg}
Let $J\subset\mathbb R$ be a bounded interval. Let
$(\mu_k)_{k\in\mathbb Z}$ be a strictly increasing sequence of real numbers.
Assume that there exist constants $\gamma>0$, $\rho>0$, and an integer
$N\ge0$ such that
\begin{align*}
    \mu_{k+1}-\mu_k\ge \gamma,
    \qquad |k|\ge N,
\end{align*}
\begin{align*}
    |J|>\frac{2\pi}{\gamma},
\end{align*}
and
\begin{align*}
    \mu_{k+1}-\mu_k\ge \rho,
    \qquad \forall  k\in\mathbb Z.
\end{align*}
Then there exist constants $c,C>0$, depending only on
$\gamma,\rho,N$, and $|J|$, such that for every finitely supported sequence
$(b_k)_{k\in\mathbb Z}$,
\begin{align*}
    c\sum_{k\in\mathbb Z}|b_k|^2
    \le
    \int_J
    \left|
        \sum_{k\in\mathbb Z} b_k e^{i\mu_k t}
    \right|^2dt
    \le
    C\sum_{k\in\mathbb Z}|b_k|^2 .
\end{align*}
\end{proposition}

Proposition \ref{yugg} immediately gives the following 
uniform cubic Ingham estimate. 
\begin{lemma}
\label{leuing}
Let $\alpha>0$ and  $T>0$.  There exist constants
$0<A\le B<\infty$ such that, for every $\theta\in[0,\alpha]$ and every
finitely supported sequence $(a_k)_{k\in\mathbb Z}$,
\begin{align*}
    A\sum_{k\in\mathbb Z}|a_k|^2
    \le
    \int_0^T
    \left|
        \sum_{k\in\mathbb Z}a_k e^{it(\theta+\alpha k)^3}
    \right|^2dt
    \le
    B\sum_{k\in\mathbb Z}|a_k|^2 .
\end{align*}
\end{lemma}

\begin{proof}
We apply  Proposition \ref{yugg} to
\begin{align*}
    J=[0,T],
    \qquad
    \mu_k=\omega_k(\theta):=(\theta+\alpha k)^3.
\end{align*}
For each fixed $\theta\in[0,\alpha]$, the sequence
$(\omega_k(\theta))_{k\in\mathbb Z}$ is strictly increasing because
$x\mapsto x^3$ is strictly increasing.

We now verify the gap assumptions uniformly with respect to $\theta$.
Write
\begin{align*}
    r=\frac{\theta}{\alpha}\in[0,1].
\end{align*}
Then
\begin{align}\label{Rtoomm}
    \omega_k(\theta)=\alpha^3(k+r)^3.
\end{align}
Therefore 
\begin{align*}
    \omega_{k+1}(\theta)-\omega_k(\theta)
    &=
    \alpha^3\bigl(3(k+r)^2+3(k+r)+1\bigr)
     =\alpha^3\left(3\left(s+\frac12\right)^2+\frac14\right)
    \ge \frac{\alpha^3}{4}.
\end{align*}
Hence we obtain the uniform global separation
\begin{align*}
    \omega_{k+1}(\theta)-\omega_k(\theta)
    \ge
    \frac{\alpha^3}{4},
    \qquad k\in\mathbb Z,\quad \theta\in[0,\alpha].
\end{align*}
Thus one may take
$\rho=\frac{\alpha^3}{4}.$
Next,  
take $\gamma=\frac{4\pi}{T}$, then 
$\gamma>\frac{2\pi}{T}.$
Because
\begin{align*}
    3(k+r)^2+3(k+r)+1\to+\infty
    \qquad\text{as } |k|\to\infty,
\end{align*}
holds uniformly for $r\in[0,1]$, there exists an integer $N=N(T,\alpha)$ such that
\begin{align*}
    \alpha^3\bigl(3(k+r)^2+3(k+r)+1\bigr)\ge \gamma
\end{align*}
for all $|k|\ge N$ and all $r\in[0,1]$. Equivalently, there holds 
\begin{align*}
    \omega_{k+1}(\theta)-\omega_k(\theta)\ge \gamma,
    \qquad |k|\ge N,\quad \theta\in[0,\alpha].
\end{align*}
Moreover,  
$ |J|=T>\frac{2\pi}{\gamma},$
by the choice of $\gamma$.

Hence, for every $\theta\in[0,\alpha]$, the sequence
$(\omega_k(\theta))_{k\in\mathbb Z}$ satisfies the hypotheses of the
Haraux--Ingham theorem with the same constants $\gamma,\rho,N$, independent of
$\theta$. Therefore there exist constants $A,B>0$, depending only on
$T$ and $\alpha$, such that
\begin{align*}
    A\sum_{k\in\mathbb Z}|a_k|^2
    \le
    \int_0^T
    \left|
        \sum_{k\in\mathbb Z}a_k e^{it\omega_k(\theta)}
    \right|^2dt
    \le
    B\sum_{k\in\mathbb Z}|a_k|^2
\end{align*}
for all $\theta\in[0,\alpha]$ and all finitely supported sequences
$(a_k)$. This is exactly the desired estimate by (\ref{Rtoomm}). 
\end{proof}

\subsection{Periodic Hausdorff sets}

Theorem \ref{th1} applies to every Hausdorff-thick set.  In periodic situations, one can also obtain a Hilbertian Hausdorff measure observation.

Let $h>0$ and let $\mathcal C\subset[0,h)$ be a Borel set. Assume that  $\nu$ is a finite
positive Borel measure supported on $\mathcal C$ with
\begin{align*}
    0<\nu(\mathcal C)<\infty.
\end{align*}
Set
\begin{align*}
    E=\mathcal C+h\mathbb Z,
    \qquad
    \mu_E=\sum_{n\in\mathbb Z}(\tau_{nh})_\#\nu,
\end{align*}
where $\tau_{nh}(y)=y+nh$. 

We restate Proposition \ref{p1} as follows for reader's convenience. 
\begin{proposition} 
\label{ziran}
 For every $T>0$, there exist constants
$0<c_{T,h,\nu}\le C_{T,h,\nu}<\infty$ such that
\begin{align*}
    c_{T,h,\nu}\|u_0\|_{L^2(\mathbb R)}^2
    \le
    \int_0^T\int_E |S(t)u_0(x)|^2\,d\mu_E(x)\,dt
    \le
    C_{T,h,\nu}\|u_0\|_{L^2(\mathbb R)}^2.
\end{align*}
For general $u_0\in L^2(\mathbb R)$, the observation in the middle is
understood as the $L^2$-extension from Schwartz data.
\end{proposition}

\begin{proof}
Set
\begin{align*}
    \alpha=\frac{2\pi}{h}.
\end{align*}
We first prove the result for  
$\widehat u_0\in C_c^\infty(\mathbb R)$. For such data all the sums below are
finite for each fixed $\theta\in[0,\alpha)$, and all applications of Fubini and
Parseval are justified.

For $x=y+nh$  with $y\in \mathcal C$ and $n\in\mathbb Z$, the Airy evolution is
\begin{align*}
    S(t)u_0(y+nh)
    =
    \frac1{\sqrt{2\pi}}
    \int_{\mathbb R}
    e^{i(y+nh)\xi}
    e^{it\xi^3}
    \widehat u_0(\xi)\,d\xi .
\end{align*}
Decompose
\begin{align*}
    \mathbb R
    =
    \bigcup_{k\in\mathbb Z}[0,\alpha)+\alpha k
\end{align*}
and write
\begin{align*}
    \xi=\theta+\alpha k,
    \qquad
    \theta\in[0,\alpha),
    \quad k\in\mathbb Z.
\end{align*}
Then
\begin{align*}
\begin{aligned}
    S(t)u_0(y+nh)
    &=
    \frac1{\sqrt{2\pi}}
    \int_0^\alpha
    \sum_{k\in\mathbb Z}
    e^{i(y+nh)(\theta+\alpha k)}
    e^{it(\theta+\alpha k)^3}
    \widehat u_0(\theta+\alpha k)
    \,d\theta .
\end{aligned}
\end{align*}
Since $h\alpha=2\pi$, we have 
\begin{align*}
\begin{aligned}
    S(t)u_0(y+nh)
    &=
    \frac1{\sqrt{2\pi}}
    \int_0^\alpha
    e^{inh\theta}
    \sum_{k\in\mathbb Z}
    \widehat u_0(\theta+\alpha k)
    e^{iy(\theta+\alpha k)}
    e^{it(\theta+\alpha k)^3}
    \,d\theta .
\end{aligned}
\end{align*}
For fixed $t\in[0,T]$ and $y\in \mathcal C$, define
\begin{align*}
    H_{t,y}(\theta)
    :=
    \sum_{k\in\mathbb Z}
    \widehat u_0(\theta+\alpha k)
    e^{iy(\theta+\alpha k)}
    e^{it(\theta+\alpha k)^3}.
\end{align*}
Then
\begin{align*}
    S(t)u_0(y+nh)
    =
    \frac1{\sqrt{2\pi}}
    \int_0^\alpha e^{inh\theta}H_{t,y}(\theta)\,d\theta .
\end{align*}
By Parseval's identity for the orthogonal system
\begin{align*}
    \{e^{inh\theta}\}_{n\in\mathbb Z}
\end{align*}
on $[0,\alpha]$, we obtain
\begin{align*}
\begin{aligned}
    \sum_{n\in\mathbb Z}|S(t)u_0(y+nh)|^2
    &=
    \frac1{2\pi}
    \sum_{n\in\mathbb Z}
    \left|
    \int_0^\alpha e^{inh\theta}H_{t,y}(\theta)\,d\theta
    \right|^2  \\
    &=
    \frac{\alpha}{2\pi}
    \int_0^\alpha |H_{t,y}(\theta)|^2\,d\theta \\
    &=\frac1h
    \int_0^\alpha |H_{t,y}(\theta)|^2\,d\theta .
\end{aligned}
\end{align*}
Using the definition of $\mu_E$, for non-negative measurable functions $F$, one has 
\begin{align*}
    \int_E F(x)\,d\mu_E(x)
    =
    \sum_{n\in\mathbb Z}\int_{\mathcal C } F(y+nh)\,d\nu(y).
\end{align*}
Hence
\begin{align*}
\begin{aligned}
    \int_0^T\int_E |S(t)u_0(x)|^2\,d\mu_E(x)\,dt
    &=
    \int_0^T
    \int_{\mathcal C} 
    \sum_{n\in\mathbb Z}|S(t)u_0(y+nh)|^2
    \,d\nu(y)\,dt  \\
    &=
    \frac1h
    \int_{\mathcal C} 
    \int_0^T
    \int_0^\alpha
    |H_{t,y}(\theta)|^2
    \,d\theta\,dt\,d\nu(y).
\end{aligned}
\end{align*}
By Fubini's theorem,
\begin{align*}
   & \int_0^T\int_E |S(t)u_0(x)|^2\,d\mu_E(x)\,dt
    \\
    &=\frac1h
    \int_{\mathcal C}\int_0^\alpha \int_0^T
    \left| \sum_{k\in\mathbb Z}
    \widehat u_0(\theta+\alpha k)
    e^{iy(\theta+\alpha k)}
    e^{it(\theta+\alpha k)^3}
    \right|^2\,dt\,d\theta\,d\nu(y).
\end{align*}

We now apply Lemma~\ref{leuing}. For fixed
$\theta\in[0,\alpha]$ and $y\in \mathcal C$, define
\begin{align*}
    a_k(\theta,y)
    :=
    \widehat u_0(\theta+\alpha k)e^{iy(\theta+\alpha k)}.
\end{align*}
Note that 
\begin{align*}
    |a_k(\theta,y)|=|\widehat u_0(\theta+\alpha k)|.
\end{align*}
By Lemma \ref{leuing}, there exist  constants $0<A\le B<\infty$,
depending only on $T$ and $\alpha$, such that
\begin{align}\label{PoLm2}
    A\sum_{k\in\mathbb Z}|\widehat u_0(\theta+\alpha k)|^2
    \le
    \int_0^T
    \left|
    \sum_{k\in\mathbb Z}
    \widehat u_0(\theta+\alpha k)
    e^{iy(\theta+\alpha k)}
    e^{it(\theta+\alpha k)^3}
    \right|^2dt
\end{align}
and
\begin{align}\label{PoLm4}
    \int_0^T
    \left|
    \sum_{k\in\mathbb Z}
    \widehat u_0(\theta+\alpha k)
    e^{iy(\theta+\alpha k)}
    e^{it(\theta+\alpha k)^3}
    \right|^2dt
    \le
    B\sum_{k\in\mathbb Z}|\widehat u_0(\theta+\alpha k)|^2 .
\end{align}
The constants are uniform in both $\theta\in[0,\alpha]$ and $y\in \mathcal C$.

Integrating (\ref{PoLm2})  and (\ref{PoLm4}) in $\theta\in[0,\alpha]$ and
$y\in\mathcal  C$, we obtain 
\begin{align*}
\begin{aligned}
    \frac{A}{h}
    \int_{\mathcal C}\int_0^\alpha
    \sum_{k\in\mathbb Z}
    |\widehat u_0(\theta+\alpha k)|^2
    d\theta\,d\nu(y)
    &\le
    \int_0^T\int_E |S(t)u_0(x)|^2\,d\mu_E(x)\,dt \\
    &\le
    \frac{B}{h}
    \int_{\mathcal C}\int_0^\alpha
    \sum_{k\in\mathbb Z}
    |\widehat u_0(\theta+\alpha k)|^2
    d\theta\,d\nu(y).
\end{aligned}
\end{align*}
Since the integrand is independent of $y$,
\begin{align*}
    \int_{\mathcal C}\int_0^\alpha
    \sum_{k\in\mathbb Z}
    |\widehat u_0(\theta+\alpha k)|^2
    \,d\theta\,d\nu(y)
    =
    \nu(\mathcal C)
    \int_0^\alpha
    \sum_{k\in\mathbb Z}
    |\widehat u_0(\theta+\alpha k)|^2
    d\theta .
\end{align*}
Finally, the change of variables $\xi=\theta+\alpha k$ gives
\begin{align*}
    \int_0^\alpha
    \sum_{k\in\mathbb Z}
    |\widehat u_0(\theta+\alpha k)|^2\,d\theta
    =
    \int_{\mathbb R}|\widehat u_0(\xi)|^2\,d\xi.
\end{align*}
Since the Fourier transform is  unitary in $L^2(\mathbb R)$, we obtain 
\begin{align*}
    \frac{A\nu(\mathcal C)}{h}\|u_0\|_2^2
    \le
    \int_0^T\int_E |S(t)u_0(x)|^2\,d\mu_E(x)\,dt
    \le
    \frac{B\nu(\mathcal C)}{h}\|u_0\|_2^2 .
\end{align*}
This proves the desired estimate for data with compactly supported
Fourier transform. The general case follows by a density argument.

\end{proof}

\subsection{The standard linear KdV equation on the line}

Consider
\begin{align}\label{rthgK}
    w_t+w_x+w_{xxx}=0.
\end{align}
Then $v(t,y)=w(t,y+t)$ solves the Airy equation.  Hence Theorem \ref{th1} gives 
\begin{equation}\label{eq:moving-kdv-line}
    \|w_0\|_2^2\le C_T\sum_j\sup_{x\in E_j}
       \|\Tr_{x}w(\cdot,\cdot+x)\|_{L^2(0,T)}^2.
\end{equation}
Let $h>0$, $\mathcal C\subset[0,h)$ be a Borel set, and let $\nu$ be a finite
positive Borel measure supported on $\mathcal C$, with
\begin{align*}
    0<\nu(\mathcal C)<\infty.
\end{align*}
Then Proposition \ref{ziran} gives the moving periodic set observability for (\ref{rthgK}): 
For every $T>0$, there exist constants
$0<c_{T,h,\nu}\le C_{T,h,\nu}<\infty$ such that
\begin{align*}
    c_{T,h,\nu}\|u_0\|_{L^2(\mathbb R)}^2
    \le
    \int_0^T\int_E |w(t,x+t)|^2\,d\mu_E(x)\,dt
    \le
    C_{T,h,\nu}\|u_0\|_{L^2(\mathbb R)}^2,
\end{align*}

Fixed trace observability or fixed periodic set observability is generally false for (\ref{rthgK}). For example, one has the following counterexample  of  (\ref{rthgK})  in the setting of periodic point sequence observation.  

\begin{proposition}\label{prop:kvd-line-fail}
For the equation $w_t+w_x+w_{xxx}=0$ and the lattice $2\pi\Z$, there is no $C_T$ such that
\begin{align*}
    \|w_0\|_2^2\le C_T\int_0^T\sum_{k\in\Z}|w(t,2\pi k)|^2\dd t
\end{align*}
for all smooth compactly supported $w_0$.
\end{proposition}

\begin{proof}
Let $\phi\in C_c^\infty(\R)$ be non-zero and set
\begin{align*}
    f_N(x)=N^{-1/2}\phi(x/N)(1-e^{ix}).
\end{align*}
Then $f_N(2\pi k)=0$ for all $k$ and $\|f_N\|_2\to c_\phi>0$. Since 
$$(\partial^3_x+\partial_x)(1-e^{ix})=0,
$$
straightforward calculations show 
\begin{align*}
    \|(\partial^3_x+\partial_x)f_N\|_{H^1}=O(N^{-1}).
\end{align*}
Therefore, uniformly for $t\in[0,T]$,
\begin{align*}
    \|e^{-t(\partial_x^3+\partial_x)}f_N-f_N\|_{H^1}\le C_TN^{-1}.
\end{align*}
By  Lemma \ref{samu}, the observation tends to zero, while $\|f_N\|_2$ does not.
\end{proof}

 \section{Point Observability for Periodic Airy Equations  with a Bounded Potential}
 \label{point}

\subsection{Notation and elementary analytic estimates}

We identify
\begin{align*}
  \T=\R/2\pi\Z
\end{align*}
and use the normalized Fourier basis
\begin{align*}
  e_n(x)=(2\pi)^{-1/2}\ee^{\ii n x},\qquad n\in\Z .
\end{align*}
Then
\begin{align*}
  \partial_x^3e_n=-\ii n^3e_n .
\end{align*}
The inner product is defined by 
$$
\langle f,g\rangle :=\int_{\mathbb T} f(x)\overline{g(x)} dx. 
$$
For a real-valued potential
$p\in L^\infty(\T)$, set
\begin{align*}
  \Lp=\partial_x^3+p(x),\qquad \Dom(\Lp)=H^3(\T).
\end{align*}
The operator $-\partial_x^3$ generates the unitary group
$S(t)$ with $S(t)e_n=\ee^{\ii n^3t}e_n$.  And  multiplication by $-p$ is bounded on
$L^2(\T)$. Then one can verify that $L_p$ defines  a strongly continuous group satisfying
\begin{align*}
  U_p(t)=\ee^{-t\Lp},\qquad
  \|U_p(t)\|_{L^2\to L^2}\le \ee^{\|p\|_\infty |t|} .
\end{align*}
 The above upper bound follows by energy and density arguments. 

We shall use the following elementary cubic non-harmonic Fourier estimate.

\begin{lemma} \label{lecubr}
Let $\sigma\in\R\setminus\{0\}$, $R_0\ge0$ and let
$(\omega_n)_{|n|\ge N_1}\subset\R$ satisfy
\begin{align*}
  |\omega_n-\sigma n^3|\le R_0(1+|n|),
  \qquad |n|\ge N_1 .
\end{align*}
Then, after increasing $N_1$ if necessary, one has 
$\omega_n\ne\omega_m$ for
$n\ne m$, and
\begin{align*}
  \rho_N:=
  \sup_{|n|\ge N}
  \sum_{\substack{|m|\ge N\, ,\, m\ne n}}
  \frac1{|\omega_n-\omega_m|}
  \longrightarrow0,
  \qquad N\to\infty .
\end{align*}
\end{lemma}

\begin{proof}
For integers $n\ne m$,
\begin{align*}
  |n^3-m^3|=|n-m|\,|n^2+nm+m^2|.
\end{align*}
Since
\begin{align*}
  n^2+nm+m^2=\frac12(n^2+m^2)+\frac12(n+m)^2
  \ge \frac12(n^2+m^2),
\end{align*}
we have
\begin{align*}
  |\sigma|\,|n^3-m^3|
  \ge \frac{|\sigma|}{2}|n-m|(n^2+m^2).
\end{align*}
Take $N_1$  large enough so that
\begin{align*}
  R_0(2+|n|+|m|)\le \frac{|\sigma|}{4}|n-m|(n^2+m^2)
\end{align*}
whenever $|n|,|m|\ge N_1$ and $n\ne m$.  This is possible since
$|n-m|\ge1$ and the right hand side grows at least quadratically in
$\max\{|n|,|m|\}$.  Then
\begin{equation}\label{cuesepa}
  |\omega_n-\omega_m|
  \ge \frac{|\sigma|}{4}|n-m|(n^2+m^2).
\end{equation}
Fix $|n|\ge N$.  Split the sum over $m$ into two parts.  If
$|m|\le 2|n|$, then by $n^2+m^2\ge n^2$ and   
\eqref{cuesepa}, there holds 
\begin{align}\label{Hy71} 
  \sum_{\substack{|m|\ge N,\,m\ne n,\, |m|\le2|n|}}
  \frac1{|\omega_n-\omega_m|}
  \le \frac{C}{n^2}
\sum_{1\le |m-n|\le3|n|}\frac1{|m-n|}
  \le C\frac{\log(2+|n|)}{n^2}.
\end{align}
If $|m|>2|n|$, then $|m-n|\ge |m|/2$ and
$n^2+m^2\ge m^2$. Thus one has 
\begin{align}\label{Hy72} 
  \sum_{|m|>2|n|}\frac1{|\omega_n-\omega_m|}
  \le C\sum_{|m|>2|n|}\frac1{|m|^3}
  \le \frac{C}{n^2}.
\end{align}
The two estimates (\ref{Hy71}) and (\ref{Hy72}) are uniform for $|n|\ge N$, and their right hand sides tend
to zero as $N\to \infty$.
\end{proof}

We recall the classical Schur test lemma below. 
\begin{lemma}\label{lem:schur}
Let $A=(a_{nm})$ be a matrix indexed by a countable set $J$.  Assume
\begin{align*}
  R:=\sup_{n\in J}\sum_{m\in J}|a_{nm}|<\infty,
  \qquad
  C:=\sup_{m\in J}\sum_{n\in J}|a_{nm}|<\infty .
\end{align*}
Then the  operator
$(Ax)_n=\sum_m a_{nm}x_m$ is a bounded operator on
$\ell^2(J)$, and
\begin{align*}
  \|A\|_{\ell^2\to\ell^2}\le (RC)^{1/2}.
\end{align*}
\end{lemma}

The following proposition gives a Hilbert-valued cubic Ingham estimate, 
which will be useful for high frequency spectral decomposition in Section  \ref{high-spectral}. 
\begin{proposition}\label{cuin}
Let $T>0$, let $\calH$ be a complex Hilbert space, and let
\begin{align*}
  \lambda_n=\alpha_n+\ii\beta_n,\qquad |n|\ge N_1,
\end{align*}
satisfy
\begin{align*}
  \sup_{|n|\ge N_1}|\alpha_n|<\infty,
  \qquad
  \beta_n=\sigma n^3+O(|n|),\quad \sigma\ne0 .
\end{align*}
Let $b_n\in\calH$ satisfy
\begin{align*}
  0<b_-\le \|b_n\|_{\calH}\le b_+<\infty,
  \qquad |n|\ge N_1.
\end{align*}
Then there exists $N\ge N_1$ such that the family
\begin{align*}
  h_n(t)=b_n\ee^{-\lambda_n t},\qquad |n|\ge N,
\end{align*}
is a Riesz sequence in $L^2(0,T;\calH)$.  Equivalently, there are constants
$0<A\le B<\infty$ such that for every scalar sequence  
$(c_n)_{|n|\ge N}$ in $\ell ^2$, 
\begin{equation}\label{gtgting}
  A\sum_{|n|\ge N}|c_n|^2
  \le
  \int_0^T\left\|\sum_{|n|\ge N}c_n b_n\ee^{-\lambda_n t}\right\|_{\calH}^2\dd t
  \le
  B\sum_{|n|\ge N}|c_n|^2 .
\end{equation}
More generally, if finitely many positive-length time intervals are prescribed,
then $N$ can be chosen so that the corresponding Riesz estimates hold on all
of them, with constants depending on the intervals.
\end{proposition}

\begin{proof}
Let $A_0=\sup|\alpha_n|$.  For $n\ne m$, integration by parts gives
\begin{align*}
  \left|
  \int_0^T \ee^{-(\lambda_n+\overline{\lambda_m})t}\dd t
  \right|
  &=\left|
  \int_0^T
  \ee^{-(\alpha_n+\alpha_m)t}
  \ee^{-\ii(\beta_n-\beta_m)t}\dd t
  \right|                                    \\
  &\le
  \frac{2\ee^{2A_0T}+2A_0T\ee^{2A_0T}}{|\beta_n-\beta_m|}
  =\frac{C_T}{|\beta_n-\beta_m|}.
\end{align*}
The diagonal terms satisfy
\begin{align*}
  T\ee^{-2A_0T}b_-^2
  \le
  \|b_n\|_{\calH}^2\int_0^T\ee^{-2\alpha_nt}\dd t
  \le
  T\ee^{2A_0T}b_+^2 .
\end{align*}
Let $G=(G_{nm})$ be the Gram matrix of the family $(h_n)_{|n|\ge N}$, i.e.,
$$
G_{nm}=\left(\int^{T}_0 e^{-(\lambda_n +\overline{\lambda_m})t} dt\right) \langle b_n,b_{m}\rangle_{\mathcal{H}}, \qquad |n|,|m|\ge N.  
$$
The diagonal part of $(G_{nm})$ is bounded below by
\begin{align*}
  d_-:=T\ee^{-2A_0T}b_-^2>0.
\end{align*}
The absolute row sums of its off diagonal part are dominated by
\begin{align}\label{KpD9i}
  b_+^2 C_T
  \sup_{|n|\ge N}
  \sum_{\substack{|m|\ge N\, ,\, m\ne n} }
  \frac1{|\beta_n-\beta_m|},
\end{align}
which by Lemma \ref{lecubr}  tends to zero as $N\to\infty$.
Choose $N$ sufficiently large so that (\ref{KpD9i}) is at most $d_-/2$. Then Lemma \ref{lem:schur}   gives that the
off diagonal part of $(G_{nm})$ has operator norm at most $d_-/2$ on $\ell^2$.  Thus the lower
bound in \eqref{gtgting} follows with $A=d_-/2$.  The upper bound
follows from the upper diagonal estimate and the same Schur bound.  The proof
on a subinterval $(a,b)$ is the same.
\end{proof}

The following elementary lemma proves the linear independence of exponential polynomials in vector spaces and is standard. We give a proof for completeness. 
\begin{lemma}\label{leexpin} 
Let $\mu_1,\ldots,\mu_L\in\C$ be distinct, let $d_1,
\ldots,d_L\in\N$, and let $v_{\ell,q}$ belong to a complex vector space.
If
\begin{align*}
  \sum_{\ell=1}^{L}\sum_{q=0}^{d_\ell-1}
  t^q\ee^{-\mu_\ell t}v_{\ell,q}=0
\end{align*}
for all $t$ in a non-empty open interval, then every $v_{\ell,q}$ is zero.
\end{lemma}

\begin{proof}
It is enough to prove the scalar case after applying arbitrary linear
functionals to the finite-dimensional span of the vectors $v_{\ell,q}$.  The
scalar functions are real analytic in $t$, so an identity on an interval is an
identity on $\R$.  We argue by induction on $L$.  For $L=1$, multiplying
by $\ee^{\mu_1t}$ gives a polynomial which is identically zero, so all its
coefficients are zero.  Assume the statement known for $L-1$ exponents. 

For fixed $\ell\ge2$, write the
corresponding scalar part as 
\begin{align}\label{Hy700pL}
\sum^{L}_{\ell=1 }\ee^{-\mu_\ell t}P_\ell(t),
\end{align}
where
$\deg P_\ell<d_\ell$. For $\ell \ge 2$, one has 
\begin{align*}
  \left(\frac{d}{dt}+\mu_1\right)^{d_1}
  \bigl(\ee^{-\mu_\ell t}P_\ell(t)\bigr)
  =\ee^{-\mu_\ell t}(D+\mu_1-\mu_\ell)^{d_1}P_\ell(t).
\end{align*}

Thus after applying   
$(\frac{d}{dt}+\mu_1)^{d_1}$ to (\ref{Hy700pL}),  the terms with exponent $\mu_1$ vanish, and
for $\ell\ge2$ the result is again a finite sum of functions
$t^q\ee^{-\mu_\ell t}$.  By the induction hypothesis, the new 
coefficient polynomial belonging to each exponent $\mu_\ell$, $\ell\ge2$, is
zero.  It remains only to verify that this forces the original coefficient
polynomial for that exponent to be zero.  
Since $\mu_1-\mu_\ell\ne0$, the operator
$D+\mu_1-\mu_\ell$ is represented in the monomial basis by an upper triangular
matrix with non-zero diagonal $\mu_1-\mu_\ell$. Hence it is invertible, and so
is its $d_1$-th power.  Therefore $P_\ell=0$ for every $\ell\ge2$.  The
original identity then contains only the exponent $\mu_1$, and the case
$L=1$ finishes the proof.
\end{proof}

We now extend Proposition \ref{cuin}  to include a finite-dimensional block. 
\begin{proposition} \label{profiniext}
Let $T>0$, let $\calH$ be a complex Hilbert space, and let
\begin{align*}
  h_n(t)=b_n\ee^{-\lambda_n t},
  \qquad |n|\ge N.
\end{align*}
Assume that
\begin{align*}
  \sup_{|n|\ge N}|\operatorname{Re}\lambda_n|<\infty,
\end{align*}
and  $\{h_n\}$ satisfies the Riesz bounds (\ref{gtgting}) of Proposition
\ref{cuin} both in $L^2(0,T;\calH)$ and in
$L^2(I;\calH)$ for one open interval
$I=(a,b)\Subset(0,T)$.  Let $G\subset L^2(0,T;\calH)$ be a
finite-dimensional space consisting of $\calH$-valued exponential polynomials,
\begin{align}\label{GtMM}
  g(t)=\sum_{\ell=1}^{L}\sum_{q=0}^{d_\ell-1}
       t^q\ee^{-\mu_\ell t}v_{\ell,q},
  \qquad v_{\ell,q}\in\calH,
\end{align}
where the distinct exponents $\mu_1,\ldots,\mu_L$ are different from all
$\lambda_n$, $|n|\ge N$.  Then
\begin{align*}
  G\cap \overline{\spn}\{h_n:|n|\ge N\}=\{0\}.
\end{align*}
Furthermore, for any norm $\|\cdot\|_G$ on $G$, there are constants
$0<A_G\le B_G<\infty$ such that
\begin{equation}\label{sgudvhkh}
  A_G\left(\|g\|_G^2+\sum_{|n|\ge N}|c_n|^2\right)
  \le
  \left\|g+\sum_{|n|\ge N}c_nh_n\right\|_{L^2(0,T;\calH)}^2
  \le
  B_G\left(\|g\|_G^2+\sum_{|n|\ge N}|c_n|^2\right)
\end{equation}
for all $g\in G$ and all finitely supported $(c_n)$.
\end{proposition}

\begin{proof}
Set 
\begin{align*}
  H_N=\overline{\spn}\{h_n:|n|\ge N\}\subset L^2(0,T;\calH).
\end{align*}
Define maps $S_T$ and $S_I$  by 
\begin{align*}
  S_Tc=\sum_{|n|\ge N}c_nh_n\quad\hbox{in }L^2(0,T;\calH),
  \qquad
  S_Ic=\sum_{|n|\ge N}c_nh_n\quad\hbox{in }L^2(I;\calH).
\end{align*}
The Riesz bounds imply that both maps are bounded from $\ell^2$ into the
corresponding $L^2$ space and that there are constants $A_T,A_I>0$ such that
\begin{align}\label{ZSE}
  A_T\|c\|_{\ell^2}^2\le \|S_Tc\|_{L^2(0,T;\calH)}^2,
  \qquad
  A_I\|c\|_{\ell^2}^2\le \|S_Ic\|_{L^2(I;\calH)}^2.
\end{align}
Thus $S_T$ is injective and $H_N:=\Ran S_T$ is closed.

We prove $G\cap H_N=\{0\}$.  Suppose that $g\in G$, $c=(c_n)\in\ell^2$,
and
\begin{equation}\label{tgv78jn}
  g=S_Tc\quad\hbox{in }L^2(0,T;\calH).
\end{equation}
Since $g\in G$, the element $g$  has a canonical
representative given by the finite exponential polynomial (\ref{GtMM}).  This representative is $C^\infty$ as a map $(0,T)\to\calH$, since
for every $r\ge0$
\begin{align*}
  g^{(r)}(t)=\sum_{\ell=1}^{L}\sum_{q=0}^{d_\ell-1}
       \frac{d^r}{dt^r}\bigl(t^q\ee^{-\mu_\ell t}\bigr)v_{\ell,q}
\end{align*}
is again a finite sum of continuous $\calH$-valued functions.

For $s\in\R$ define the forward difference
\begin{align*}
  \Delta_s f(t)=f(t+s)-f(t),
  \qquad
  \Delta_s^r=(\Delta_s)^r .
\end{align*}
Fix $r\ge1$.  Choose $\delta_r>0$ so small that
\begin{align*}
  a-r\delta_r>0,
  \qquad
  b+r\delta_r<T.
\end{align*}
For $0<|s|<\delta_r$, the functions
$\Delta_s^r g(t)$ and $\Delta_s^r S_Tc(t)$ are defined for almost every
$t\in I$.  From \eqref{tgv78jn},  one has
\begin{equation}\label{equdiffeq}
  \frac{\Delta_s^r g}{s^r}
  =
 \sum_{|n|\ge N}\left(\frac{\ee^{-\lambda_n s}-1}{s}\right)^r c_n h_n
  \quad\hbox{in }L^2(I;\calH).
\end{equation}
Indeed,  (\ref{equdiffeq}) follows from
$\Delta_s^r h_n(t)=(\ee^{-\lambda_n s}-1)^r h_n(t)$ and the uniform bound in  $n\ge N$ for fixed
$s\ne0$, i.e., 
\begin{align*}
  \sup_{|n|\ge N}\left|\frac{\ee^{-\lambda_n s}-1}{s}\right|^r
  \le
  \left(\frac{\ee^{A_0|s|}+1}{|s|}\right)^r<\infty,
  \qquad
  A_0=\sup_{|n|\ge N}|\Re\lambda_n|.
\end{align*}

We claim that the difference quotients of $g$ are bounded in $L^2(I;\calH)$ uniformly for
$0<|s|<\delta_r$.  In fact the iterated fundamental theorem of calculus gives
\begin{align*}
  \frac{\Delta_s^r g(t)}{s^r}
  =
  \int_{[0,1]^r}
  g^{(r)}\bigl(t+(\theta_1+\cdots+\theta_r)s\bigr)
  \,d\theta_1\cdots d\theta_r .
\end{align*}
By Jensen's inequality,
\begin{align}
  \left\|\frac{\Delta_s^r g}{s^r}\right\|_{L^2(I;\calH)}^2
  &\le
  \int_{[0,1]^r}\int_I
  \left\|g^{(r)}\bigl(t+(\theta_1+\cdots+\theta_r)s\bigr)\right\|_{\calH}^2
  \,dt\,d\theta          \nonumber  \\
  &\le
  \|g^{(r)}\|_{L^2(a-r\delta_r,b+r\delta_r;\calH)}^2
  =:C_r<\infty .\label{u7IIL}
\end{align}
Combining (\ref{u7IIL}) with \eqref{equdiffeq} and the lower Riesz
bound (\ref{ZSE}) for $S_I$ gives
\begin{align*}
  A_I\sum_{|n|\ge N}
  \left|\frac{\ee^{-\lambda_n s}-1}{s}\right|^{2r}|c_n|^2
  \le C_r,
  \qquad 0<|s|<\delta_r .
\end{align*}
Letting $s\to0$ and using Fatou's lemma yields
\begin{equation}\label{eqradce}
  \sum_{|n|\ge N}|\lambda_n|^{2r}|c_n|^2<\infty,
  \qquad r=1,2,\ldots .
\end{equation}
Together with $c\in\ell^2$, this also covers $r=0$.

Let
\begin{align*}
  Q(\tilde{D})=\prod_{\ell=1}^{L}(\tilde{D}+\mu_\ell)^{d_\ell},
  \qquad \tilde{D}=\frac{d}{dt},
\end{align*}
and let $m=\deg Q=\sum_\ell d_\ell$.  By construction of $G$,
$Q(\tilde{D})g=0$ pointwise on $(0,T)$, see the proof of Lemma \ref{leexpin} .  On the other hand,
\eqref{eqradce} implies that
\begin{align*}
  \bigl(Q(-\lambda_n)c_n\bigr)_{|n|\ge N}\in\ell^2,
  \qquad
  Q(-\lambda_n)=\prod_{\ell=1}^{L}(\mu_\ell-\lambda_n)^{d_\ell}.
\end{align*}
 For each $0\le q\le m$,
\eqref{eqradce} and the boundedness of $S_I$ show that
\begin{align*}
  \sum_{|n|\ge N}(-\lambda_n)^q c_nh_n
\end{align*}
converges in $L^2(I;\calH)$. Thus one has 
\begin{align*}
  \tilde{D}^qS_Ic=
  \sum_{|n|\ge N}(-\lambda_n)^q c_nh_n,
  \qquad 0\le q\le m.
\end{align*}
Applying this identity to the polynomial differential operator $Q(\tilde{D})$ and
using \eqref{tgv78jn} restricted to $I$, we obtain
\begin{align*}
  0=Q(\tilde{D})g
  =
  \sum_{|n|\ge N}Q(-\lambda_n)c_nh_n
  \quad\hbox{in }L^2(I;\calH).
\end{align*}
The lower Riesz bound (\ref{ZSE}) for $S_I$ gives
\begin{align*}
  Q(-\lambda_n)c_n=0,
  \qquad |n|\ge N.
\end{align*}
Since $\lambda_n\ne\mu_\ell$ for every $n$ and every $\ell$, one has
$Q(-\lambda_n)\ne0$.  Hence $c_n=0$ for all $|n|\ge N$, and then
$g=0$.  Thus $G\cap H_N=\{0\}$.

It remains to prove the two sided estimate (\ref{sgudvhkh}).  Since $G$ is finite-dimensional ,
$H_N$ is closed, and $G\cap H_N=\{0\}$, the sum $G\oplus H_N$ is a
topological direct sum.  Precisely, there is $\gamma>0$ such that
\begin{equation}\label{eq:topological-direct-lower}
  \|g+h\|_{L^2(0,T;\calH)}^2
  \ge \gamma\bigl(\|g\|_{L^2(0,T;\calH)}^2+
                   \|h\|_{L^2(0,T;\calH)}^2\bigr)
\end{equation}
for all $g\in G$ and $h\in H_N$.  If no such $\gamma$ existed, then there
would be $g_j\in G$ and $h_j\in H_N$ with
$\|g_j\|_{L^2}^2+\|h_j\|_{L^2}^2=1$ and $g_j+h_j\to0$ in $L^2$.  Since $G$ is finite-dimensional, passing
to a subsequence, we obtain  $g_j\to g\in G$.  Then $h_j\to -g$.  Since $H_N$ is
closed, $-g\in H_N$, so $g\in G\cap H_N=\{0\}$. Thus 
$g_j\to0$ and $h_j\to0$, which contradicts the normalization.

Now take $h=S_Tc$.  The lower Riesz bound (\ref{ZSE}) for $S_T$,
\eqref{eq:topological-direct-lower}, and the equivalence of all norms on the
finite-dimensional space $G$ give the lower bound in
\eqref{sgudvhkh}.  The upper bound follows from
\begin{align*}
  \|g+S_Tc\|_{L^2}^2
  \le 2\|g\|_{L^2}^2+2\|S_Tc\|_{L^2}^2,
\end{align*}
and the upper Riesz bound for $S_T$ and the norm equivalence on
$G$.

\end{proof}

\begin{remark}
Note that the constant $N$ in Proposition \ref{cuin}  and thus in Proposition \ref{profiniext} may depend on $T>0$.  
\end{remark}

\subsection{Periodic Airy and periodic linear KdV}

We now prove Proposition \ref{c1}. 

The one point observability for Airy on $\T$ is as follows. 

\begin{lemma}\label{thtouai}
For every $T>0$ and every $x_0\in\T$, there exist
$0<a_T\le b_T<\infty$ such that
\begin{align*}
  a_T\sum_{n\in\Z}|c_n|^2
  \le
  \int_0^T\left|
  \sum_{n\in\Z}c_n\ee^{\ii n x_0}\ee^{\ii n^3t}
  \right|^2\dd t
  \le
  b_T\sum_{n\in\Z}|c_n|^2.
\end{align*}
\end{lemma}

\begin{proof}
This follows directly by  
Haraux-Ingham theorem, see Proposition \ref{yugg}.  
\end{proof}

Consider  the periodic linear KdV equation
\begin{align}\label{254}
  u_t+u_x+u_{xxx}=0, \qquad (t,x)\in \mathbb R\times \mathbb T. 
\end{align}
For $u_0\in L^2$, the solution is of the form 
\begin{align*}
 u(t,x)=\sum_{n\in\Z}c_n\ee^{\ii n x}\ee^{\ii(n^3-n)t}.
 \end{align*}
Note that the frequency $n^3-n$ is zero for $n=-1,0,1$.

The following lemma shows that  one or two points observability fails for the linear KdV  on $\T$. 
\begin{lemma}\label{prkdvfai}
For any two points $x_1,x_2\in\T$, there is a non-zero
stationary solution  to (\ref{254}) which vanishes at both $x_1$ and $x_2$.  Hence one or two points cannot observe
the periodic linear KdV equation.
\end{lemma}

\begin{proof}
Set $z_j=\ee^{\ii x_j}$ and take
\begin{align*}
  u_0(x)=\ee^{-\ii x}(\ee^{\ii x}-z_1)(\ee^{\ii x}-z_2)
       =\ee^{\ii x}-(z_1+z_2)+z_1z_2\ee^{-\ii x}.
\end{align*}
This $u_0$ is stationary, and
vanishes at $x_1,x_2$.
\end{proof}

The following lemma implies that three points  suffice for the observation of periodic linear KdV. 
\begin{lemma}\label{lemkdvth} 
Let $E\subset\T$ contain three distinct points.  Then for every $T>0$ and any mild solution to (\ref{254}) with $u_0\in L^2$, there exist $C_{T,E},C'_{T,E}>0$ so that
\begin{align}
  \|u_0\|_{L^2(\T)}^2
  &\le
  C_{T,E}\sup_{x\in E}\|\Tr_xu\|_{L^2(0,T)}^2 \label{Fh78K}\\
 \sup_{x\in E}\|\Tr_xu\|_{L^2(0,T)}^2  &\le  C'_{T,E}\|u_0\|_{L^2(\T)}^2.\label{Sh78K}
\end{align}
If $\mu$ is a finite positive Borel measure whose support contains at least
three points, then
\begin{align}\label{7uIU}
  \|u_0\|_{L^2(\T)}^2
  \asymp_{T,\mu}
  \int_\T\|\Tr_xu\|_{L^2(0,T)}^2\dd\mu(x).
\end{align}
\end{lemma}
\begin{proof}
(\ref{Sh78K}) is a special case of Lemma \ref{trace-tori} below.

{\bf Step 1.} Choose three distinct points $x_1,x_2,x_3\in E$.  For the stationary modes, set
\begin{align*}
  P_x=c_{-1}\ee^{-\ii x}+c_0+c_1\ee^{\ii x}.
\end{align*}
The matrix
$$
\begin{pmatrix}
e^{-\ii x_1} & 1  & e^{\ii x_1} \\
e^{-\ii x_2}& 1 & e^{\ii x_2} \\
e^{-\ii x_3} & 1 & e^{\ii x_3}
\end{pmatrix}
$$
is a Vandermonde matrix in $\ee^{\ii x_r}$ multiplied by non-zero factors, 
thus being invertible. Hence the three values $P_{x_r}$ determine and control
$(c_{-1},c_0,c_1)$.

{\bf Step 2.} Let $e_{n}=(2\pi)^{-\frac{1}{2}}e^{\ii nx}$, $n\in \mathbb Z$.  We have  
$$
(\partial^3_x+\partial_x)e_{n}= (\ii n-\ii n^3)e_{n},
$$
and for distinct $n,m\in \mathbb Z$, 
\begin{align}\label{Pokm}
(\ii n-\ii n^3)=(\ii m-\ii m^3), \qquad {\rm iff}\qquad n, m\in \{-1,0,1\}.
\end{align}
For $u_0=\sum_{n\in \mathbb Z}c_n e_n(x)$ in $L^2$, one has 
\begin{align*}
u(t,x)=\sum_{n\neq 0,\pm 1} c_n e^{\ii(n^3-n) t}e_{n}(x)+(2\pi)^{-\frac{1}{2}}(c_0+c_1e^{\ii x}+c_{-1}e^{-\ii x}).
\end{align*}
For $n\notin\{-1,0,1\}$, set $\lambda_n=-\ii(n^3-n)$ and $
  b^{(j)}_n= \ee^{\ii n x_j}.
$
Then $|b^{(j)}_n|=1$, for $j=1,2,3$, and
$\operatorname{Im}\lambda_n=-(n^3-n)=-n^3+O(n)$. Hence by  Lemma
\ref{lecubr} and Proposition \ref{cuin}, for $I=(0,T)$, there exists $\tilde{N}\ge 2$
such that 
\begin{align*}
c_{I}\sum_{|n|\ge \tilde{N}}|c_n|^2
\le \|\sum_{|n|\ge \tilde{N}} c_n e^{\ii(n^3-n) t}e_{n}(x)\|^2_{L^2(I)}
\le C_{I}\sum_{|n|\ge \tilde{N}}|c_n|^2, \mbox{ } \forall x\in \mathbb T
\end{align*}
for some $c_{I},C_{I}>0$. 
Then applying Proposition  \ref{profiniext} to the finite-dimensional space 
$$
G_j:= \left\{\sum_{2\le |n|<\tilde{N}} l_n e^{\ii(n^3-n) t}e_{n}(x_j)+ l_0+l_1e^{\ii x_j}+l_{-1}e^{-\ii x_j}: 
\{l_m\}^{\tilde{N}-1}_{m=-1}\subset \mathbb  C\right\} 
$$
gives that for some $c_I,C_I>0$  
\begin{align}
c_{I}(\sum_{|n|\ge \tilde{N}} |c_n|^2+\|g_j\|_{L^2(I)})
\le \|u(t,x_j)\|^2_{L^2(I)} 
\le C_{I}(\sum_{|n|\ge \tilde{N}} |c_n|^2+\|g_j\|_{L^2(I)}),\label{yOLK}
\end{align}
where 
$$
g_j=\sum_{2\le |n|< \tilde{N}} c_n e^{\ii(n^3-n) t}e_{n}(x_j)+(2\pi)^{-\frac{1}{2}}(c_0+c_1e^{\ii x_j}+c_{-1}e^{-\ii x_j}).
$$ 

We now prove 
\begin{align}\label{PPI9}
g_j=0, \mbox{ } \forall j=1,2,3 
\Longrightarrow  
\mbox{ }c_n=0, \mbox{ } \forall\mbox{ } 0\le |n|<\tilde{N}. 
\end{align}
In fact, (\ref{Pokm}) and Lemma \ref{leexpin} imply that 
\begin{align*}
g_j=0, \mbox{ } \forall j=1,2,3
\Longrightarrow &c_n=0, \mbox{ }\forall \mbox{ }2\le |n|< \tilde{N}\\
& {\rm and} \mbox{ }c_0+c_1e^{ix_j}+c_{-1}e^{-ix_j}=0,\mbox{ } j=1,2,3. 
\end{align*}
Then Step 1 shows $c_0=c_1=c_{-1}=0$. 
Therefore (\ref{PPI9}) holds, and thus there exists $C_{I}>0$ such that 
\begin{align*}
\sum_{0\le |n|< \tilde{N}}|c_n|^2\le C_{I} \sup_{j}\|g_j\|^2_{L^2(I)}, 
\end{align*}
which combined with (\ref{yOLK}) gives the stated observability inequality (\ref{Fh78K}).

{\bf Step 3.} We prove (\ref{7uIU}). We claim  the quadratic form
\begin{align*}
  (c_{-1},c_0,c_1)\mapsto
  \int_\T|c_{-1}\ee^{-\ii x}+c_0+c_1\ee^{\ii x}|^2\dd\mu(x)
\end{align*}
is positive definite. In fact, if it vanished, the trigonometric polynomial
$c_{-1}\ee^{-\ii x}+c_0+c_1\ee^{\ii x}$ would vanish on $\supp\mu$. After
multiplication by $\ee^{\ii x}$, one gets  a polynomial of degree at most two but having
at least three distinct zeros.  Hence it must be identically zero. 

{\bf Step 4.}
For the version (\ref{7uIU}), take $\calH=L^2(\T,\mu)$ and
$b_n(x)=\ee^{\ii n x}$.  Then
$\|b_n\|_{\calH}=\mu(\T)^{1/2}$ for every $n$, provided
$\mu(\T)>0$.  Again by Proposition \ref{cuin}, deleting
finitely many low frequencies gives the estimates 
of  high modes, say $\{e_{n}\}^{\infty}_{|n|\ge \bar{N}}$.  
It remains to deal with the remaining finite modes.
The new part is to prove 
\begin{align}\label{KKI9}
\int_{\mathbb T}\|g_x\|^2_{L^2(I)}\, d \mu(x) =0, \mbox{ } 
\Longrightarrow  
\mbox{ }c_n=0, \mbox{ } \forall\mbox{ } 0\le |n|<\bar{N},
\end{align}
where 
$$
g_x=\sum_{2\le |n|< \bar{N}} c_n e^{\ii(n^3-n) t}e_{n}(x)+(2\pi)^{-\frac{1}{2}}(c_0+c_1e^{\ii x}+c_{-1}e^{-\ii x}).
$$
In fact, Lemma \ref{leexpin} shows  $c_n=0$ for all $2\le |n|<\bar{N}$, and 
Step 3 yields $c_n=0$ with $n=0,\pm 1$.
Thus Proposition \ref{profiniext} with $\calH=L^2(\T,\mu)$ gives  (\ref{7uIU}).

\end{proof}

Until now, we have proved Proposition \ref{c1} by Lemma \ref{prkdvfai} and  Lemma \ref{lemkdvth}.  
The proof of \ref{prkdvfai} gives us the basic route mapping to the general case, see Section \ref{high-spectral} and Section \ref{finite} below. 

\section{High-frequency spectral decomposition for the Airy operator with potential}\label{high-spectral}

In this section, we construct the high-frequency spectral decomposition of
$L_p$.

Throughout this section, assume  $p\in L^\infty(\T;\R)$. Let
\begin{align}\label{br}
  L_0=\partial_x^3,
  \qquad
  \nu_n=-\ii n^3, \qquad
  e_n(x)=(2\pi)^{-1/2}\ee^{\ii n x},\qquad n\in\Z .
\end{align}
For $n\in\Z$, let $P_n^0$ be the orthogonal projection onto $\C e_n$.
The operator $\Lp$ has compact resolvent.  Indeed, choose
$z\in\rho(L_0)$ such that $\|p(z-L_0)^{-1}\|<1$.  Then
\begin{align*}
  (z-\Lp)^{-1}=(z-L_0)^{-1}\bigl(I-p(z-L_0)^{-1}\bigr)^{-1},
\end{align*}
and $(z-L_0)^{-1}:L^2(\T)\to H^3(\T)\hookrightarrow L^2(\T)$ is compact.
Thus the spectrum of $\Lp$ consists only of isolated eigenvalues of finite
algebraic multiplicity, with no finite accumulation point.

We will use Kato's perturbation method (e.g. \cite{Kato}) and Bari-Markus theorem on projection operators (see \cite{B1951,M1962}) to prove 
 the following spectral decomposition proposition of $L_p$. 

\begin{proposition}\label{prospec}
Let $p\in L^\infty(\T;\R)$.  There exist sufficiently large $N_0\ge1$, a finite-dimensional 
$\Lp$-invariant subspace $X_0$, simple eigenvalues $\lambda_n$, and
eigenfunctions $\phi_n$, indexed by $|n|\ge N_0$, such that 
\begin{enumerate}[label=\textup{(S\arabic*)}]
  \item $\Lp\phi_n=\lambda_n\phi_n$, and the algebraic spectral subspace
  associated with $\lambda_n$ is one-dimensional.
  \item There is a Riesz direct decomposition
  \begin{equation}\label{eqried}
    L^2(\T)=X_0\dotplus \overline{\spn}\{\phi_n:|n|\ge N_0\}.
  \end{equation}
  Thus every $f\in L^2(\T)$ has a unique expansion
  \begin{align*}
    f=f_0+\sum_{|n|\ge N_0}c_n\phi_n,
    \qquad f_0\in X_0,
  \end{align*}
  and there are constants $0<A\le B<\infty$ such that
  \begin{equation}\label{rieseq}
    A\left(\|f_0\|_2^2+\sum_{|n|\ge N_0}|c_n|^2\right)
    \le \|f\|_2^2
    \le
    B\left(\|f_0\|_2^2+\sum_{|n|\ge N_0}|c_n|^2\right).
  \end{equation}
  \item The eigenvalues satisfy
  \begin{equation}\label{eq:lambda-asymp}
    \lambda_n=-\ii n^3+\bar p+O(n^{-2}),
    \qquad |n|\to\infty,
  \end{equation}
  where
  \begin{align*}
    \bar p=(2\pi)^{-1}\int_\T p(x)\dd x .
  \end{align*}
  In particular, $|\operatorname{Re}\lambda_n|$ is bounded and
  $\operatorname{Im}\lambda_n=-n^3+O(1)$.
  \item The eigenfunctions $\phi_n$ can be normalized by
  $\langle\phi_n,e_n\rangle=1$, and then 
  \begin{equation}\label{eq:phi-asymp}
    \|\phi_n-e_n\|_{C^0(\T)}\longrightarrow 0,
    \qquad |n|\to\infty.
  \end{equation}
And  there holds 
  \begin{equation}\label{qyunbou}
   \frac12(2\pi)^{-1/2}\le |\phi_n(a)|\le \frac32(2\pi)^{-1/2},
    \qquad \forall a\in\T,
    \quad |n|\ge N_0 .
  \end{equation}
\end{enumerate}
\end{proposition}
\begin{proof}
Recall that $\nu_n=-\ii n^3$. Let $K=\|p\|_\infty$, and choose a number $R>K+1$.  For $|n|$ large, the
closed discs
\begin{align*}
  D_n=\{z\in\C: |z-\nu_n|\le R\}
\end{align*}
are pairwise disjoint.  If $z\in\partial D_n$, then
$\dist(z,\{\nu_k:k\in\Z\})=R$ for all sufficiently large $|n|$, and therefore
\begin{align*}
  \|p(z-L_0)^{-1}\|\le K/R<1.
\end{align*}
Thus $z\in\rho(L_0+sp)$ for every $s\in[0,1]$.  The Riesz projection
\begin{align*}
  P_n(s)=\frac1{2\pi\ii}\int_{\partial D_n}(z-L_0-sp)^{-1}\dd z
\end{align*}
is consequently well-defined and depends continuously on $s$ in operator
norm. The rank of $P_n(s)$ is constant in $s$, by the well-known fact that  two projections $P$ and $Q$ of finite rank satisfying $\|P-Q\|<1$ must have the same rank.   At $s=0$ the rank is one, since
$D_n$ contains only the eigenvalue $\nu_n$ of $L_0$, which is simple.  Hence
\begin{align*}
  P_n:=P_n(1)
\end{align*}
has rank one.

Let $\lambda_n$ be the unique eigenvalue of $\Lp$ in $D_n$, counted with
algebraic multiplicity.  Since $\operatorname{rank}P_n=1$, this eigenvalue is
algebraically simple.  Let $\phi_n\in\Ran P_n\setminus\{0\}$.  We first prove
that $n$-th Fourier coefficient of $\phi_n$ is non-zero for $|n|$ sufficiently large. We prove this by contradiction.  If
$\widehat\phi_n(n)=0$, then for every $k\ne n$,
\begin{align*}
  (\nu_k-\lambda_n)\widehat\phi_n(k)=-\widehat{p\phi_n}(k).
\end{align*}
For $z\in D_n$ and $k\ne n$, the distance $|\nu_k-z|$ is bounded below by
$c n^2$ for all sufficiently large $|n|$.  Therefore
\begin{align*}
  \|\phi_n\|_2
  \le Cn^{-2}\|p\phi_n\|_2
  \le CKn^{-2}\|\phi_n\|_2,
\end{align*}
which is impossible for large $|n|$. 

Since $\widehat{\phi}_n(n)\neq 0$ for $|n|$ large, we can normalize $\phi_n$ by
$\widehat\phi_n(n)=\langle\phi_n,e_n\rangle=1$.  Writing
$\phi_n=e_n+r_n$ gives
\begin{align*}
  \widehat r_n(k)=-\frac{\widehat{p\phi_n}(k)}{\nu_k-\lambda_n},
  \qquad k\ne n; \qquad  \widehat r_n(n)=0.
\end{align*}
Thus
\begin{equation}\label{eq:r-L2}
  \|r_n\|_2
  \le Cn^{-2}\|p\phi_n\|_2
  \le CKn^{-2}(1+\|r_n\|_2),
\end{equation}
and hence
\begin{equation}\label{eq:phi-L2-close}
  \|\phi_n-e_n\|_2\le Cn^{-2}.
\end{equation}

Apply the same argument to the adjoint operator
$\Lp^*=-\partial_x^3+p$.  The adjoint spectral projection $P_n^*$ has rank
one and is associated with $\overline{\lambda_n}$.  Let $\psi_n\in\Ran P_n^*$
be normalized by $\langle\psi_n,e_n\rangle=1$.  Then
\begin{equation}\label{oosiclo}
  \|\psi_n-e_n\|_2\le Cn^{-2}.
\end{equation}
And we can explicitly write $P_n$ as 
\begin{align*}
  P_nf=\frac{\langle f,\psi_n\rangle}{\langle\phi_n,\psi_n\rangle}\phi_n .
\end{align*}
Indeed, letting $P_nf=k\phi_n$ and taking inner product with $\psi_n$ on both sides, using
$P_n^*\psi_n=\psi_n$, we determine the scalar coefficient  $k$.  From
\eqref{eq:phi-L2-close} and \eqref{oosiclo}, one has
\begin{align}\label{newp}
\langle\phi_n,\psi_n\rangle=1+O(n^{-2}).
\end{align}
For $\|f\|_2=1$, write $r_n=\phi_n-e_n$ and $s_n=\psi_n-e_n$.  Then we have 
\begin{align}\label{LAO}
  (P_n-P_n^0)f
  &=\left(\frac{\langle f,\psi_n\rangle}{\langle\phi_n,\psi_n\rangle}
        -\langle f,e_n\rangle\right)e_n
    +\frac{\langle f,\psi_n\rangle}{\langle\phi_n,\psi_n\rangle}r_n .
\end{align}
By (\ref{oosiclo}) and (\ref{newp}), the scalar coefficient before $e_n$ in the right side of (\ref{LAO}) is bounded by  
\begin{align*}
  C\bigl(|\langle f,s_n\rangle|+|\langle\phi_n,\psi_n\rangle-1|\bigr)
  \le Cn^{-2}.
\end{align*}
And by (\ref{eq:r-L2}) and (\ref{newp}), the $r_n$ involved term in (\ref{LAO})  is dominated by 
$$C\|r_n\|_2\le Cn^{-2}.$$  Hence 
\begin{equation}\label{eqprcstr}
  \|P_n-P_n^0\|\le Cn^{-2}.
\end{equation}

Choose $N_0$ so large that all previous estimates hold and
\begin{align}\label{FRkk}
  \sum_{|n|\ge N_0}\|P_n-P_n^0\|<\frac14,
  \qquad
  \sum_{|n|\ge N_0}\|\phi_n-e_n\|_2^2<\frac14.
\end{align}
Now we can apply Bari-Markus theorem to obtain the spectral decomposition theorem of $L_p$. In fact, let 
$$P_\infty^0:=\sum_{|n|\ge N_0}P^0_{n},  
$$
which converges strongly in $\mathcal{L}(L^2; L^2)$ by the Parseval identity. 
Then there exists a projection operator $P_{\infty}$ such that 
\begin{align}\label{IuHJ}
\lim_{J\to \infty}  \sum_{N_0\le |n|\le J}P_{n}=P_{\infty}, \mbox{ }   {\rm strongly}\mbox{ } {\rm in} \mbox{ }\mathcal{L}(L^2; L^2). 
\end{align}
In fact, one has 
\begin{align}\label{YWRREE}
\sum_{N_0\le |n|\le J}P_{n} =\sum_{N_0\le |n|\le J}P^0_n+\sum_{N_0\le |n|\le J}(P_n-P_n^0),
\end{align}
Notice that  the first series on the right hand side of (\ref{YWRREE}) converges strongly  in $\mathcal{L}(L^2; L^2)$  as $J\to \infty$, and the second series converges in the operator norm $\mathcal{L}(L^2; L^2)$  by (\ref{FRkk}). 
Thus (\ref{IuHJ}) holds. 

Furthermore, one  obtains from  (\ref{FRkk})-(\ref{YWRREE}) that 
\begin{equation}\label{eq:Pinfty-close}
  \|P_\infty-P_\infty^0\|<\frac14.
\end{equation}
Set
\begin{align*}
  X_0=\Ker P_\infty,
  \qquad
  H_\infty=\Ran P_\infty.
\end{align*}
Then we obtain  
$$L^2(\T)=X_0\dotplus H_\infty$$
by writing 
\begin{align}\label{RtGVV}
f=P_{\infty} f+(I-P_{\infty})f.
\end{align}
 The space $X_0$ is finite-dimensional. In fact, if $x\in X_0$, then by \eqref{eq:Pinfty-close}
\begin{align*}
  \|P_\infty^0x\|_2\le\frac14\|x\|_2,
\end{align*}
so the projection $x\mapsto(\Id-P_\infty^0)x$ is injective from $X_0$ into
the finite-dimensional space
$\spn\{e_n:|n|<N_0\}$. Moreover, $X_0$ and $H_\infty$ are invariant under the group $U_p(t)$, since each Riesz projection $P_n$ commutes with the group $U_p(t)$, and so does the strong limit $P_{\infty}$.    

Define a map
\begin{align*}
  S:\overline{\spn}\{e_n:|n|\ge N_0\}\to L^2(\T),
  \qquad
  S\left(\sum c_ne_n\right)=\sum c_n\phi_n .
\end{align*}
Since $\sum\|\phi_n-e_n\|_2^2<1/4$, the perturbation
$S-\Id$ is Hilbert--Schmidt with norm at most $1/2$.  Thus $S$ is an
isomorphism from the high Fourier subspace $\overline{\spn}\{e_n:|n|\ge N_0\}$ onto
$\overline{\spn}\{\phi_n:|n|\ge N_0\}$.  Note that $\overline{\spn}\{\phi_n:|n|\ge N_0\}$ is exactly
$H_\infty$. Hence by the inverse operator theorem,  there are constants
$a,b>0$ such that, for $h=\sum c_n\phi_n\in H_\infty$,
\begin{align}\label{Ys7FFD}
  a\sum |c_n|^2\le \|h\|_2^2\le b\sum |c_n|^2.
\end{align}
For $f\in L^2$, let $h=P_\infty f$ and
$f_0=(I-P_\infty)f$. Then 
\begin{align*}
  \|h\|_2+\|f_0\|_2& \le (\|P_\infty\|+\|I-P_\infty\|)\|f\|_2\\
  \|f\|_2&\le\|f_0\|_2+\|h\|_2
\end{align*}
These two estimates, together with 
(\ref{Ys7FFD}),  give \eqref{eqried} and \eqref{rieseq}.

It remains to prove the asymptotics.  Taking inner product with $e_n$ on both sides of $\Lp\phi_n=\lambda_n\phi_n$ and using $\langle\phi_n,e_n\rangle=1$, we obtain 
\begin{align*}
  \lambda_n=\nu_n+\langle p\phi_n,e_n\rangle .
\end{align*}
Since $\phi_n=e_n+O(n^{-2})$ in $L^2$,
\begin{align*}
  \langle p\phi_n,e_n\rangle
  =\langle pe_n,e_n\rangle+O(n^{-2})
  =\bar p+O(n^{-2}),
\end{align*}
which proves \eqref{eq:lambda-asymp}.

Recall that $\phi_n=e_n+r_n$. Then $\widehat{r}_n(n)=0$ and for $k\ne n$,
\begin{align*}
  \widehat r_n(k)=-\frac{\widehat{p\phi_n}(k)}{-\ii k^3-\lambda_n}.
\end{align*}
Using \eqref{eq:lambda-asymp}, for all large $|n|$, there holds 
\begin{align*}
  \sup_{k\ne n}\frac{\langle k\rangle}{|-\ii k^3-\lambda_n|}
  \le \frac{C}{|n|}.
\end{align*}
Therefore, we see 
\begin{align*}
  \|r_n\|_{H^1}
  \le \frac{C}{|n|}\|p\phi_n\|_2
  \le \frac{C}{|n|}.
\end{align*}
The Sobolev embedding $H^1(\T)\hookrightarrow C^0(\T)$ proves
\eqref{eq:phi-asymp}.  Since $|e_n(a)|=(2\pi)^{-1/2}$ for every $a$, choose
$N_1\ge N_0$ so that
$\|\phi_n-e_n\|_{C^0}\le \frac12(2\pi)^{-1/2}$ for all $|n|\ge N_1$.  If
$N_1>N_0$, replace the low space by
\begin{align*}
  X_0\oplus\spn\{\phi_n:N_0\le |n|<N_1\}
\end{align*}
and replace the high tail by $|n|\ge N_1$.  This is only a finite-dimensional
modification: invariance, the Riesz direct decomposition, and the coefficient
norm equivalence are preserved.  Renaming $N_1$ as $N_0$, we obtain
\eqref{qyunbou}.

\end{proof}

\begin{remark}
The construction of $P_\infty$ also shows that all spectral subspaces not
represented by the projections $P_n$, $|n|\ge N_0$, lie in $X_0$.  Indeed,
if $Y$ is a generalized eigenspace associated with a spectral value different
from all high eigenvalues $\lambda_n$, then the Riesz projection $P_n$
annihilates $Y$ for every $|n|\ge N_0$.  Hence $P_\infty Y=0$, and
$Y\subset X_0$.
Thus  possible Jordan  chains occur only in the finite-dimensional space  $X_0$.
\end{remark}

As in the real line case, we also need the definition of  point trace in the periodic case.  

\begin{definition}\label{def:trace}
Let $p\in L^\infty(\T;\R)$ and assume the decomposition of Proposition
\ref{prospec}.  If
\begin{align*}
  f=f_0+\sum_{|n|\ge N_0}c_n\phi_n,
  \qquad f_0\in X_0,
\end{align*}
then for $a\in\T$ we define
\begin{align}\label{kiop}
  \Tr_aU_p(t)f
  =(U_p(t)f_0)(a)+
  \sum_{|n|\ge N_0}c_n\ee^{-\lambda_nt}\phi_n(a)
\end{align}
as an element of $L^2_{\mathrm{loc}}(\R_t)$. 
\end{definition}

\begin{remark}
Note that the high-frequency series in (\ref{kiop}) is
well-defined by Proposition \ref{cuin} with
$\calH=\C$ and $b_n=\phi_n(a)$.  The constants in the upper estimate are
uniform in $a$, because \eqref{qyunbou} gives uniform upper and
lower bounds for $|\phi_n(a)|$ on the high tail.  
\end{remark}

The point trace in Definition \ref{def:trace} is $L^2$ admissible and 
agrees with classical point values. 
\begin{lemma}\label{trace-tori}
If $f\in H^3(\T)$, then for every $a\in\T$ and every $t\in\R$,
\begin{align*}
  \Tr_aU_p(t)f=(U_p(t)f)(a).
\end{align*}
Moreover, for each $T>0$ there exists $C_T>0$ such that 
\begin{align}\label{KDS}
\sup_{a\in \mathbb T} \| \Tr_aU_p(t)f\|_{L^2(0,T)}\le C_T\|f\|_{L^2}.
\end{align}
\end{lemma}

\begin{proof}
Let
$$
f=f_0+\sum_{|n|\geq N_0}c_n\phi_n
$$
be the Riesz decomposition given by  Proposition
\ref{prospec}. The finite dimensional component satisfies
\begin{align}
\sup_{a\in\mathbb T}
\|\bigl(U_p(t)f_0\bigr)(a)\|_{L^2(0,T)}
\leq C_T\|f_0\|_{L^2}.
\label{10.21}
\end{align}
For the high frequency part define
$$
b_n(a)=\phi_n(a).
$$
By the uniform estimate of  Proposition
\ref{prospec},
$$
0<c_0\leq |b_n(a)|\leq C_0 ,
\qquad
|n|\geq N_0 ,
$$
uniformly in $a$.

The vector-valued Riesz estimate of Proposition \ref{cuin} gives
a number $N_1\geq N_0$ such that
\begin{align}
\int_0^T
\left|
\sum_{|n|\geq N_1}
c_ne^{-\lambda_nt}b_n(a)
\right|^2dt
\leq
C_T
\sum_{|n|\geq N_1}|c_n|^2
\label{10.22}
\end{align}
uniformly for $a\in\mathbb T$.

It remains to treat the finite block
$$
Y=
X_0\oplus
\operatorname{span}
\{\phi_n:N_0\leq |n|<N_1\}.
$$
Since $Y$ is finite dimensional and contained in $H^3(\mathbb T)$,
point evaluation is uniformly bounded by 
\begin{align}
\sup_{a\in\mathbb T}
\|(U_p(t)y)(a)\|_{L^2(0,T)}
\leq C_T\|y\|_{L^2},
\qquad y\in Y .
\label{10.23}
\end{align}
Combining the finite dimensional estimate with (\ref{10.22}), and using the
Riesz equivalence
$$
\|f\|_{L^2}^2
\asymp
\|f_0\|_{L^2}^2+
\sum_{|n|\geq N_0}|c_n|^2 ,
$$
we obtain
$$
\sup_{a\in\mathbb T}
\|\operatorname{Tr}_aU_p(\cdot)f\|_{L^2(0,T)}
\leq
C_T\|f\|_{L^2}.
$$
Finally, for $f\in H^3(\mathbb T)$,
the spectral expansion converges in the graph norm of $L_p$ because
$$
\sum_{|n|\geq N_0}|\lambda_nc_n|^2<\infty .
$$
Since
$$
D(L_p)=H^3(\mathbb T)
\hookrightarrow C^2(\mathbb T),
$$
the pointwise evaluation is justified and the abstract trace coincides with
the classical point value.
 
\end{proof}

\section{Finite-dimensional visibility and finite point observability}\label{finite}

In this section, we prove Theorem \ref{th3}. 

Let $F\subset\T$ be finite and non-empty.  By \eqref{eq:phi-asymp}, there exists sufficiently large $N_0$ depending only on $p$ such that
\begin{equation}\label{eq:visible-high}
  |\phi_n(a)|\ge\frac{1}{2}(2\pi)^{-\frac{1}{2}},
  \qquad \forall\mbox{ } |n|\ge N_0, \forall \mbox{ }a\in \mathbb T.
\end{equation}
Recall the spectral decomposition in Proposition \ref{prospec}, especially 
the  finite-dimensional space $X_0$ is $\Lp$-invariant and is contained in
$H^3(\T)$.  Set
\begin{align*}
  \mathcal{A}=\Lp|_{X_0},
  \qquad
  C_Fy=(y(a))_{a\in F}\in\C^F .
\end{align*}

\begin{definition}[Visibility on $X_0$]
The set $F$ is visible on $X_0$ if
\begin{align*}
  C_F\ee^{-t{\mathcal A}}y=0\quad\hbox{for all }t\in(0,T)
  \quad\Longrightarrow\quad y=0.
\end{align*}
\end{definition}

Note that this condition is independent of the particular positive value of $T$. In fact,  it is well-known that 
in the finite-dimensional space $X_0$, visibility is equivalent to the Kalman  rank condition, see Lemma \ref{lem:kalman} below.

\begin{lemma}\label{lem:kalman}
Let $d=\dim X_0$.  The following statements are equivalent:
\begin{enumerate}[label=\textup{(\roman*)}]
  \item $F$ is visible on $X_0$.
  \item If $y\in X_0$ and
  \begin{align*}
    C_F{\mathcal A}^q y=0,
    \qquad q=0,1,\ldots,d-1,
  \end{align*}
  then $y=0$.
  \item The linear map
  \begin{align*}
    \begin{pmatrix}
    C_F\\ C_F{\mathcal A}\\ \vdots\\ C_F{\mathcal A}^{d-1}
    \end{pmatrix}:X_0\to(\C^F)^d
  \end{align*}
  has rank $d$.
\end{enumerate}
\end{lemma}

\begin{proof}
This is a standard result, we give a short proof for completeness. For fixed $y\in X_0$, the map $t\mapsto C_F\ee^{-t\mathcal{A}}y$ is an entire
$\C^F$-valued function of $t$.  It vanishes on a non-empty interval if and
only if all derivatives at $0$ vanish, that is,
\begin{align*}
  C_F\mathcal{A}^q y=0\qquad(q\ge0).
\end{align*}
By the Cayley--Hamilton identity for the $d\times d$ matrix $\mathcal{A}$, the
conditions for $q\ge d$ follow from those for $0\le q\le d-1$.  This proves
the equivalence of (i) and (ii), and (ii) is exactly the rank statement (iii).
\end{proof}



We define the finite point observability
as follows. 
\begin{definition}
The finite set $F$ is observable for $\Lp$ in time $T>0$ if there is a
constant $C_{T,F,p}>0$ such that
\begin{equation}\label{eq:finite-obs-def}
  \|v_0\|_{L^2(\T)}^2
  \le
  C_{T,F,p}
  \int_0^T\sum_{a\in F}|\Tr_aU_p(t)v_0|^2\dd t
\end{equation}
for all $v_0\in L^2(\T)$.
\end{definition}

We now prove the 
finite point observability criterion
in Theorem \ref{th3}. 

\begin{proposition}
\label{propfincri}
Let $p\in L^\infty(\T;\R)$ and let $F\subset\T$ be finite and non-empty.
Let $X_0$ be the finite-dimensional space defined  in Proposition \ref{prospec}. The following statements are equivalent for
every $T>0$:
\begin{enumerate}[label=\textup{(\roman*)}]
  \item $F$ is observable in the sense of \eqref{eq:finite-obs-def}.
  \item $F$ is visible on $X_0$.
\end{enumerate}
\end{proposition}

\begin{proof}
First assume visibility on $X_0$. Fix $T>0$.  For $|n|\ge N_0$, set
\begin{align*}
  h_n(t)=\ee^{-\lambda_nt}C_F\phi_n\in L^2(0,T;\C^F).
\end{align*}
By \eqref{eq:visible-high}, one has 
\begin{align}\label{67yhP}
\|C_F\phi_n\|_{\C^F}\ge |\phi_n(a)|\ge \frac{1}{2}(2\pi)^{-\frac{1}{2}}
  \qquad \forall \mbox{ }|n|\ge N_0.
\end{align}
Moreover, \eqref{eq:phi-asymp} gives a uniform upper bound for $\|C_F\phi_n\|$.
 By Proposition
\ref{cuin}, there exists $N_2\ge N_0$ sufficiently large  such that the tail $(h_n)_{|n|\ge N_2}$ satisfies Riesz
bounds in $L^2(0,T;\C^F)$. 

Let 
$$
X_2:=X_0\oplus {\rm span}\{\phi_n: N_0\le |n|\le N_2\}.$$
Then $X_2$ is invariant under $L_p$. 
We claim that 
the map
\begin{align}\label{DRUJ}
  X_2\ni y\longmapsto C_F\ee^{-t\Lp|_{X_2}}y\in L^2(0,T;\C^F)
\end{align}
is injective. 
In fact, suppose that $C_{F}e^{-tL_p|_{X_2}}y=0$ with $y\in X_2$. Then $y\in X_2$
can be decomposed as $$y=y_0+\sum^{N_2}_{|n|\ge N_0}l_n \phi_n, \qquad y_0\in X_0, \mbox{ }l_n \in \mathbb C.$$
Thus 
$$C_{F}e^{-tL_p |_{X_2}} y= C_{F}e^{-t L_p|_{X_0}}y_0+\sum_{N_0\le |n|\le N_2} l_n e^{-\lambda_n t} C_{F}\phi_n=0. $$
From Lemma \ref {leexpin} and (\ref{67yhP}), we conclude  that 
\begin{align*}
 C_{F}e^{-t L_p|_{X_0}}y_0=0;\qquad
  l_n =0, \mbox { }N_0\le |n|\le N_2. 
\end{align*} 
Then, by assumption (ii), $y_0=0$.  
 Therefore, the map (\ref{DRUJ}) is injective.
Its range, denoted by $G_2$, is finite-dimensional and consists
of $\C^F$-valued exponential polynomials. By the injectivity  of (\ref{DRUJ}),  there exists $c>0$ such that 
\begin{align}\label{RhT}
\|C_{F}e^{-tL_p|X_2}y\|_{{L^2(0,T; \mathbb C}^F)}\ge c \|y\|_{L^2}, \qquad \forall \mbox{ }y\in X_2.
\end{align}

Applying Proposition \ref{profiniext} and Proposition
\ref{cuin} to $G_2$ and the tail
$(h_n)_{|n|> N_2}$, by (\ref{RhT}), there  exist constants $a,b>0$ such that
\begin{align}\label{Gnmbou}
a\left(\|y\|_2^2+\sum_{|n|>N_2}|c_n|^2\right)
  &\le
  \int_0^T
  \left\|C_F\ee^{-t\Lp|_{X_2}}y+
  \sum_{|n|> N_2}c_n\ee^{-\lambda_nt}C_F\phi_n\right\|_{\C^F}^2\dd t \\
  &\le
  b\left(\|y\|_2^2+\sum_{|n|> N_2}|c_n|^2\right).
\end{align}
Every $v_0\in L^2(\T)$ has the unique decomposition
\begin{align*}
  v_0=y+
  \sum_{|n|>N_2}c_n\phi_n,
  \qquad y\in X_2.
\end{align*}
Using \eqref{Gnmbou} and then the Riesz equivalence
\eqref{rieseq}, we obtain 
\begin{align*}
  \int_0^T\sum_{a\in F}|\Tr_aU_p(t)v_0|^2\dd t
  \ge
  c\left(\|y\|_2^2+
  \sum_{|n|> N_2}|c_n|^2\right)
  \ge c'\|v_0\|_2^2.
\end{align*}
This is \eqref{eq:finite-obs-def}. 

Conversely, if visibility on $X_0$ fails, there exists $0\ne y\in X_0$ such
that $C_F\ee^{-t\mathcal{A}}y=0$ for $0<t<T$.  Taking $v_0=y$ gives a non-zero
initial datum with zero observation.  Hence \eqref{eq:finite-obs-def} cannot
hold.
\end{proof}

{\bf Proof of  Theorem \ref{th3}}

Therefore, Theorem \ref{th3} follows by Lemma \ref{lem:kalman} and Proposition \ref{propfincri}.

\section{Observation sets with an accumulation point}\label{hausdorff}

In this section, we prove Corollary \ref{c2}. 

The preceding finite point criterion assumes only $p\in L^\infty(\mathbb T; \mathbb R)$.  To show that sets with accumulation points automatically remove the finite-dimensional
obstruction, we need a finite regularity assumption on $p$.

We first introduce a preliminary result on zero sets. 
\begin{lemma}\label{lem:zeros}
Let $Y\subset\Dom(\Lp)$ be a non-zero finite-dimensional $\Lp$-invariant
subspace.  Let $m=m(Y)$ be the degree of the minimal polynomial of
$\Lp|_Y$.  Assume
\begin{align*}
  p\in C^{3m-4}(\T;\R),
\end{align*}
with the convention that $C^{-1}(\T)=L^\infty(\T)$ when $m=1$.  If
$0\ne y\in Y$, then the zero set of $y$ has no accumulation point in $\T$.
\end{lemma}

\begin{proof}
Let
\begin{align*}
  q(z)=z^m+a_{m-1}z^{m-1}+\cdots+a_0
\end{align*}
be the minimal polynomial of $\Lp|_Y$.  For $0\le j\le m-1$, set 
\begin{align}\label{8iEP}
  z_j=\Lp^j y.
\end{align}
Since $Y$ is $\Lp$-invariant, every $z_j\in Y\subset H^3(\T)$.  Define
\begin{align*}
  W=(z_0,z_0',z_0'',z_1,z_1',z_1'',\ldots,z_{m-1},z_{m-1}',z_{m-1}'').
\end{align*}
For $0\le j\le m-2$, (\ref{8iEP}) yields 
\begin{align*}
  z_j'''+pz_j=z_{j+1},
\end{align*}
and the identity $q(\Lp|_Y)=0$ gives
\begin{align*}
  z_{m-1}'''+pz_{m-1}
  =-(a_{m-1}z_{m-1}+\cdots+a_0z_0).
\end{align*}
Thus $W$ satisfies a first order linear system
\begin{equation}\label{eq:root-system}
  \frac{d}{dx}W(x)=A_p(x)W(x),
\end{equation}
where the entries of $A_{p}$ consist only of $p(x)$, the constants $a_j$, and the constants
$0,1,-1$.  If $p\in L^\infty$, then $A_p\in L^\infty$, and the uniqueness of solutions to 
\eqref{eq:root-system} follows directly by applying Gronwall inequality to
the integral equation
\begin{align*}
  W(x)=W(x_0)+\int_{x_0}^xA_p(s)W(s)\dd s.
\end{align*}

When $m=1$, the vector $W=(y,y',y'')$.  Since $y\in H^3(\T)\subset
C^2(\T)$, an accumulation of zeros at $x_0$ implies
$y(x_0)=y'(x_0)=y''(x_0)=0$.  Hence $W(x_0)=0$, and the uniqueness gives
$y\equiv0$, which is a contradiction.

Assume $m\ge2$.  Since $p\in C^{3m-4}$, the coefficients of
\eqref{eq:root-system} are $C^{3m-4}$. 

Note that we have an elementary bootstrap:
if $A\in C^k$ and a continuous function $W$ satisfies
$W(x)=W(x_0)+\int_{x_0}^xA(s)W(s)\,ds$, then $W\in C^{k+1}$.   Hence
$W\in C^{3m-3}$.  In particular,
$z_0''=y''\in C^{3m-3}$, so $y\in C^{3m-1}$.  If the zeros of $y$ had an
accumulation point $x_0$, Taylor's expansion  gives
\begin{align}\label{6YKL}
  y^{(\ell)}(x_0)=0,
  \qquad 0\le \ell\le 3m-1.
\end{align}
We now verify that (\ref{6YKL}) forces $W(x_0)=0$.  By induction on $j$,
$z_j=\Lp^j y$ is a linear combination of
$y,y',\ldots,y^{(3j)}$  with coefficients involving derivatives of $p$ of
order at most $3j-3$ when $j\ge1$. Hence $z_j'$ and $z_j''$ are
linear combinations of $y,y',\ldots,y^{(3j+2)}$ with coefficients involving
derivatives of $p$ of order at most $3j-1$.  For $j\le m-1$, this order is
at most $3m-4$, which is available by hypothesis, and
$3j+2\le3m-1$.  Thus all components of $W(x_0)$ vanish.  Then uniqueness for
\eqref{eq:root-system} yields $W\equiv0$, and in particular $y\equiv0$, which is again
a contradiction.
\end{proof}

The following lemma connects  observations on sets with an accumulation point with observations on finite points. 

\begin{lemma}\label{lem:finite-sampling}
Let $Y$ be a finite-dimensional space of continuous functions on $\T$.  Let
$E\subset\T$.  If no non-zero element of $Y$ vanishes on all of $E$, then
there exists a finite subset $F\subset E$ such that the  map
\begin{align*}
  Y\ni y\longmapsto (y(a))_{a\in F}
\end{align*}
is injective.
\end{lemma}

\begin{proof}
For finite $F\subset E$, set
\begin{align*}
  K_F=\{y\in Y:y(a)=0\hbox{ for every }a\in F\}.
\end{align*}
Choose $F$ so that $\dim K_F$ is minimal among all finite subsets of $E$.
If $K_F\ne\{0\}$, take $0\ne y\in K_F$.  By the hypothesis, there is
$a\in E$ with $y(a)\ne0$.  Then
$K_{F\cup\{a\}}$ is a proper subspace of $K_F$, contradicting the minimality
of $\dim K_F$.  Therefore $K_F=\{0\}$.
\end{proof}

We are ready to prove the time-trace observability for $U_{p}f$ on $\mathbb T$ from sets with an accumulation point. 
Let $\{\phi_n\}$ and $X_0$ be the ones defined in Proposition  \ref{prospec}. 
\begin{proposition}
 \label{usdorff}
Let $p\in L^{\infty}(\mathbb T; \mathbb R)$ and $E\subset\T$ have an accumulation point.   Let $N_0$ be the 
integer given by Proposition \ref{prospec} such that
\begin{align}\label{FrnnW}
  |\phi_n(a)|\ge\frac{1}{2}(2\pi)^{-\frac{1}{2}}
  \qquad \forall \mbox{ } |n|\ge N_{0}, \mbox{ }\forall a\in \mathbb T. 
\end{align}
 If $X_{0}\ne\{0\}$, let $m_{0}$ be the degree of the minimal polynomial of $\Lp|_{X_{0}}$ and
assume additionally
\begin{align*}
  p\in C^{3m_{0}-4}(\T;\R),
\end{align*}
with the convention $C^{-1}=L^\infty$ if $m_{0}=1$.  If
$X_{0}=\{0\}$, no additional regularity is required.  Then for every
$T>0$ there exists $C_{T,E,p}>0$ such that
\begin{equation}\label{eq:hausdorff-obs}
  \|v_0\|_{L^2(\T)}^2
  \le
  C_{T,E,p}\sup_{a\in E}\|\Tr_aU_p(\cdot)v_0\|_{L^2(0,T)}^2
\end{equation}
for all $v_0\in L^2(\T)$.
\end{proposition}

\begin{proof}
If $X_{0}\ne\{0\}$, Lemma \ref{lem:zeros} applied to
$Y=X_{0}$ shows that no non-zero element of $X_{0}$ can vanish on all
of $E$, because $E$ has an accumulation point. Then by Lemma \ref{lem:finite-sampling}, there is a finite set
$F_0\subset E$ such that evaluation on $F_0$ is injective on
$X_{0}$. This injective property is obvious if $X_{0}=\{0\}$, and in this case we set $F_0=\emptyset$. 

Fix $a_0\in E$,  set 
\begin{align*}
  F=F_0\cup \{a_0\}. 
\end{align*}
 If
$C_F\ee^{-t\Lp|_{X_0}}y=0$ for $0<t<T$, then there holds 
$C_Fy=0$.  The injectivity of evaluation on $F$ thus gives $y=0$.  Hence
$F$ is visible on $X_0$.

The finite point observability criterion, Proposition \ref{propfincri},
therefore implies 
\begin{align*}
  \|v_0\|_2^2
  \le C_{T,F,p}\int_0^T\sum_{a\in F}|\Tr_aU_p(t)v_0|^2\dd t.
\end{align*}
Since $F\subset E$,
\begin{align*}
  \int_0^T\sum_{a\in F}|\Tr_aU_p(t)v_0|^2\dd t
  \le |F|\sup_{a\in E}\|\Tr_aU_p(\cdot)v_0\|_{L^2(0,T)}^2,
\end{align*}
which proves \eqref{eq:hausdorff-obs}.

\end{proof}

{\bf Proof of  Corollary \ref{c2}}

Note that the proof of Proposition \ref{prospec} indeed implies that $N_0$ has an upper bound depending only on $\|p\|_{L^{\infty}}$. Thus $3m_0-4\le 3n_0-4$ is also bounded by some constant depending only on $\|p\|_{L^{\infty}}$. 
Therefore,  Corollary \ref{c2} follows by Proposition \ref{usdorff} and (\ref{KDS}).

\section{Appendix A. Counterexamples}\label{example}

One cannot find  a finite set $E\subset \mathbb T$ such that the time-trace observability  on $E$ for $U_{p}(t)f$ with any $p\in L^{\infty}(\mathbb T; \mathbb R)$  holds simultaneously. Indeed, given a finite set $F\subset\T$, one can construct  a smooth potential $p$ and a stationary solution for $U_{p}(t)f$ which vanishes at all $x\in F$. 
\begin{lemma}
\label{thm:negative}
Let $F\subset\T$ be finite.  There exist $p\in C^\infty(\T;\R)$ and a nonzero function 
$\phi\in C^\infty(\T;\R)$ such that
\begin{align*}
  \phi'''+p(x)\phi=0,
  \qquad \phi|_F=0.
\end{align*}
Particularly, finite-point observability from $F$ fails for this potential in
every time $T>0$.
\end{lemma}

\begin{proof}
For each  $a\in F$, choose pairwise disjoint coordinate arcs $I_a$ around the points
$a$. Let
\begin{align*}
  g(x)=\prod_{a\in F}\bigl(2-2\cos(x-a)\bigr).
\end{align*}
Then $g\ge0$, the zero set of $g$ is exactly $F$, and every zero is of
order two.  For each $a\in F$, choose a smaller arc $J_a\Subset I_a$ and a
smooth cutoff function $\chi_a$ with $0\le\chi_a\le1$, which is supported in $I_a$ and equal to one on $J_a$.  In the
local coordinate $y=x-a$ on $I_a$, replace $g$ by $y^2$ on $J_a$ and
interpolate by
\begin{align*}
  \phi(x)=\chi_a(x)y^2+(1-\chi_a(x))g(x)
  \qquad (x\in I_a).
\end{align*}
Outside $\bigcup_a I_a$, set $\phi=g$.  The arcs are disjoint, so these
local definitions are compatible.  By choosing the arcs small, both $g$ and
$y^2$ are positive on $I_a\setminus\{a\}$. Hence $\phi>0$ on
$\T\setminus F$, while $\phi(a)=0$ for every $a\in F$.  Moreover,
$\phi(x)=y^2$ on $J_a$, so $\phi'''=0$ on $J_a$.

Define
\begin{align*}
  p(x)=-\frac{\phi'''(x)}{\phi(x)}
\end{align*}
where $\phi(x)\ne0$. In the neighborhoods $J_a$ of the zeros, define
$p(x)=0$.  These definitions agree on overlaps because $\phi'''=0$ in each
$J_a$.  Thus $p\in C^\infty(\T;\R)$, and by construction
$\phi'''+p\phi=0$.  The solution $v(t,x)=\phi(x)$ is stationary, non-zero and has zero observation at every point of $F$.
\end{proof}

\section{Appendix B. Gevrey cutoff function}

This Appendix proves the existence of compactly supported Gevrey cut-off functions. 

\begin{lemma}\label{lem:compact-gevrey-cutoff}
Let $1<\sigma<3$. Then there exists a nonnegative function
\begin{align*}
    \chi\in C_c^\infty(\mathbb R)
\end{align*}
such that
\begin{align}\label{G7hp}
    \chi(x)=1 \quad \text{for } |x|\le 1,
    \qquad
    \chi(x)=0 \quad \text{for } |x|\ge 2.
\end{align}
Moreover, $\chi$ is of Gevrey order $\sigma$ on $\mathbb R$, namely there exist
constants $C,A>0$ such that
\begin{align*}
    \|\chi^{(n)}\|_{L^\infty(\mathbb R)}
    \le C A^n (n!)^\sigma,
    \qquad n=0,1,2,\dots .
\end{align*}
\end{lemma}

\begin{proof}
Set
$
    \alpha=\frac{1}{\sigma-1}.
$
Then $\alpha>0$ and $
    1+\frac1\alpha=\sigma.
$ 

Define
\begin{align*}
    \theta(t)=
    \begin{cases}
        \exp(-t^{-\alpha}), & t>0,\\
        0, & t\le 0.
    \end{cases}
\end{align*}
It is a somewhat standard fact that $\theta\in G^\sigma(\mathbb R)$, we give a proof in Lemma \ref {onegevr} below for reader's convenience. Recall also that affine changes of variables and finite products
preserve the Gevrey order. 
Therefore, the function
\begin{align*}
\varphi(x)=\theta\!\left(\frac12+x\right)\theta\!\left(\frac12-x\right)
\end{align*}
belongs to $G^\sigma(\mathbb R)\cap C_c^\infty(\mathbb R)$. Moreover, one has 
\begin{align*}
    \varphi(x)=0 \quad \text{for } |x|\ge \frac12,
    \qquad
    \varphi(x)>0 \quad \text{for } |x|<\frac12.
\end{align*}
Hence
\begin{align*}
    c_0:=\int_{\mathbb R}\varphi(x)\,dx>0.
\end{align*}
Define
\begin{align*}
    \eta(x)=\frac{\varphi(x)}{c_0}.
\end{align*}
Then
\begin{align*}
    \eta\in G^\sigma(\mathbb R)\cap C_c^\infty(\mathbb R),
    \qquad
    \eta\ge0,
    \qquad
    \operatorname{supp}\eta\subset\left[-\frac12,\frac12\right],
    \qquad
    \int_{\mathbb R}\eta(x)\,dx=1.
\end{align*}
In particular, there exist constants $C_\eta,A_\eta>0$ such that
\begin{align*}
    \|\eta^{(n)}\|_{L^\infty(\mathbb R)}
    \le C_\eta A_\eta^n (n!)^\sigma,
    \qquad n=0,1,2,\dots .
\end{align*}
Now, set
$a=\frac32
$
and define
\begin{align*}
    \chi(x)
    =
    (\mathbf 1_{[-a,a]}*\eta)(x)
    =
    \int_{-a}^{a}\eta(x-y)\,dy .
\end{align*}
Since $\eta\in C_c^\infty(\mathbb R)$, differentiation under the integral sign gives
\begin{align*}
    \chi^{(n)}(x)
    =
    \int_{-a}^{a}\eta^{(n)}(x-y)\,dy,
    \qquad n=0,1,2,\dots .
\end{align*}
Hence there holds 
\begin{align*}
    |\chi^{(n)}(x)|
    \le
    2a\,\|\eta^{(n)}\|_{L^\infty(\mathbb R)}
    \le
    2a C_\eta A_\eta^n (n!)^\sigma .
\end{align*}
Thus $\chi\in G^\sigma(\mathbb R)$.  Moreover, since $\eta\in C_c^\infty(\mathbb R)$,
we have $\chi\in C^\infty(\mathbb R)$.

It remains to verify (\ref{G7hp}). By the change of variables
$z=x-y$,
\begin{align*}
    \chi(x)
    =
    \int_{x-a}^{x+a}\eta(z)\,dz.
\end{align*}
If $|x|\le1$, then
\begin{align*}
    \left[-\frac12,\frac12\right]\subset [x-a,x+a].
\end{align*}
Since $\operatorname{supp}\eta\subset[-1/2,1/2]$ and $\int_{\mathbb R}\eta=1$, it follows that
\begin{align*}
    \chi(x)=1,
    \qquad |x|\le1.
\end{align*}
If $|x|\ge2$, then the interval $[x-a,x+a]$ does not intersect
$(-1/2,1/2)$. Since $\eta(z)=0$ for $|z|\ge1/2$, we obtain
\begin{align*}
    \chi(x)=0,
    \qquad |x|\ge2.
\end{align*}
Therefore,
\begin{align*}
    \operatorname{supp}\chi\subset[-2,2].
\end{align*}
The proof is complete.
\end{proof}

\begin{lemma}
\label{onegevr}
Let $\alpha>0$, and define
\begin{align*}
\theta(t):=
\begin{cases}
\exp(-t^{-\alpha}),& t>0,\\
0,& t\leq0.
\end{cases}
\end{align*}
Then $\theta\in C^\infty(\mathbb R)$, all derivatives of $\theta$
vanish at the origin, and there exists a constant
$A_\alpha\geq1$ such that
\begin{equation}
\|\theta^{(n)}\|_{L^\infty(\mathbb R)}
\leq
A_\alpha^n (n!)^{1+1/\alpha},
\qquad n=0,1,2,\ldots.
\label{thetagev}
\end{equation}
Particularly, if
\begin{align*}
\alpha=\frac{1}{\sigma-1},
\qquad \sigma>1,
\end{align*}
then $\theta\in G^\sigma(\mathbb R)$.
\end{lemma}

\begin{proof}
Let
\begin{align*}
\mathbb H_+:=\{z\in\mathbb C:\operatorname{Re}z>0\}.
\end{align*}
On $\mathbb H_+$, choose the holomorphic branch of the logarithm
satisfying
\begin{align*}
-\frac{\pi}{2}<\arg z<\frac{\pi}{2},
\end{align*}
and define
\begin{align*}
\Theta(z):=\exp(-z^{-\alpha}),
\qquad
z^{-\alpha}:=\exp(-\alpha\log z).
\end{align*}
Thus $\Theta$ is holomorphic in $\mathbb H_+$, and
\begin{align*}
\Theta(t)=e^{-t^{-\alpha}}
\qquad (t>0).
\end{align*}

Set
\begin{align*}
\delta_\alpha
:=
\min\left\{
\frac14,\frac{\pi}{8\alpha}
\right\}.
\end{align*}
Fix $t>0$, and suppose that
\begin{align}\label{y7M}
|z-t|\leq\delta_\alpha t.
\end{align}
Since $\delta_\alpha\leq1/4$, we have
\begin{align*}
\operatorname{Re}z
\geq
(1-\delta_\alpha)t>0.
\end{align*}
Thus  the closed disk
$\overline{D(t,\delta_\alpha t)}$ is contained in
$\mathbb H_+$. Moreover, for $z$  satisfying (\ref{y7M}), there holds 
\begin{align}
|\operatorname{Im}z|
\leq\delta_\alpha t
\end{align}
and hence
\begin{align*}
|\arg z|
\leq
\arctan\frac{\delta_\alpha}{1-\delta_\alpha}
\leq
\frac{\delta_\alpha}{1-\delta_\alpha}
\leq
2\delta_\alpha.
\end{align*}
By the definition of $\delta_\alpha$,
\begin{align*}
|\alpha\arg z|
\leq
2\alpha\delta_\alpha
\leq
\frac{\pi}{4}.
\end{align*}
And for $z$ in  (\ref{y7M}) one has 
\begin{align*}
|z|
\leq
(1+\delta_\alpha)t
\leq
\frac54t.
\end{align*}
Writing $z=|z|e^{i\arg z}$, we therefore obtain
\begin{align*}
\operatorname{Re}(z^{-\alpha})
&=
|z|^{-\alpha}\cos(\alpha\arg z)\geq
\left(\frac54t\right)^{-\alpha}
\cos\frac{\pi}{4}\\
&=
c_\alpha t^{-\alpha},
\end{align*}
where
\begin{align*}
c_\alpha
:=
\frac{1}{\sqrt2}
\left(\frac54\right)^{-\alpha}>0.
\end{align*}
Thus one has 
\begin{equation}
|\Theta(z)|
=
\exp\bigl(-\operatorname{Re}(z^{-\alpha})\bigr)
\leq
\exp(-c_\alpha t^{-\alpha})
\label{thecomp}
\end{equation}
whenever $|z-t|\leq\delta_\alpha t$.

By Cauchy's derivative estimate  applied on the circle
$|z-t|=\delta_\alpha t$,  we have  for every $n\geq0$ that  
\begin{equation}
|\Theta^{(n)}(t)|
\leq
n!(\delta_\alpha t)^{-n}
\exp(-c_\alpha t^{-\alpha}).
\label{thetacau}
\end{equation}
For each fixed $n\geq0$, the right-hand side tends to zero as
$t\downarrow0$, because
\begin{align*}
\lim_{t\downarrow0}t^{-k}e^{-c_\alpha t^{-\alpha}}=0
\qquad
\text{for every }k\geq0.
\end{align*}
In particular,
\begin{align*}
\theta^{(n)}(t)=
\begin{cases}
\Theta^{(n)}(t),&t>0,\\
0,&t\leq0,
\end{cases}
\end{align*}
and $\theta\in C^{\infty}(\mathbb R)$. 

It remains to prove the uniform Gevrey estimate (\ref{thetagev}). The case $n=0$ of (\ref{thetagev}) is
immediate since 
$ 0\le \theta(t)\leq1$. 
Let $n\geq1$. From \eqref{thetacau}, one finds 
\begin{align*}
\|\theta^{(n)}\|_{L^\infty(\mathbb R)}
\leq
n!\delta_\alpha^{-n}
\sup_{t>0}
\left(
t^{-n}e^{-c_\alpha t^{-\alpha}}
\right).
\end{align*}
Make the change of variables
\begin{align*}
y=c_\alpha t^{-\alpha}.
\end{align*}
Then
\begin{align*}
\sup_{t>0}
t^{-n}e^{-c_\alpha t^{-\alpha}}
&=
c_\alpha^{-n/\alpha}
\sup_{y>0}y^{n/\alpha}e^{-y}\\
&=
c_\alpha^{-n/\alpha}
\left(\frac{n}{\alpha e}\right)^{n/\alpha}.
\end{align*}
Here we used the fact that the maximum of $y^p e^{-y}$ over
$y>0$ is attained at $y=p$.

The elementary Stirling bound
\begin{align*}
n!\geq\left(\frac ne\right)^n
\end{align*}
implies
\begin{align*}
\left(\frac{n}{\alpha e}\right)^{n/\alpha}
\leq
\alpha^{-n/\alpha}(n!)^{1/\alpha}.
\end{align*}
It follows that
\begin{align*}
\begin{aligned}
\|\theta^{(n)}\|_{L^\infty(\mathbb R)}
&\leq
\delta_\alpha^{-n}
c_\alpha^{-n/\alpha}
\alpha^{-n/\alpha}
(n!)^{1+1/\alpha}\\
&=
\left[
\delta_\alpha^{-1}
(\alpha c_\alpha)^{-1/\alpha}
\right]^n
(n!)^{1+1/\alpha}.
\end{aligned}
\end{align*}
Thus \eqref{thetagev} holds with
\begin{align*}
A_\alpha
:=
\max\left\{
1,\,
\delta_\alpha^{-1}
(\alpha c_\alpha)^{-1/\alpha}
\right\}.
\end{align*}

Finally, if $\alpha=(\sigma-1)^{-1}$, then
$ 
1+\frac1\alpha=\sigma$, 
and 
\begin{align*}
\|\theta^{(n)}\|_{L^\infty(\mathbb R)}
\leq
A_\alpha^n(n!)^\sigma,
\end{align*}
which proves that $\theta\in G^\sigma(\mathbb R)$.
\end{proof}

\end{document}